\documentclass[11pt,twoside,letterpaper]{report}
\usepackage[portuges]{babel}
\usepackage[latin1]{inputenc}
\usepackage{indentfirst}
\usepackage[dvips]{graphicx}
\usepackage{amstext}
\usepackage{latexsym}
\usepackage{makeidx}
\usepackage{textcomp}
\usepackage{amssymb}
\usepackage{amscd}
\usepackage{multicol}
\usepackage{amssymb}
\usepackage{amsfonts}
\usepackage{amsthm}
\usepackage[nottoc]{tocbibind} 			
\usepackage[letterpaper,top=2.54cm, bottom=2.54cm, left=2.54cm, right=2.54cm]{geometry}

\newcommand {\fim}{\end{enumerate}}
\newtheorem{teo}{Teorema}[chapter]
\newtheorem{lema}[teo]{Lema}
\newtheorem{prop}[teo]{Proposição}
\newtheorem{corolario}[teo]{Corolário}
\newtheorem{definicao}[teo]{Definição}
\newtheorem{obs}[teo]{Observação}

\newcommand{\lb}{\linebreak}
\newcommand{\vs}{\vspace{0.5cm}}

\begin{document}


\thispagestyle{empty}
\begin{center}
    \vspace*{2.3cm}
    \textbf{\Large{K-Teoria de operadores pseudodiferenciais  \\
    com símbolos semi-periódicos no cilindro }}\\
	
    \vspace*{1.2cm}
    \Large{Patricia Hess}
    
    \vskip 2cm
	\textsc{
	Tese apresentada\\[-0.25cm] 
	ao\\[-0.25cm]
	Instituto de Matemática e Estatística\\[-0.25cm]
	da\\[-0.25cm]
	Universidade de São Paulo\\[-0.25cm]
	para\\[-0.25cm]
	obtenção do título\\[-0.25cm]
	de\\[-0.25cm]
	Doutor em Ciências}
    
    \vskip 1.5cm
    Programa: Matemática Aplicada\\
    Orientador: Prof. Dr. Severino Toscano do Rego Melo

    \vskip 1cm
	\normalsize{Durante o desenvolvimento deste trabalho a autora recebeu auxílio
	financeiro do CNPq}
	
    \vskip 0.5cm
    \normalsize{São Paulo, dezembro de 2008}
\end{center}

\newpage


\thispagestyle{empty}
	\begin{center}
    	\vspace*{2.3 cm}
    	\textbf{\Large{K-Teoria de operadores pseudodiferenciais  \\
    com símbolos semi-periódicos no cilindro}}\\
	    \vspace*{2 cm}
	\end{center}

	\vskip 2cm

	\begin{flushright}
	Este exemplar corresponde à redação\\
	final da tese devidamente corrigida\\
	e defendida por Patricia Hess\\
	e aprovada pela Comissão Julgadora.
	\vskip 2cm

	\end{flushright}
	\vskip 3.8cm

	\begin{quote}
	\noindent Banca Examinadora:
	
	\begin{itemize}
		\item Prof. Dr. Severino Toscano do Rego Melo (orientador) - IME-USP
		\item Prof. Dra. Cristina Cerri - IME-USP
		\item Prof. Dra. Beatriz Abadie - URU
		\item Prof. Dr. Antônio Roberto da Silva - UFRJ
		\item Prof. Dr. Ricardo Bianconi - IME-USP
	\end{itemize}
	  
	\end{quote}
\pagebreak


\pagenumbering{roman} 	

\chapter*{Agradecimentos}

Agradeço a todos aqueles que sempre me incentivaram e apoiaram, em especial, ao professor Toscano, por ter aceito ser meu orientador e estar sempre disposto a me ajudar; aos amigos que fiz nestes últimos anos, por tornar a minha estadia em São Paulo tão agradável; e por fim, ao professor Daniel Gonçalves, que, embora não tenha seu nome citado na banca examinadora, esteve presente à defesa  e colaborou com muitas sugestões. 

\newpage


\chapter*{Resumo}

Seja $\mathcal{A}$ a C*-álgebra dos operadores limitados em $L^2(\mathbb{R}\times \mathbb{S}^1)$ gerada por: operadores $a(M)$ de multiplicação por funções $a \in C^{\infty}(\mathbb{S}^1)$, operadores  $b(M)$ de multiplicação por funções  $b \in C([-\infty, + \infty]),$ operadores de multiplicação por funções contínuas $2\pi$-periódicas, $\Lambda = (1-\Delta_{\mathbb{R}\times \mathbb{S}^1})^{-1/2},$ onde $\Delta_{\mathbb{R}\times \mathbb{S}^1}$ é o Laplaciano de $\mathbb{R}\times \mathbb{S}^1,$ e $\partial_t \Lambda, ~ \partial_x \Lambda$ para $t\in \mathbb{R}$ e $x\in \mathbb{S}^1$. Calculamos a K-teoria de $\mathcal{A}$ e de $\mathcal{A}/\mathcal{K}(L^2(\mathbb{R}\times \mathbb{S}^1))$,  onde $\mathcal{K}(L^2(\mathbb{R}\times \mathbb{S}^1))$ é o ideal dos operadores compactos em $L^2(\mathbb{R}\times \mathbb{S}^1)$.

\noindent \textbf{Palavras-chave:} K-teoria, Operadores pseudodiferenciais.

\chapter*{Abstract}
Let $\mathcal{A}$ denote the C*-algebra of bounded operators on $L^2(\mathbb{R}\times \mathbb{S}^1)$ generated by: all multiplications $a(M)$ by functions $a \in C^{\infty}(\mathbb{S}^1)$, all multiplications $b(M)$ by functions  $b \in C([-\infty, + \infty]),$ all multiplications by $2\pi$-periodic continuous functions, $\Lambda = (1-\Delta_{\mathbb{R}\times \mathbb{S}^1})^{-1/2},$ where $\Delta_{\mathbb{R}\times \mathbb{S}^1}$ is the Laplacian on $\mathbb{R}\times \mathbb{S}^1,$ and $\partial_t \Lambda, ~ \partial_x \Lambda$, for $t\in \mathbb{R}$ and $x\in \mathbb{S}^1$. We compute the K-theory of $\mathcal{A}$ and $\mathcal{A}/\mathcal{K}(L^2(\mathbb{R}\times \mathbb{S}^1))$,  where $\mathcal{K}(L^2(\mathbb{R}\times \mathbb{S}^1))$ is the ideal of compact operators  on $L^2(\mathbb{R}\times \mathbb{S}^1)$. 

\noindent \textbf{Keywords:} K-theory, Pseudodifferential operators.


\tableofcontents


\chapter*{Introdu\c{c}\~ao} \label{intro}
\pagenumbering{arabic}
\addcontentsline{toc}{chapter}{Introdu\c{c}\~ao}

Este trabalho tem como objetivo calcular a K-teoria da C*-álgebra $\mathcal{A}$ de operadores limitados em $(L^2(\mathbb{R}\times \mathbb{S}^1)),$ gerada pelos operadores: (i) de multiplicação por funções que se estendem continuamente para $[-\infty,+\infty]\times \mathbb{S}^1$, (ii)  de multiplicação por funções contínuas em $\mathbb{R}$ de período $2\pi$, (iii) $\Lambda = (1-\Delta_{\mathbb{R}\times \mathbb{S}^1})^{-1/2},$ onde $\Delta_{\mathbb{R}\times \mathbb{S}^1}$ é o operador Laplaciano em $L^2(\mathbb{R}\times \mathbb{S}^1),$ e (iv) $1/i\partial_t \Lambda, ~ 1/i\partial_x \Lambda$ para $t\in \mathbb{R}$ e $x\in \mathbb{S}^1$.


Para calcular a K-teoria de $\mathcal{A}$, 
é importante conhecer a estrutura desta álgebra. Desta forma, o capítulo 1 é destinado essencialmente a descrever sua estrutura, apresentando diversos resultados de \cite{melo} restritos ao caso $\mathbb{B} = \mathbb{S}^1.$ 

A C*-álgebra $\mathcal{A}$ não é comutativa, e seus comutadores não são todos compactos. Temos então que o ideal comutador $\mathcal{E}_{\mathcal{A}}$ de $\mathcal{A}$ contém o ideal dos operadores compactos $\mathcal{K}(L^2(\mathbb{R}\times \mathbb{S}^1))$ e além disso, temos o seguinte isomorfismo: 
\begin{equation} \frac{\mathcal{E}_{\mathcal{A}}}{{\cal{K}}} \cong C(S^{1}\times\{-1,+1\},{\cal{K}}(L^2({\mathbb{Z}\times \mathbb{S}^{1}}))).
\label{e}
\end{equation}
Compondo $ \mathcal{E}_{\mathcal{A}}\rightarrow \mathcal{E}_{\mathcal{A}}/{\cal{K}}$ com o isomorfismo acima, conseguimos estender esta aplicação e obter um *-homomorfismo
$$\gamma: \mathcal{A}\rightarrow C(S^{1}\times\{-1,+1\},{\cal{L}}(L^2({\mathbb{Z}\times \mathbb{S}^{1}}))),$$ 
que nos será muito útil nos cálculos  das aplicações de conexão em K-teoria. 

Apresentamos ainda, no capítulo 1, o isomorfismo entre a álgebra comutativa $\mathcal{A}/\mathcal{E}_{\mathcal{A}}$  e o conjunto das funções contínuas em $M$ tomando valores complexos, denotado por $C(M)$, sendo $M$ um subconjunto compacto de $[-\infty, +\infty]\times \mathbb{S}^1\times S^1\times S^1$, onde este isomorfismo se dá pela aplicação de Gelfand. Note que $S^1$ também representa o círculo, mas preferimos distinguí-lo do círculo proveniente da variedade. A composição da projeção canônica $ \mathcal{A} \rightarrow {\mathcal{A}}/\mathcal{E}_{\mathcal{A}}$ com este isomorfismo nos dá a aplicação que denominamos de símbolo, isto porque esta aplicação é dada pelo símbolo principal dos operadores em ${\mathcal{A}}$. 

Apesar de termos todas estas informações sobre $\mathcal{A}$, elas não foram suficientes para o cálculo de sua K-teoria. 
Definimos então, no capítulo 2, duas C*-subálgebras de $\mathcal{A}$, as quais denominamos $\mathcal{A}^{\dagger}$ e $\mathcal{A}^{\diamond},$ e tais que 
$$ \mathcal{A}^{\dagger} \subset \mathcal{A}^{\diamond} \subset \mathcal{A}.$$
Estas álgebras não são comutativas e não contêm a álgebra dos operadores compactos em $L^2({\mathbb{R}\times \mathbb{S}^{1}})$, denotada por  $\mathcal{K}(L^2({\mathbb{R}\times \mathbb{S}^{1}})).$ Determinamos seus ideais comutadores e mostramos os seguintes isomorfismos:
$$ \frac{\mathcal{A}^{\dagger}}{\mathcal{E}^{\dagger}} \cong C(S^1\times \mathbb{S}^1)~, ~~~~ 
\frac{\mathcal{A}^{\diamond}}{\mathcal{E}^{\diamond}} \cong C(S^1\times S^1\times \mathbb{S}^1).$$
O resultado mais importante deste capítulo é que a C*-álgebra $\mathcal{A}^{\diamond}$ tem estrutura de produto cruzado e este teorema segue como uma generalização do teorema 8 de \cite{cintia}.

No terceiro e último capítulo, apresentamos os cálculos de K-teoria das álgebras $\mathcal{A}^{\dagger}, ~\mathcal{A}^{\diamond}$ e $\mathcal{A}.$
Partindo da seqüência induzida pela aplicação do símbolo quocientada pelo ideal $\mathcal{K}(L^2(\mathbb{R}\times \mathbb{S}^1))$
$$
0 \longrightarrow \frac{\mathcal{E}_{\mathcal{A}}}{\mathcal{K}(L^2(\mathbb{R}\times \mathbb{S}^1))} \stackrel{i}{\longrightarrow} 
\frac{{\mathcal{A}}}{\mathcal{K}(L^2(\mathbb{R}\times \mathbb{S}^1))} \stackrel{\pi}{\longrightarrow}
\frac{{\mathcal{A}}}{\mathcal{E}_{\mathcal{A}}} \longrightarrow 0,
$$
conseguimos calcular a aplicação do índice da seqüência exata de seis termos em K-teoria:
$$\delta_1: K_1(\mathcal{A}/ {\mathcal{E}_{\mathcal{A}}})\longrightarrow  K_0(\mathcal{E}_{\mathcal{A}}/\mathcal{K}(L^2(\mathbb{R}\times \mathbb{S}^1))).$$
Analogamente a proposição 3 de \cite{cintia}, conseguimos mostrar que $\gamma\circ\delta_1$ calculado num elemento $[[A]_{\mathcal{E}_{\mathcal{A}}}]_1 \in K_1(\mathcal{A}/ {\mathcal{E}_{\mathcal{A}}})$ é igual a $\texttt{ind}\gamma_A(1,\pm 1)[E]_0,$ cujo operador $E$ é uma projeção de posto 1 de \linebreak $\mathcal{K}(L^2(\mathbb{Z}\times \mathbb{S}^1))$ e $\texttt{ind}$ é o índice de Fredholm. Para calcular este índice, usamos uma particularização da fórmula do índice para operadores elípticos de Atiyah-Singer dada por Fedosov \cite{fedosov}, no caso em que a variedade tem \textit{Todd class} igual a 1. Esta fórmula determina o índice do operador a partir de seu símbolo, fazendo com que não precisemos entrar em argumentos topológicos, necessários na fórmula de Atiyah-Singer. 

Para calcular $K_0$ e $K_1$ de $\mathcal{A}^{\dagger}$ e $\mathcal{A}^{\diamond}$, usamos a mesma idéia acima, mas foi com a seqüência de Pimsner-Voiculescu \cite{pv} que conseguimos concluir que 
$$  K_0(\mathcal{A}^{\dagger}) \cong \mathbb{Z} ~, ~~  K_1(\mathcal{A}^{\dagger}) \cong \mathbb{Z}^2 ~, ~~ 
K_0(\mathcal{A}^{\diamond}) \cong K_1(\mathcal{A}^{\diamond}) \cong \mathbb{Z}^3 .$$ 

Com todos estes resultados e a partir do cálculo da K-teoria do ideal comutador $\mathcal{E}_{\mathcal{A}}$, provamos que existe um operador de Fredholm $T \in M_n(\mathcal{E}_{\mathcal{A}}^{+})$  cujo índice é igual a 1. Diante disso, conseguimos mostrar que a aplicação do índice da seqüência exata de seis termos em K-teoria associada a seqüência 
$$ 0\longrightarrow \mathcal{K}(L^2(\mathbb{R}\times\mathbb{S}^1)) \longrightarrow \mathcal{A} \longrightarrow \mathcal{A}/\mathcal{K}(L^2(\mathbb{R}\times\mathbb{S}^1)) \longrightarrow 0, $$
é sobrejetora. E esta informação foi crucial para chegarmos ao resultado desejado. 

No trabalho de Melo e Silva \cite{cintia}, eles determinam a K-teoria da C*-álgebra $A \subset \mathcal{L}(L^2(\mathbb{R}))$ gerada por: (i)operadores de multiplicação por funções que admitem extensões contínuas em $[-\infty, +\infty]$, (ii) os operadores descritos em (i) conjugados com a transformada de Fourier $F$, e (iii) multiplicadores por funções contínuas de período $2\pi$. Eles apresentam duas maneiras diferentes para calcular a K-teoria de $A .$ Uma delas se baseia na seqüência exata proveniente da aplicação do símbolo principal, onde se conhece o núcleo dessa aplicação e sua imagem. A outra, surge da seqüência exata induzida pela  aplicação que chamamos de $\gamma'$, que estende o homomorfismo sobrejetor $E \rightarrow C(S^1\times\{\pm1\}, \mathcal{K}(\ell^2(\mathbb{Z}))).$

Apesar de conhecermos o núcleo e a imagem de  $\gamma: \mathcal{A}\rightarrow C(S^{1}\times\{\pm1\},{\cal{L}}(L^2({\mathbb{Z}\times \mathbb{S}^{1}}))),$ a seqüência induzida por esta aplicação não nos possibilitou ir adiante no cálculo da K-teoria de $\mathcal{A}$ e trabalhamos apenas com a seqüência do símbolo. Mas o que é importante ressaltar, é que tivemos que construir várias outras seqüências envolvendo a álgebra, seu comutador e $\mathcal{K}(L^2(\mathbb{R}\times S^1))$ para podermos inferir algo. 

De um modo geral, o que usamos para calcular a K-teoria de $\mathcal{A}$ são argumentos bem conhecidos, que estão associados ao símbolo principal e ao símbolo principal de fronteira, ao qual chamamos de $\gamma$ no nosso trabalho. Conhecer a imagem e o núcleo desses símbolos é uma ferramenta que nos possibilita o cálculo da K-teoria de uma C*-álgebra $A$, como pode ser visto nos trabalhos de Melo, Nest e Schrohe \cite{nest} e Melo, Schick e Schrohe \cite{sch}. A estrutura da álgebra $\mathcal{A}$ se assemelha a estrutura das álgebras estudadas em \cite{nest} e \cite{sch} e de tantas outras C*-álgebras geradas por operadores pseudodiferenciais de ordem zero (ver por exemplo \cite{cordes2}, \cite{cordes1}, \cite{lauter}, \cite{lauter2}, \cite{melo}  e \cite{mor}).

Apesar de termos estudado a K-teoria de ${\mathcal{A}}$ apenas em $\mathbb{R}\times\mathbb{S}^1$, suspeitamos de que vários resultados podem ser generalizados para  $\mathbb{R}\times\mathbb{B}$, onde $\mathbb{B}$ é uma variedade Riemanniana compacta de dimensão n. Mas este é uma assunto que pretendemos abordar num trabalho futuro.

\chapter{A C*-álgebra $\mathcal{A}$}

Começaremos este primeiro capítulo definindo a C*-álgebra $\cal{A},$ com a qual iremos trabalhar, e descrevendo vários resultados já conhecidos sobre ela. Melo, \cite{melo}, e Cordes, \cite{cordes1}, estudaram esta álgebra e deles obtemos informações necessárias para o cálculo da K-teoria de $\mathcal{A}$. 

Nas duas primeiras seções, introduziremos $\cal{A}$ como uma C*-subálgebra de operadores limitados em $L^2(\mathbb{R}\times{\mathbb{B}}),$ cujas funções em $L^2(\mathbb{R}\times{\mathbb{B}})$ são complexas e $\mathbb{B}$ é uma variedade Riemanniana compacta de dimensão $n$. Na terceira seção, nos restringiremos ao caso em que $\mathbb{B}$ é o círculo $\mathbb{S}^1 $ e a partir de então este será o caso que iremos estudar nos próximos capítulos.  

No decorrer do trabalho, usamos duas notações para o círculo. $S^1$ denota um círculo qualquer e $\mathbb{S}^1$ corresponde especificamente à variedade escolhida no lugar de $\mathbb{B}$.

\section{Descrição de $\cal{A}$}

Considere $\mathbb{B}$ uma variedade Riemanniana compacta. Escreva $\Omega = \mathbb{R}\times \mathbb{B}$ e seja $\Delta_{\Omega}$ o operador Laplaciano definido em $L^2(\Omega)$. 
Definimos a álgebra $\cal{A}$ como sendo a menor C*-subálgebra de ${\cal{L}}(L^2(\Omega))$ contendo 

\begin{itemize}

\item[(i)] operadores $a(M_x)$ de multiplicação por $ a\in C^\infty({\mathbb{B}}):$
$$ [a(M_x)f](t,x) = a(x)f(t,x) ;$$
 
\item[(ii)] operadores $b(M_t)$ de multiplicação por $ b\in C([-\infty,+\infty]),$ onde $C([-\infty,+\infty])$ representa o conjunto das funções contínuas em $\mathbb{R}$ que possuem limite quando $t\rightarrow\pm\infty:$ 
$$ [b(M_t)f](t,x) = b(t)f(t,x) , ~~ (t,x)\in\mathbb{R}\times\mathbb{S}^1 ; $$

\item[(iii)] operadores $p_j(M_t)$ de multiplicação na primeira variável por funções contínuas $2\pi$-periódicas $p_j(t)=e^{ijt},~j\in\mathbb{Z}$; 

\item[(iv)] $\Lambda:=(1-\Delta_{\Omega})^{-1/2};$ 

\item[(v)] $\frac{1}{i}\frac{\partial}{\partial t}\Lambda,$ com $t\in \mathbb{R};$

\item[(vi)] $L_x\Lambda,$ onde $L_x$ é um operador diferencial de primeira ordem em $\mathbb{B}$. 

\end{itemize}

Com o intuito de facilitar a notação, denotemos por $A_1, A_2, \dots, A_6$ os operadores (ou classes de operadores) listados nos itens (i), (ii), \dots, (vi) acima, respectivamente. 


Usando a estrutura de espaço de Hilbert de $L^{2}(\Omega)$, temos que este é isomorfo a $L^{2}(\mathbb{R})\otimes L^{2}(\mathbb{B})$\footnote{$L^{2}(\mathbb{R})\otimes L^{2}(\mathbb{B})$ significa o espaço de Hilbert proveniente do produto tensorial $\otimes$ entre $L^{2}(\mathbb{R})$ e $L^{2}(\mathbb{B})$}. Podemos então conjugar a álgebra $\cal{A}$ com o operador unitário $F\otimes I_{\mathbb{B}}$ definido em $L^{2}(\mathbb{R})\otimes L^{2}(\mathbb{B})$,
$$(F \otimes I_{\mathbb{B}})^{-1} \mathcal{A} (F\otimes I_{\mathbb{B}}),$$
onde $F$ é a transformada de Fourier em $L^{2}(\mathbb{R}),$
$$ (Fu)(\tau) = \frac{1}{\sqrt{2\pi}}\int_{\mathbb{R}} e^{-it\tau}u(t)dt ,$$
e $I_{\mathbb{B}}$ denota o operador identidade em $L^2(\mathbb{B}).$ 
Esta conjugação será útil daqui em diante para obtermos uma série de resultados sobre $\cal{A}.$

Vamos denotar um operador do tipo $A\otimes I_{\mathbb{B}},$ onde $A \in \mathcal{L}(L^{2}(\mathbb{R}))$ e $I_{\mathbb{B}}$ é o operador identidade em $L^2(\mathbb{B}),$  e $I_{\mathbb{R}} \otimes B,$ onde $I_{\mathbb{R}}$ é o operador identidade em $L^2(\mathbb{R})$  e $B \in \mathcal{L}(L^{2}(\mathbb{B})),$ apenas por $A$ e $B$, respectivamente, deixando sempre claro a que espaço cada um deles pertence. 

Seja $\cal{B}$ a C*-álgebra $\cal{A}$ conjugada por $F$ e para cada $i=1,\dots, 6$, $F^{-1} A_i F$ está listado a seguir.   
\begin{itemize}

\item[(i')] $a(M_x)$

\item[(ii')] $b(D):= F^{-1}b(M_t)F;$

\item[(iii')] $ F^{-1}p_j(M_t)F = T_j,$ onde $(T_j u)(\tau,\cdot)=u(\tau + j, \cdot),$ para $u\in L^2(\Omega)$ e $j\in \mathbb{Z}$; 

\item[(iv')] $ (F^{-1}{\Lambda}Fu)(\tau,\cdot)= (1+\tau^2-\Delta_{\mathbb{B}})^{-1/2}u(\tau,\cdot),~u\in L^2(\Omega)$;

\item[(v')] $ (F^{-1} \frac{1}{i}\frac{\partial}{\partial t}\Lambda F u)(\tau,\cdot)= -\tau(1+{\tau}^2-\Delta_{\mathbb{B}})^{-1/2} u(\tau, \cdot),~u\in L^2(\Omega)$;

\item[(vi')] $ (F^{-1} L_x\Lambda F u)(\tau, \cdot) = L_x(1+{\tau}^2-\Delta_{\mathbb{B}})^{-1/2} u(\tau, \cdot),~u\in L^2(\Omega).$

\end{itemize}

Podemos considerar a álgebra das funções contínuas de $\mathbb{R}$ limitadas que tomam valores em $\mathcal{L}(L^2(\mathbb{B}))$, denotada por $C_b(\mathbb{R}, \mathcal{L}(L^2(\mathbb{B})))$, como uma subálgebra de $\mathcal{L}(L^2(\mathbb{R}, L^2(\mathbb{B})))$ no seguinte sentido: uma função $f$ em 
$C_b(\mathbb{R}, \mathcal{L}(L^2(\mathbb{B})))$ é identificada com operador de multiplicação $f(M_t)$ em $\mathcal{L}(L^2(\mathbb{R}, L^2(\mathbb{B})))$ definido por 
$$(f(M_t)u)(t) = f(t) u(t)~ \in ~ L^2(\mathbb{B}),$$
para $u \in L^2(\mathbb{R}, L^2(\mathbb{B}))$. 
Note que $L^2(\mathbb{R}\times \mathbb{B})\cong L^2(\mathbb{R}, L^2(\mathbb{B}))$.

As funções 
\begin{equation}\label{bis}
B_4(\tau) = \tilde{\Lambda}(\tau) ~, ~~ B_5(\tau) = -\tau \tilde{\Lambda}(\tau) ~, ~~ B_6(\tau) = L_x \tilde{\Lambda}(\tau), ~\tau \in \mathbb{R},
\end{equation}
com $\tilde{\Lambda}(\tau) = (1 + {\tau}^2 - {\Delta}_{\mathbb{B}})^{-1/2}$, estão em $C_b(\mathbb{R}, \mathcal{L}(L^2(\mathbb{B})))$ (isto será demonstrado mais adiante no lema \ref{funcaoltda}). Assim, obtemos os operadores 
$$ B_4(M_t) ~, ~B_5(M_t)~, ~ B_6(M_t) ~\in ~ \mathcal{L}(L^2(\mathbb{R}, L^2(\mathbb{B})))$$
\label{opb}
{opb}
de multiplicação pelas funções $B_4, B_5, B_6.$ E estes operadores são exatamente os encontrados nos itens (iv'), (v') e (vi'), respectivamente.

Para simplificar a notação, vamos escrever o conjunto dos operadores limitados em $L^2(Y)$ por $\mathcal{L}_Y$, e o conjunto dos operadores compactos em  $L^2(Y)$ por $\mathcal{K}_Y$, para uma variedade $Y$ qualquer. 



\section{O ideal comutador de $\cal{A}$}

\begin{definicao} Chamamos de ideal comutador de uma C*-álgebra $A$ ao ideal fechado gerado pelos comutadores $[a,b] = ab - ba,$ para todo $a$ e $b$ em $A.$
\end{definicao}

Sejam $\cal{E}_{\cal{A}}$ o ideal comutador de $\mathcal{A}$ e $\cal{E}_{\cal{B}}$ o ideal comutador de $\mathcal{B}$. É claro que 
$$
{\cal{E}_{\cal{B}}} = F^{-1} {\cal{E}_{\cal{A}}} F. 
$$
A proposição 1.2 de \cite{melo} nos dá uma descrição do ideal comutador de $\cal{B}:$

\begin{prop} O ideal comutador $\cal{E}_{\cal{B}}$ de $\mathcal{B}$ é o fecho do seguinte conjunto de operadores:
$$
{\cal{E}}_{{\cal{B}},0} = \{\sum_{j=-N}^{N} b_j(D) K_j(M_t) T_j + K ;~ b_j \in C([-\infty, +\infty]),~K_j \in C_0(\mathbb{R}, {\cal{K}}_{\mathbb{B}}), K \in {\cal{K}}_{\Omega}, ~ N \in \mathbb{N} \},
$$
onde $K_j(M_t)$ representa o operador de multiplicação pela função $K_j \in C_0(\mathbb{R}, {\cal{K}}_{\mathbb{B}}):$
$$(K_j(M_t)u)(t) = K_j(t)u(t),~ u\in L^2(\mathbb{R}, L^2(\mathbb{S}^1)).$$
\end{prop}

Considerando $J$ o fecho do conjunto
$$
J_0 =\{\sum_{j=-N}^{N} b_j(D) a_j(M_t) T_j + K ; ~b_j \in C([-\infty, +\infty]),~ a_j \in C_0(\mathbb{R}),~ K \in {\cal{K}}_{\mathbb{R}},~N \in \mathbb{N} \},
$$
podemos também escrever o ideal $\cal{E}_{\cal{B}}$ como sendo 
$$
J\otimes{\cal{K}}_{\mathbb{B}},
$$
onde este produto tensorial entre C*-álgebras é único, pois ${\cal{K}}_{\mathbb{B}}$ é nuclear.

Temos em \cite{melo}, teorema 1.6, o isomorfismo entre ${\cal{E}}_{\cal{A}}/{\cal{K}}_{\Omega}$ e duas cópias de $C(S^{1},{\cal{K}}_{\mathbb{Z}\times \mathbb{B}})$. Antes de enunciá-lo, daremos uma breve idéia da construção necessária para chegar a este resultado (a demostração completa pode ser encontrada com todos os detalhes em \cite{melo}).

Para cada $\varphi$ real, seja $U_{\varphi}$ o operador em $L^2(S^1)$ dado por $U_{\varphi}f (z) = z^{-\varphi}f(z),~z\in S^1,$ onde $z=e^{i\theta}$ com $\theta \in [-\pi, \pi]$, e seja $Y_{\varphi}$ o operador em $\ell^2(\mathbb{Z})$ da seguinte forma:
$$
Y_\varphi := F_d U_{\varphi} F_d^{-1}, 
$$
onde $F_d: L^2(S^1)\rightarrow {\ell}^2(\mathbb{Z}) $ é a transformada de Fourier discreta
$$ 
(F_d u)_j = \frac{1}{\sqrt{2\pi}}\int_{0}^{2\pi} u(e^{i\theta}) e^{-ij\theta} d\theta ~,~u\in L^2(S^1) \mbox{ e } j\in \mathbb{Z}. 
$$ 
Então $Y_{\varphi}$ é um operador unitário que satisfaz $(Y_k u)_j = u_{j+k},$ se $k\in \mathbb{Z},$ e $Y_{\varphi} Y_{\omega} = Y_{\varphi +\omega}.$

Seja $u$ uma função em $L^2(\mathbb{R}).$ Defina para cada $\varphi$ em $\mathbb{R},$ a seqüência em ${\ell}^2(\mathbb{Z})$ 
$$
u^\diamond(\varphi) := (u(\varphi - j))_{j\in \mathbb{Z}}
$$
($u^\diamond(\varphi)\in {\ell}^2(\mathbb{Z})$ para quase todo $\varphi \in \mathbb{R}$ pelo Teorema de Fubini, pois $L^2(\mathbb{R})$ pode ser visto como $L^2([0,1]\times\mathbb{Z})$). Podemos definir então o operador unitário 
\begin{equation}
 W \ : \ L^2(\mathbb{R}) \longrightarrow L^2(S^1, d\varphi; {\ell}^2(\mathbb{Z})) ~, ~~~
 u \longmapsto Wu(e^{2\pi i\varphi}) = Y_\varphi u^\diamond(\varphi) \ .
\label{w}
\end{equation}

Denote $\cal{S}$ a C*-subálgebra de ${\cal{L}}_{{S}^1}$ gerada por 
\begin{itemize}

\item[(i)] operadores de multiplicação por funções em $C^\infty({S}^1)$; 

\item[(ii)] operadores do tipo $b(D_\theta)={F_d}^{-1}b(M_j)F_d,$ onde $b(M_j)$ é um operador definido em ${\ell}^2(\mathbb{Z}),$ de multiplicação pela seqüência $b=(b(j))_j$, com limite quando $j\rightarrow\pm\infty.$
\label{s}
\end{itemize}

$\cal{S}$ contém o ideal $\mathcal{K}_{{S}^1}$ e todos os seus comutadores são compactos (\cite{cordes3}, capítulos V e VI e \cite{rodrigo}). 
Temos ainda de \cite{melo}, que vale o  isomorfismo $\mathcal{S}/ \mathcal{K}_{{S}^1}\cong C(S^1\times \{-1,+1\})$. Compondo-o com a projeção canônica $\mathcal{S}\rightarrow \mathcal{S} / \mathcal{K}_{{S}^1},$ obtemos a seguinte aplicação: 
$$ 
\rho : {\mathcal{S}} \longrightarrow C({S}^1\times \{-1,+1\}) 
$$
$$
{\rho}_{a(M_x)}(x,\pm 1)= a(x) , ~ a \in C^{\infty}({S}^1) ~ ~ ~; ~ ~ ~
{\rho}_{b(D_{\theta})}(x,\pm 1)= b(\pm \infty) .
$$

Temos então a seguinte proposição 1.5 de \cite{melo}:

\begin{prop}\label{e} A aplicação
$$
\begin{array}{ccc}
 \mathcal{E}_{\mathcal{B}} & \! \longrightarrow & \!  \mathcal{S}\otimes{\mathcal{K}}_{\mathbb{Z}}\otimes{\mathcal{K}}_{\mathbb{B}} \\     
  B & \! \longmapsto & \! W B W^{-1}
\end{array}
$$
é um $\ast$-isomorfismo. Para $B$ em $\cal{E}_{\cal{B}}$ da forma $B=b(D) K(M_t)T_j$, com $b$ em $ C([-\infty, +\infty]),\ K$ em $C_{0}(\mathbb{R}, {\cal{K}}_{\mathbb{B}})$ e $j \in \mathbb{Z},$ temos 
\begin{equation}\label{w}
WBW^{-1} = b(D_\theta)Y_{\varphi} K(\varphi - M_j)Y_{-\varphi -j} + K', \ com \ K' \in {\cal{K}}_{{S}^1\times\mathbb{Z}\times\mathbb{B}}.
\end{equation}
Para cada $\varphi \in \mathbb{R}$, $K(\varphi - M_j)$ denota o operador compacto em $\ell^2(\mathbb{Z}, L^2(\mathbb{B}))$ de multiplicação pela seqüência $K(\varphi - j)$ em $\mathcal{K}_{\mathbb{B}}$. 
\end{prop} 

Denotando por  $C(S^1\times \{-1,+1\},{\mathcal{K}}_{\mathbb{Z}\times\mathbb{B}})$ o conjunto das funções definidas em $S^1\times \{-1,+1\}$ que tomam valores em ${\mathcal{K}}_{\mathbb{Z}\times\mathbb{B}},$ enunciamos o teorema 1.6 de \cite{melo}.  

\begin{teo} Existe um $\ast$-isomorfismo   
\begin{equation}\label{psi}
\Psi : \frac{\cal{E}_{\cal{A}}}{{\cal{K}}_{\Omega}} \longrightarrow C(S^{1}\times\{-1,+1\},{\cal{K}}_{\mathbb{Z}\times \mathbb{B}}).
\end{equation}
\end{teo}

\begin{proof} Tome $\Psi$ como sendo a composta da seqüência de aplicações abaixo:
$$
\frac{\cal{E}_{\cal{A}}}{{\cal{K}}_{\Omega}} 
\stackrel{(1)}{\longrightarrow} 
\frac{\cal{E}_{\cal{B}}}{{\cal{K}}_{\Omega}} 
\stackrel{(2)}{\longrightarrow}
\frac{{\cal{S}} \otimes {\cal{K}}_{\mathbb{Z}} \otimes {\cal{K}}_{\mathbb{B}}}{{\cal{K}}_{S^{1} \times \mathbb{Z} \times \mathbb{B}}}
\stackrel{(3)}{\longrightarrow}
C(S^{1}\times\{-1,+1\},{\cal{K}}_{\mathbb{Z}\times \mathbb{B}})
$$
onde (1) representa a aplicação $A + {\cal{K}}_{\Omega} \mapsto F^{-1}AF + {\cal{K}}_{\Omega},$ (2) leva este último elemento em \lb $ WF^{-1}AFW^{-1} + {\cal{K}}_{S^{1} \times \mathbb{Z} \times \mathbb{B}}$ e (3) é o $\ast$-isomorfismo:
$$
\begin{array}{ccc}
 \displaystyle\frac{{\cal{S}} \otimes {{\cal{K}}_{\mathbb{Z}}} \otimes {{\cal{K}}_{\mathbb{B}}}}
 {{\cal{K}}_{S^{1} \times \mathbb{Z} \times \mathbb{B}}} & \! 
 \longrightarrow & \!  
 C(S^1\times \{-1,+1\},{\cal{K}}_{\mathbb{Z} \times \mathbb{B}}) \\     
 A \otimes K_1 \otimes K_2 + {\cal{K}}_{S^1 \times \mathbb{Z} \times \mathbb{B}} & \! 
 \longmapsto & \! 
 {\rho}_{A}(\cdot,\pm 1) K_1\otimes K_2 ~.
\end{array}
$$
\end{proof}

\section{O caso $\mathbb{B} = \mathbb{S}^1$}

Desta seção em diante, estudaremos o caso em que a variedade $\mathbb{B}$ é o círculo, apesar dos resultados apresentados nessa seção serem válidos para uma variedade $\mathbb{B}$ compacta de dimensão $n$. 

A classe de geradores denotados por $A_6$ na definição da C*-álgebra $\mathcal{A}$ pode ser substituída pelo operador 
$$
\mbox{(vi)}~ \frac{1}{i}\frac{\partial}{\partial \beta}\Lambda,~ \mbox{com} ~e^{i\beta}\in \mathbb{S}^1.
$$
já que a variedade $\mathbb{S}^1$ tem fibrado tangente trivial. 

Definiremos nesta seção duas aplicações de $\mathcal{A}$ num certo espaço de funções contínuas. Estas aplicações são $\ast$-homomorfismos muito importantes e serão  muito citadas no último capítulo.  

\begin{definicao} Seja $A$ uma C*-álgebra e $I$ seu ideal comutador. Chamamos de \textbf{espaço símbolo} $M_{\mathcal{A}}$ de $\mathcal{A}$ ao espaço de ideais maximais da álgebra comutativa $\cal{A}/\cal{I}$. 
\end{definicao}

\begin{definicao} Chamamos de \textbf{$\sigma$-símbolo} de uma C*-álgebra $\mathcal{A}$ a aplicação
$$ \sigma: \mathcal{A} \rightarrow C(M_{\mathcal{A}}),$$
que é dada pela composição da projeção canônica $\pi: \mathcal{A} \rightarrow \mathcal{A}/\mathcal{I}$ com a transformada de Gelfand da C*-álgebra $\mathcal{A}/\mathcal{I},$ onde $I$ é o ideal comutador de  $\mathcal{A}.$
\end{definicao}

Dado um elemento $d$ em $\cal{D}$, chamaremos de símbolo de $d$ a função contínua ${\sigma}_{d}\in C(M_{\cal{D}})$ associada ao elemento $d$ módulo $\cal{I}.$ Enunciaremos agora o teorema 2.2 de \cite{melo}, restrito ao caso $\mathbb{B}=\mathbb{S}^1$, que descreve o espaço símbolo de $\cal{A}$ e nos dá os símbolos de seus geradores.

\begin{teo}\label{ma2} O espaço símbolo de $\cal{A}$ é o subconjunto de $[-\infty,+\infty] \times {\mathbb{S}}^1 \times S^1 \times S^1$ dado por 
$$
\begin{array}{cc}
{\bf M_{\cal{A}}} \ = & \! \{(t,x,(\tau ,\xi),e^{i\theta});~t\in[-\infty,+\infty],~x\in {\mathbb{S}}^1 ,~
(\tau,\xi)\in\mathbb{R}^2 : \tau^{2}+\xi^{2}=1, \\
& \! \ \ \ \ \ \ \theta \in \mathbb{R}~ e ~  \theta = t ~ se ~
  |t|<\infty\},
\label{ma}
\end{array}
$$
onde o elemento $(t, x, (\tau ,\xi))$ pertence ao fibrado das coesferas de $\mathbb{R} \times {\mathbb{S}}^1.$
Com esta descrição de ${\bf M_{\cal{A}}},$ os $\sigma$-símbolos da classe dos geradores da C*-álgebra $\cal{A}$ calculados em $(t,x,(\tau ,\xi),e^{i\theta})$ são respectivamente dados por 
\begin{equation}\label{sigma}
 a(x),~ ~ ~ b(t),~ ~ ~ e^{ij\theta},~ ~ ~ 0, ~ ~ ~ \tau,~ ~ ~ \xi.
\end{equation}
\end{teo}

Vejamos agora que todo operador A em $\cal{A}$ determina um operador limitado de ${\cal{L}}
({\cal{E}_{\cal{A}}}/{\cal{K}}_{\Omega}),$ já que $
{\cal{E}_{\cal{A}}}/{\cal{K}}_{\Omega}$ é um ideal de $
{\cal{A}}/{\cal{K}}_{\Omega}.$ Fica bem definido então o operador
abaixo:   
$$
\begin{array}{cccc}
 T \ : & \! {\cal{A}} & \! \longrightarrow & \!
 {\cal{L}}({\cal{E}_{\cal{A}}}/{\cal{K}}_{\Omega}) \\     
 & \! A & \! \longmapsto & \! T_A 
\end{array}
$$
onde $T_A (E + {\cal{K}}_{\Omega})= AE + {\mathcal{K}}_{\Omega}.$

A partir daí, pode-se definir a aplicação
$$
\begin{array}{cccc}
 \gamma \ : & \! \cal{A} & \! \longrightarrow & \!
 {\cal{L}}(C(S^{1}\times \{-1,+1\},{\cal{K}}_{\mathbb{Z} \times {\mathbb{S}^1}})) \\     
 & \! A & \! \longmapsto & \! \gamma_{A} \ = \ \Psi T_{A} \Psi^{-1}, 
\end{array}
$$
para $\Psi$ definido em (\ref{psi}).

Seja $C(S^1\times\{-1,+1\},{\mathcal{L}}_{\mathbb{Z} \times {\mathbb{S}^1}})$ o conjunto das funções contínuas em \lb $S^1\times\{-1,+1\}$ que tomam valores em ${\mathcal{L}}_{\mathbb{Z}\times {\mathbb{S}^1}}.$ Identificando as funções em \lb $C(S^1 \times \{-1,+1\}, {\mathcal{L}}_{\mathbb{Z} \times {\mathbb{S}^1}})$ com os operadores de multiplicação correspondentes em \lb $\mathcal{L}(C(S^1 \times \{-1,+1\}, {\mathcal{K}}_{\mathbb{Z}
\times {\mathbb{S}^1}})),$ temos da proposição 2.9 de \cite{melo} que a imagem de $\gamma$ está contida em $C(S^1\times \{-1,+1\},{\mathcal{L}}_{\mathbb{Z} \times {\mathbb{S}^1}}).$ Obtém-se então o seguinte $\ast$-homomorfismo dado em \cite{melo}, proposição 2.4:


\begin{prop}\label{gamma} Existe um $\ast$-homomorfismo
\begin{eqnarray} 
\gamma : {\cal{A}}  \longrightarrow & C(S^1\times\{-1,+1\},{\mathcal{L}}_{\mathbb{Z}\times \mathbb{S}^1}) \\
 ~ ~ A  \longmapsto & \gamma_A  \nonumber
\end{eqnarray}
tal que, $\gamma_{A_i}(e^{2\pi i\varphi}, \pm 1)$ para $i=1,\dots, 6,$ é dado respectivamente por:
$$a(M_x) ~, ~~ b(\pm \infty) ~, ~~ Y_{-j} ~,  ~~ Y_{\varphi} B_i(\varphi - M_j) Y_{-\varphi}~, ~ i= 4,5,6, $$
onde $B_i(\varphi - M_j) \in \mathcal{L}(\ell^2(\mathbb{Z},L^2(\mathbb{S}^1))), ~ i = 4,5,6, $ são operadores de multiplicação pela seqüências $B_i(\varphi -j) \in L^2(\mathbb{S}^1),$ para $B_i$'s  dadas em (\ref{bis}).
\end{prop}

Podemos notar que a imagem de $\gamma$ restrita a $\mathcal{E}_{\mathcal{A}}$ é o conjunto $C(S^1\times\{-1,+1\},{\mathcal{K}}_{\mathbb{Z}\times \mathbb{S}^1})$, e $ \gamma_A = \Psi([A]_{\mathcal{K}_{\Omega}}) $ para $A\in \mathcal{E}_{\mathcal{A}}$. 







\chapter{C*-subálgebras de $\mathcal{A}$}

Neste capítulo, definiremos duas C*-álgebras de $\cal{A},$ as quais denotaremos por $\mathcal{A}^{\dagger}$ e ${{\mathcal{A}}^{\diamond}}.$ Determinaremos seus ideais comutadores e espaços símbolo. 

Mostraremos um resultado análogo ao que foi feito em \cite{cintia}, teorema 8, onde Melo e Silva mostraram que uma  C*-álgebra $A\subset\mathcal{L}_{\mathbb{R}}$, gerada por um certo tipo de operadores de multiplicação, é isomorfa ao produto cruzado $C([-\infty, +\infty]){\rtimes}_{\alpha} \mathbb{Z},$ onde $\alpha$ é dada pela translação por um. 
Ve\-ri\-fi\-ca\-remos que ${{\mathcal{A}}^{\diamond}}$ tem estrutura de produto cruzado, isto é, existe um isomorfismo entre  ${{\mathcal{A}}^{\diamond}}$ e ${{\mathcal{A}}^{\dagger}}\times_{\alpha}\mathbb{Z},$ para uma ação $\alpha $ de $ \mathbb{Z}$ por automorfismos de ${{\mathcal{A}}^{\dagger}}$.

Conhecendo um pouco da estrutura destas álgebras, estaremos aptos a calcular suas K-teorias no próximo capítulo. 


\section{A álgebra $\mathcal{A}^{\dagger}$}

Chamemos de ${\mathcal{A}}^{\dagger}$ a  C*-sub\'algebra de 
${\cal{L}}(L^2(\Omega))$ gerada por:  
\begin{itemize}

\item[(i)] operadores de multiplicação $a(M_x),$ com $a\in C^\infty({\mathbb{S}}^{1});$
 
\item[(ii)] $\Lambda:=(1-\Delta_{\Omega})^{-1/2};$

\item[(iii)] $\frac{1}{i}\frac{\partial}{\partial t}\Lambda,$ com $t\in \mathbb{R};$

\item[(iv)] $\frac{1}{i}\frac{\partial}{\partial \beta}\Lambda,$ com $e^{i\beta} \in {\mathbb{S}}^{1}.$

\end{itemize}

Com base no que foi feito no capítulo anterior, descreveremos o ideal comutador e o espaço símbolo de $\mathcal{A}^{\dagger}$. 

Os geradores do item (i), quando conjugados com a transformada de Fourier $F$, não sofrem nenhuma alteração. Já os geradores dos itens (ii), (iii) e (iv) conjugados com $F$ são os operadores $B_4(M_t), B_5(M_t)$ e $B_6(M_t)$ descritos no início do capítulo 1, página \pageref{opb}. Denotemos por $\mathcal{B}^{\dagger}$ a álgebra $\mathcal{A}^{\dagger}$ conjugada com transformada de Fourier.

\begin{lema}\label{funcaoltda} Para cada $j= 4, 5, 6,~ B_j$ pertence ao conjunto $C_b(\mathbb{R},\mathcal{L}_{\mathbb{S}^1}).$  
\end{lema}

\begin{proof} Como $B_4 \in C_0(\mathbb{R}, {\mathcal{K}}_{\mathbb{S}^1})$ (\cite{cordes1}, proposição 1.1), então fica claro que  $B_4$ está em $C_b(\mathbb{R}, {\mathcal{L}}_{\mathbb{S}^1}).$ 

Como $B_5(\tau) = -\tau B_4(\tau)$ e $B_6(\tau) = \frac{1}{i}\frac{\partial}{\partial\beta}B_4(\tau),$ então estas funções também são contínuas em $\mathbb{R}.$ 

Para mostrarmos que, para cada $\tau$ real, os operadores $B_5(\tau)$ e $B_6(\tau)$ estão em ${\mathcal{L}}_{\mathbb{S}^1}$, primeiro vamos conjugá-los com a   transformada de Fourier discreta $F_d,~B_5'(\tau) := F_d B_5(\tau) F_d^{-1}$ e $B_6'(\tau) := F_d B_6(\tau) F_d^{-1}$, e então mostrar que $B_5'(\tau)$ e $B_6'(\tau)$ são operadores limitados em $\ell^2(\mathbb{Z}).$

Os operadores $B_5'(\tau ),~ B_6'(\tau): \ell^2(\mathbb{Z}) \rightarrow \ell^2(\mathbb{Z})$ são operadores de multiplicação pelas seqüências
$$
-\tau(1+{\tau}^2+j^2)^{-1/2} ~ , ~~ j(1+{\tau}^2+j^2)^{-1/2},  ~~j \in \mathbb{Z}, 
$$
respectivamente. Calculando a norma destes operadores, temos:
\begin{eqnarray*}
|| B_5'(\tau)||^2 & = & \mbox{sup}~\{~|| B_5'(\tau)(v_j)_j||^2_{{\ell}^{2}(\mathbb{Z})};\ (v_j)_j \in
 {{\ell}^{2}({\mathbb{Z}})}, \ ||(v_j)_j||=1  \} \\
& = & \mbox{sup}\{ \sum_{j\in \mathbb{Z}}|\tau (1+{\tau}^2 + j^2)^{-1/2}v_j|^2;\ (v_j)_j \in
 {{\ell}^{2}(\mathbb{Z})}, \ ||(v_j)_j||=1  \} \\
& \le & \sum_{j\in \mathbb{Z}}|v_j|^2 \ = \ 1
\end{eqnarray*}
\begin{eqnarray*}
|| B_6'(\tau)||^2 & = & \mbox{sup}\{ \sum_{j\in \mathbb{Z}}|j (1+{\tau}^2 + j^2)^{-1/2}v_j|^2;\ (v_j)_j \in
 {{\ell}^{2}(\mathbb{Z})}, \ ||(v_j)_j||=1  \} \\
& \le & \sum_{j\in \mathbb{Z}}|v_j|^2 \ = \ 1
\end{eqnarray*}

Logo, $B_5'(\tau)$ e $B_6'(\tau)$ são operadores limitados em $\ell^2(\mathbb{Z})$. 
Como a conjugação é um isomorfismo, temos que $B_5(\tau)$ e $B_6(\tau)$ são operadores limitados em $L^2(\mathbb{S}^1)$ e portanto as funções $B_5$ e $B_6$ pertencem a $C_b(\mathbb{R},\mathcal{L}_{{\mathbb{S}}^1}).$ 

\end{proof}

Para conhecer a estrutura da C*-álgebra $\mathcal{A}^{\dagger}$ e podermos descrever seu espaço símbolo, primeiro encontraremos um isomorfismo entre a álgebra gerada pelos operadores $A_4, A_5$ e $A_6$ e a álgebra de funções contínuas definidas numa compactificação de ${\mathbb{R}\times\mathbb{Z}}$. Antes de enunciarmos o lema que nos mostra este isomorfismo, definiremos a compactificação de ${\mathbb{R}\times\mathbb{Z}}$.

Considere a seguinte compactificação de $\mathbb{R}^2$ dada em \cite{cordes2}, página 133. Dado o homeomorfismo de $\mathbb{R}^2$ na bola aberta
$$
\begin{array}{cccc}
 h~: & \! {\mathbb{R}}^{2} & \! \longrightarrow & \! B = \{y \in {\mathbb{R}}^{2};~ ||y|| < 1 \} \\     
     & \! y & \! \longmapsto & \! \displaystyle\frac{y}{(1+||y||^{2})^{1/2}} 
\end{array}
$$
temos o conjunto $CS(\mathbb{R}^2) = \{ f \in C(\mathbb{R}^2) ; f \circ h^{-1} \in C(D)\},$ onde $D = \{y \in {\mathbb{R}}^{2};~ ||y|| \le 1 \}$ é o fecho de $B$. $CS(\mathbb{R}^2)$ é a classe das funções contínuas limitadas $f$ sobre $\mathbb{R}^2$ tais que a função $g(y) = f \circ h^{-1}(y)$ admite uma extensão contínua em $D$. Temos então que $CS(\mathbb{R}^2) = C(\mathbb{B}^2)$, onde $\mathbb{B}^2$ é a compactificação de $\mathbb{R}^2$ homeomorfa a $D$. 


Denotemos por $\overline{\mathbb{R}\times\mathbb{Z}}$ o fecho de $\mathbb{R}\times\mathbb{Z}$ em $\mathbb{B}^{2}$ e por $S_{\infty}^1$ a fronteira de $\mathbb{B}^{2}$. Vamos mostrar que $\overline{\mathbb{R}\times\mathbb{Z}}$ é $(\mathbb{R}\times\mathbb{Z}) \cup S_{\infty}^1.$

Seja $(x_1, x_2)$ em $S_{\infty}^1$. Então a este elemento está relacionado um único elemento $z=(z_1, z_2)\in S^{1}=\partial B$ pelo homeomorfismo entre $\mathbb{B}^2$ e $D$.  Tome uma seqüência de racionais
$\{(\frac{p_j}{m_j}, \frac{q_j}{n_j})\}_j $ em $B$ que converge para $(z_1,
z_2)$ de forma que $ m_j , n_j \to \infty .$ Considerando então a seqüência
$\{(n_j p_j , m_j q_j)\}_j$ em  ${\mathbb{Z}}^2 \subset
\mathbb{R}\times\mathbb{Z} \subset {\mathbb{R}}^{2},$ temos 
\begin{eqnarray*}
h((n_j p_j, m_j q_j)) & = & \displaystyle\frac{(n_j p_j, m_j
q_j)}{\sqrt[]{1 + ({n_j p_j})^2 + ({ m_j q_j})^2}} \\
& = & \displaystyle\frac{(\displaystyle\frac{p_j}{m_j}, \displaystyle\frac{q_j}{n_j})}{\sqrt[]{\frac{1}{(n_j m_j)^2} + \frac{p_j^2}{m_j^2} + \frac{q_j^2}{n_j^2}}} \\
& \stackrel{j\to \infty}{\longrightarrow} & \frac{(z_1,z_2)}{\sqrt[]{z_1^2 + z_2^2}} ~ = ~ (z_1,z_2)
\end{eqnarray*}
Então, $(n_j p_j , m_j q_j) \stackrel{j\to \infty}{\rightarrow} (x_1, x_2).$  

Fica assim bem caracterizada a seguinte compactificação de $\mathbb{R} \times \mathbb{Z}:$ 
\begin{eqnarray*}
\overline{\mathbb{R}\times\mathbb{Z}} & = &
\{ (t,j) \in \mathbb{R}\times\mathbb{Z}\} \cup  \partial\mathbb{B}^{2} \\
& = & (\mathbb{R}\times\mathbb{Z}) \cup S_{\infty}^1.
\end{eqnarray*}

Agora podemos enunciar o seguinte lema:

\begin{lema}\label{compactificacao}  Seja $\cal{C}$ a C*-álgebra comutativa gerada pelos operadores
$A_4,A_5$ e $A_6.$   Então 
$$
{\cal{C}} \ \cong \ C(\overline{\mathbb{R}\times\mathbb{Z}})
$$
\end{lema}

\begin{proof} Se conjugarmos os operadores $A_4, A_5, A_6$ com o operador $F\otimes F_d^{-1},$ obteremos ope\-ra\-do\-res de multiplicação em $L^2(\mathbb{R}\times \mathbb{Z})$ dados, respectivamente, pelas funções $B_4', B_5', B_6' : \mathbb{R}\times\mathbb{Z} \rightarrow \mathbb{R}$ definidas abaixo:
$$B_4'(\tau,j) = (1+{\tau}^2+j^2)^{-1/2} ~~~ B_5'(\tau,j) = -\tau B_4'(\tau,j) ~~~~ B_6'(\tau,j) = -j  B_4'(\tau,j) ~, ~~ (\tau, j) \in \mathbb{R}\times \mathbb{Z}$$
(lema \ref{funcaoltda}). Denotando por ${\mathcal{C}}'$ a álgebra $(F^{-1}\otimes F_d){\mathcal{C}}(F\otimes F_d^{-1})$, podemos vê-la como uma subálgebra de $C(\overline{\mathbb{R}\times\mathbb{Z}})$. 

Queremos mostrar, usando o teorema de Stone-Weierstrass, que ${\mathcal{C}}' = C(\overline{\mathbb{R}\times\mathbb{Z}})$. 
Como \linebreak $(B_4'(t,j))^2 + (B_5'(t,j))^2 + (B_6'(t,j))^2=1$, para todo $(t,j)\in \mathbb{R}\times\mathbb{Z},$ a álgebra $\mathcal{C}'$ possui unidade. Vamos verificar agora,  que sempre é possível encontrar uma função em $\mathcal{C}'$ que separa pontos. 

\textit{Caso 1.}  
Sejam $(t_1, k_1)\neq(t_2, k_2)$ em $\mathbb{R}\times\mathbb{Z}$. Se $B_4'(t_1, k_1)\neq B_4'(t_2, k_2), $ temos o que queríamos. Caso contrário, então ou 
$$B_5'(t_1, k_1) = \displaystyle\frac{-t_1}{(1+t_1^2+k_1^2)^{1/2}} \neq \displaystyle\frac{-t_2}{(1+t_2^2+k_2^2)^{1/2}} = B_5'(t_2, k_2), $$ 
ou 
$$B_6'(t_1, k_1) = \displaystyle\frac{-k_1}{(1+t_1^2+k_1^2)^{1/2}} \neq \displaystyle\frac{-k_2}{(1+t_2^2+k_2^2)^{1/2}} = B_6'(t_2, k_2), $$
já que $t_1 \neq t_2$ ou $k_1 \neq k_2.$

\textit{Caso 2.}
Sejam $(t,k)$ em $\mathbb{R}\times\mathbb{Z}$ e $z=(z_1,z_2)$ em $S_{\infty}^1$. Vamos primeiro calcular $B_4'$ em $z$.
Para facilitar as contas, suponha que $B_4'$ está definida em todo $\mathbb{R}^2$. Temos então que $g = B_4'\circ h^{-1}$ admite uma extensão contínua em $D= \{y \in {\mathbb{R}}^{2};~ ||y||\leq 1 \}$. Logo, para $y \in B= \{y \in {\mathbb{R}}^{2};~ ||y|| < 1 \}$ temos 
$$ 
B_4'\circ h^{-1}(y) = B_4'\left(\frac{y}{(1-||y||^2)^{1/2}}\right) = \frac{1}{\left(1 + \frac{y_1^2}{1-||y||^2} + \frac{y_2^2}{1-||y||^2} \right)^{1/2}} = (1- y_1^2 - y_2^2)^{1/2} . 
$$
Seja $(y_j)_{j\in\mathbb{Z}}$ uma seqüência em $B$ que converge para $s=(s_1,s_2)$ em $S^1$ quando $j\rightarrow\infty$, então 
$$
B_4'\circ h^{-1} (y_j) = (1- (y_j)_1^2 - (y_j)_2^2)^{1/2} \stackrel{j\rightarrow\infty}{\longrightarrow} (1- s_1^2 - s_2^2)^{1/2} = 0,
$$
já que $||s||=1$. Assim, para qualquer $s=(s_1,s_2)\in S^1 = \partial D, ~ B_4'\circ h^{-1}(s) = 0.$  Portanto, 
$$ B_4'(z_1,z_2) = 0 .$$
Como
$$ B_4'(t,k)=\displaystyle\frac{1}{(1+t^2+k^2)^{1/2}} \neq 0 ,$$
então $B_4'(t,k) \neq B_4'(z)$. 

\textit{Caso 3.}
Sejam $z=(z_1,z_2)$ e $w=(w_1,w_2)$ distintos em $S_{\infty}^1$. Usando o mesmo raciocínio feito no caso 2, primeiramente faremos a composição das funções $B_5'$ e $B_6'$ com $h^{-1}$ para depois podermos calculá-las em $z$ e $w$. Para $y\in B,$
$$ 
B_5'\circ h^{-1}(y) = B_5'\left(\frac{y}{(1-||y||^2)^{1/2}}\right) = \frac{-y_1/ \sqrt{1-||y||^2}}{\left(1 + \frac{y_1^2}{1-||y||^2} + \frac{y_2^2}{1-||y||^2} \right)^{1/2}} = -y_1,
$$
$$
B_6'\circ h^{-1}(y) = B_6'\left(\frac{y}{(1-||y||^2)^{1/2}}\right) = \frac{-y_2/ \sqrt{1-||y||^2}}{\left(1 + \frac{y_1^2}{1-||y||^2} + \frac{y_2^2}{1-||y||^2} \right)^{1/2}} = -y_2.
$$
Dada $(-y_j)_{j\in\mathbb{Z}}$ uma seqüência em $B$ que converge para $s=(s_1,s_2)$ em $S^1$ quando $j\rightarrow\infty$, temos 
$$
B_5'\circ h^{-1} (-y_j) = (-y_j)_1 \stackrel{j\rightarrow\infty}{\longrightarrow} s_1 ~~~\mbox{e}~~~
B_6'\circ h^{-1} (-y_j) = (-y_j)_2 \stackrel{j\rightarrow\infty}{\longrightarrow} s_2.
$$
Assim, para qualquer $s=(s_1,s_2)\in S^1 = \partial D, ~ B_5'\circ h^{-1}(s_1, s_2) = s_1$ e $B_6'\circ h^{-1}(s_1, s_2) = s_2$.  Portanto, 
$$ B_5'(z_1,z_2) = z_1 \neq w_1 = B_5'(w_1,w_2)~~~\mbox{ou}~~~ B_6'(z_1,z_2) = z_2 \neq w_2 = B_6'(w_1,w_2)$$
Assim, $B_5'$ ou $B_6'$ separam pontos. 

Concluímos então que $\mathcal{C}$ é isomorfa a $C(\overline{\mathbb{R}\times\mathbb{Z}}).$

\end{proof}

\begin{obs} Vimos no teorema \ref{ma2} que os $\sigma$-símbolos dos operadores $A_4, A_5, A_6$ calculados em $(t, x, (\tau, \xi), e^{i\theta})$, onde $(\tau, \xi)\in S^1$ são dados por $0, \tau, \xi$, respectivamente. Note que isto coincide com os cálculos de $B_4', B_5', B_6'$ nos pontos $(z_1, z_2)\in S^1_{\infty} $. 
\end{obs}

\begin{teo}\label{edagger} O ideal comutador de $\mathcal{B}^{\dagger}$ é $C_0(\mathbb{R}, {\cal{K}}_{{\mathbb{S}}^1})$.
\end{teo}

\begin{proof}
Seja ${\cal{E}}^{\dagger}_{\cal{B}}$ o ideal comutador de $\mathcal{B}^{\dagger}$.


Note que a álgebra ${\mathcal{B}}^{\dagger}$ está contida em $C_b(\mathbb{R}, \mathcal{S}),$ onde $\mathcal{S}$ foi definido no capítulo 1, página \pageref{s}. De fato,  $a(M_x) \in  \mathcal{S}$ para $a\in C^{\infty}(\mathbb{S}^1)$ (este operador não depende da variável real), e para $\tau\in \mathbb{R}, ~B_i(\tau)\in \mathcal{S},$ pois $B_i(\tau) = F_d^{-1} B_i'(\tau) F_d $, para $i=4,5,6.$

Fixado $\tau \ \in \ \mathbb{R},$ mostremos que $\{B(\tau);B \ \in \ {\mathcal{B}}^{\dagger}\} = \mathcal{S}.$ 

Primeiro, ${\cal{V}}:= \{B(\tau); B  \in {\mathcal{B}}^{\dagger}\}$ é C*-subálgebra de 
${\cal{L}}_{{\mathbb{S}}^1}$ pois a aplicação $ B \in  {\mathcal{B}}^{\dagger} \mapsto B(\tau) \in  {\cal{L}}_{{\mathbb{S}}^1}$ 
é um C*-homomorfismo, portanto tem imagem fechada.
Queremos mostrar que $F_d {\cal{V}} F_{d}^{-1} = F_d \mathcal{S} F_{d}^{-1}$ 
usando o Teorema de Stone-Weierstrass. 
Os operadores $B_i'(\tau) = F_d B_i(\tau)F_d^{-1},\ i = 4,5,6 ,$ 
podem ser vistos como funções em 
$ C(\mathbb{Z}\cup \{-\infty, +\infty\}),$ já que 
$$
-1 \le B_i'(\tau, j) \le 1 \ , \ \ \forall \ j \ \in \ \mathbb{Z}\cup \{-\infty, +\infty\}. 
$$
Temos ainda que $B_6'(\tau,j)=-j(1+{\tau}^2 + j^2)^{-1/2}$ separa pontos em
$C(\mathbb{Z}\cup \{-\infty, +\infty\}),$ pois 
$$ 
\frac{j}{(1+{\tau}^2 + j^2)^{1/2}} =  \frac{k}{(1+{\tau}^2 + k^2)^{1/2}} \ \iff \ j=k. 
$$
Como  $F_d {\cal{V}} F_{d}^{-1}$ contém a identidade, pois $Id = {B_4'}^2(\tau)+{B_5'}^2(\tau)+{B_6'}^2(\tau),$  então $F_d {\cal{V}} F_d^{-1} = F_d \mathcal{S} F_{d}^{-1}$  e portanto $ {\cal{V}} = \mathcal{S}.$

Sabe-se de \cite{cordes1}, proposição 1.2 e 1.3, que todos os comutadores de $\cal{B}^{\dagger}$ e
seus adjuntos estão em $C_0(\mathbb{R},{\cal{K}}_{{\mathbb{S}}^1})$, logo
${\mathcal{E}}^{\dagger}_{\cal{B}} \subseteq C_0(\mathbb{R},{\cal{K}}_{{\mathbb{S}}^1})$. Assim, para
$\tau \in \mathbb{R},$ o conjunto 
${\cal{U}}:= \{ A(\tau); ~ A \ \in \ {\mathcal{E}}^{\dagger}_{\cal{B}}\} \subseteq {\cal{K}}_{{\mathbb{S}}^1}.$ 
Na verdade, ${\cal{U}}:= \{ A(\tau); ~A \ \in \ {\mathcal{E}}^{\dagger}_{\cal{B}}\} = {\cal{K}}_{{\mathbb{S}}^1}.$  
De fato,  dado $T$ pertencente a $\mathcal{S},$ para um $\tau\in \mathbb{R},$ existe 
$ B  \in  \mathcal{B}^{\dagger}$ tal que $B(\tau) = T.$ 
Daí, para $A \ \in \ \mathcal{E}^{\dagger}_{\cal{B}}$
$$
A(\tau) T = A(\tau) B(\tau) = (AB)(\tau) \ \in \ \cal{U}, 
$$
pois $AB$ está em $\mathcal{E}^{\dagger}_{\cal{B}}.$  Logo, $\cal{U}$ é um ideal de 
$\mathcal{S}$ não nulo contido em ${\cal{K}}_{{\mathbb{S}}^1}.$ Segue que  ${\cal{U}} = {\cal{K}}_{{\mathbb{S}}^1}.$  

Por \cite{dixmier}, 10.4.5, temos que $\mathcal{E}^{\dagger}_{\cal{B}} =
C_0(\mathbb{R},{\cal{K}}_{{\mathbb{S}}^1}) $ se os estados puros de
$C_0(\mathbb{R},{\cal{K}}_{{\mathbb{S}}^1})$ são separados por $\mathcal{E}^{\dagger}_{\cal{B}}.$ Os
estados puros de $C_0(\mathbb{R},{\cal{K}}_{{\mathbb{S}}^1})$ são os funcionais lineares da
forma
\begin{equation}\label{estadopuro}
f_{t,v}(A) = (v,A(t)v) \ , \ A \ \in \ C_0(\mathbb{R},{\cal{K}}_{{\mathbb{S}}^1}),
\end{equation}
onde  $(v,A(t)v)$ denota o produto interno de  $v$ por $A(t)v$, $t\in\mathbb{R}$ fixo e $v\in L^2({\mathbb{S}}^1)$ tem norma 1 (ver \cite{dixmier}, 2.5.2 e 4.1.4, e \cite{takesaki}, IV 4.14). Para mostrar
que $\mathcal{E}^{\dagger}_{\cal{B}} = C_0(\mathbb{R},{\cal{K}}_{{\mathbb{S}}^1}),$ é suficiente
mostrar que, dados dois funcionais da forma (\ref{estadopuro}), $f_{t_1, v_1}\ne
f_{t_2, v_2},$ podemos achar $A \ \in \ \mathcal{E}^{\dagger}_{\cal{B}}$ tal que
$f_{t_1, v_1}(A)\ne f_{t_2, v_2}(A).$

\begin{enumerate}
\item Suponha $t_1\neq t_2, ~t_1, t_2 \in \mathbb{R}$. Tome $D \ \in \ F_d \mathcal{E}^{\dagger}_{\cal{B}} F_d^{-1}\subseteq C_0(\mathbb{R},{\mathcal{K}}(\ell^2({\mathbb{Z}})))$ tal que $D(t_1) = (w_1, \cdot)w_1,$ para $ w_1 \in {\ell}^2(\mathbb{Z})$ e defina para cada $l\in \mathbb{Z},$ $a_l \in C(\overline{\mathbb{R}\times\mathbb{Z}})$ da seguinte maneira:
$$ 
a_l(t_1, k) = 1 \ , \ \forall \ k \ \in \ \mathbb{Z} ~~\mbox{e} ~~
a_l(t_2, k) = 0 \ , \ \mbox{se} \ |k|<l.
$$  
Considerando os funcionais lineares $f_{t_i, w_i}(D) = (w_i, D(t_i)w_1)$ em $C_0(\mathbb{R},{\mathcal{K}}(\ell^2({\mathbb{Z}})))$ para $t_i\in\mathbb{R}$  e $w_i \in \ell^2(\mathbb{Z}), ~i=1,2,$ temos 
$$ 
f_{t_1,w_1}(a_l D) = (w_1, (a_l D)(t_1)w_1) = 1~, 
$$
$$
f_{t_2,w_2}(a_l D) = (w_2, (a_l D)(t_2)w_2) = (w_2, a_l(t_2)(M_k) D(t_2)w_2) \stackrel{l\rightarrow\infty}{\longrightarrow }0 ~.
$$
Isto é, existe $l$ em $\mathbb{Z}$ tal que $f_{t_1,w_1}(a_l D)\neq
f_{t_2,w_2}(a_l D)$ e  $a_l D \ \in \ F_d \mathcal{E}^{\dagger}_{\cal{B}}
F_d^{-1}.$  Portanto, para $t_1, t_2$ reais distintos e $v_1=F_d^{-1}w_1$ e $v_2=F_d^{-1}w_2$ em $L^2(S^1)$, tome $ A = F_d^{-1} a_l D  F_d$ em
$\mathcal{E}^{\dagger}_{\cal{B}}$ para termos $f_{t_1,v_1}(A)\neq f_{t_2,v_2}(A)$

\item Suponha $t_1 = t_2 = t_0$. Se  $v_1$ e  $v_2$ têm norma 1, então
  $|(v_1,v_2)|=1 \ \iff \ v_1 = \alpha v_2, \ |\alpha|=1,$ e portanto temos os
  mesmos estados puros. Logo, tomemos $v_1$ e $v_2$ tais que  $|(v_1,v_2)|\neq
  1.$ Temos então, para $A \ \in \ \mathcal{E}^{\dagger}_{\cal{B}},$ onde
  $A(t_0)=(v_1, \cdot)v_1,$
$$
f_{t_0,v_1}(A) = (v_1, A(t_0)v_1) = (v_1, (v_1,v_1)v_1) = 1 \neq
|(v_1,v_2)|^2 = f_{t_0,v_2}(A).
$$
\end{enumerate}

Portanto $\mathcal{E}^{\dagger}_{\cal{B}}$ separa os estados puros de $C_0(\mathbb{R},{\cal{K}}_{{\mathbb{S}}^1})$ e podemos concluir que $\mathcal{E}^{\dagger}_{\cal{B}}= C_0(\mathbb{R}, {\cal{K}}_{{\mathbb{S}}^1})$.
\end{proof}

Denotemos por ${\mathcal{E}}^{\dagger}$ o ideal comutador de ${\mathcal{A}}^{\dagger}$. Lembrando que ${\mathcal{B}}^{\dagger} = F^{-1}{\mathcal{A}}^{\dagger}F,$  podemos então escrever ${\mathcal{E}}^{\dagger} = F{\mathcal{E}}^{\dagger}_{\mathcal{B}}F^{-1}.$

\begin{teo}\label{mdagger} O espaço símbolo ${\bf M}^{\dagger}$ de ${\mathcal{A}}^{\dagger}$ é homeomorfo a $\mathbb{S}^1 \times S_{\infty}^1$ e os $\sigma$-símbolos da classe dos geradores da C*-álgebra $\cal{A}^{\dagger}$ calculados em $(x, w=\tau +i\xi)$ são respectivamente dados por 
\begin{equation}\label{sigmadagger}
 a(x),~ ~ ~ 0, ~ ~ ~ \tau,~ ~ ~ \xi.
\end{equation}
\end{teo}

\begin{proof} Considere as inclusões
$$
\begin{array}{ccccccccc}
 i_1\ :  & \! C^{\infty}(\mathbb{S}^1) & \! \longrightarrow &
 \!\displaystyle\frac{\cal{A}^{\dagger}}{\mathcal{E}^{\dagger}} & \! \ \ \  & \!
 i_2\ :  & \! \cal{C} & \! \longrightarrow &
 \!\displaystyle\frac{\cal{A}^{\dagger}}{\mathcal{E}^{\dagger}} \\     
 & \! a & \! \longmapsto & \! [a(M_x)]_{\mathcal{E}^{\dagger}} & \!  & \!  & \! A & \! \longmapsto & \! [A]_{\mathcal{E}^{\dagger}}
\end{array}
$$

Temos  que $\mathbb{S}^1$ e  $\overline{\mathbb{R}\times\mathbb{Z}}$ são os espaços
símbolo de  $C^{\infty}(\mathbb{S}^1)$ e $\cal{C},$ respectivamente, pois pelo lema \ref{compactificacao}, $\cal{C}$ e
$C(\overline{\mathbb{R}\times\mathbb{Z}})$ são isomorfos. Seja a aplicação dual 
$$ 
\begin{array}{cccc}
 {\iota}^{*} \ :  & \! {\bf M}^{\dagger} & \! \longrightarrow &
 \!\mathbb{S}^1 \times  \overline{\mathbb{R}\times\mathbb{Z}} \\
 & \! \omega & \! \longmapsto & \! (\omega \circ i_1, \omega \circ i_2) 
\end{array}
$$
 
As imagens de $i_1$ e $i_2$ geram $\cal{A}^{\dagger}/\mathcal{E}^{\dagger},$
então ${\iota}^{*}$ é injetora. Como $ {\bf M}^{\dagger}$ é compacto, $\mathbb{S}^1
\times  (\overline{\mathbb{R}\times\mathbb{Z}})$ é Hausdorff e ${\iota}^{*}$ é
contínua, temos que ${\iota}^{*}$ é um homeomorfismo sobre a sua
imagem. Determinemos então a imagem de ${\iota}^{*}$.

Note que 
$$
\mathbb{S}^1 \times (\overline{\mathbb{R} \times \mathbb{Z}}) \ \cong \
(\mathbb{S}^1 \times \mathbb{R} \times \mathbb{Z}) \cup (\mathbb{S}^1 \times S_{\infty}^1).
$$

Suponha que exista $(z,t,j) \in (\mathbb{S}^1 \times {\mathbb{R} \times \mathbb{Z}})\cap Im {\iota}^{*}.$ 
Considere a aplicação do $\sigma$-símbolo 
$$
\begin{array}{cccc}
\sigma^{\dagger} \ :  & \! \mathcal{A}^{\dagger} & \! \longrightarrow  & \! C({\bf M}^{\dagger}) \\
                      & \!      A                & \! \longmapsto      & \! \sigma^{\dagger}_A(\omega) = \omega(A) 
\end{array}
$$ 
Associamos ao ponto $(z,t,j)$ um funcional linear $\omega_0 \in {\bf M}^{\dagger}$ dado pela evaluação neste ponto. 

Para $A_4 \in \mathcal{C}\subset \mathcal{A}^{\dagger},$ temos que $\omega_0(A_4) = (0, \omega_0 \circ i_2(A_4)),$ onde $\omega_0 \circ i_2$ é a evaluação em $(t,j)$. Portanto,  
$$ {\sigma}^{\dagger}_{A_4}(z,t,j)= B_4'(t,j) = \displaystyle\frac{1}{(1+t^2+j^2)^{1/2}} \neq 0 ,$$
onde $B_4'\in C(\overline{\mathbb{R} \times \mathbb{Z}})$.

Mas  $A_4 \in {\mathcal{E}}^{\dagger}, $ então  ${\sigma}^{\dagger}_{A_4} = 0. $ Portanto, com esta contradição obtemos  
$(\mathbb{S}^1 \times{\mathbb{R} \times \mathbb{Z}})\cap Im {\iota}^{*} = \emptyset .$

Temos então que $Im {\iota}^{*}\subseteq \mathbb{S}^1 \times S_{\infty}^1.$ Vamos mostrar que vale a igualdade.

\begin{itemize} 
\item \textit{Afirmação 1:} Se $(z_0,v_0) \in Im {\iota}^{*},$ então $(z,v_0) \in Im {\iota}^{*}$ para todo $z \in \mathbb{S}^1 . $

Seja $\zeta \in \mathbb{S}^1$. Considere o operador $T_{\zeta}\in \mathcal{L}_{\Omega}$ tal que $(T_{\zeta}u)(t,x) = u(t,x\cdot\zeta)$. Defina 
$ \phi_{\zeta} :\mathcal{A}^{\dagger}\rightarrow \mathcal{A}^{\dagger}$ por 
$$\phi_{\zeta}(A) = T_{\zeta} A T_{{\zeta}^{-1}}.$$
Assim, $\phi_{\zeta}$ é um automorfismo contínuo e $\phi_{\zeta}(a(M_x))=a(M_x\cdot\zeta),$ para $a \in C^{\infty}(\mathbb{S}^1)$, e $\phi (A) = A, $ para $A \in \mathcal{C}.$

Seja $w_0$ o funcional linear multiplicativo de $\mathcal{A}^{\dagger}/\mathcal{E}^{\dagger}$ associado ao ponto $(z_0,v_0) \in Im{\iota}^{*}$. Defina um novo funcional $w_1$ da seguinte forma:
$$w_1([A]_{\mathcal{E}^{\dagger}}) = w_0([\phi_{\zeta}(A)]_{\mathcal{E}^{\dagger}});$$
então $w_1$ é o funcional linear multiplicativo de $\mathcal{A}^{\dagger}/\mathcal{E}^{\dagger}$ associado ao ponto $(z_0 \cdot \zeta,v_0) \in \mathbb{S}^1 \times S_{\infty}^1 $. Logo, $(z,v_0) \in Im {\iota}^{*}$ para todo $z\in \mathbb{S}^1$. 
\end{itemize}

Suponha agora que exista um ponto $(z_1,v_1)$ em $\mathbb{S}^1 \times S_{\infty}^1 $ que não pertença a $Im{\iota}^{*}.$ Portanto, pela afirmação acima, podemos concluir que $(z, v_1)$ não pertence a $Im{\iota}^{*}$ para todo $z$ em $\mathbb{S}^1$. Então existe uma vizinhança aberta $N_{v_1}$ de $v_1$ em $\overline{\mathbb{R} \times \mathbb{Z}}$ tal que $(\mathbb{S}^1 \times N_{v_1}) \cap Im{\iota}^{*} = \emptyset$. 

Seja $B$ em $C(\overline{\mathbb{R} \times \mathbb{Z}})$ tal que $B(v_1)\neq 0$ e o suporte de $B$ está contido em $N_{v_1}.$ A $B$ corresponde um único operador $A$ em $\mathcal{C} \subset \mathcal{A}^{\dagger}.$ Daí, 
$${\sigma}_{A}(z,v)=B(v)=0 ~ ~ ~  \forall ~ ~ (z,v)\in Im {\iota}^* .$$

\begin{itemize} 
\item \textit{Afirmação 2:} $A$ não pertence ao ideal comutador $\mathcal{E}^{\dagger}.$

Suponha que $A~\in ~\mathcal{E}^{\dagger}.$ Então $B$ pertence a $C_0(\mathbb{R},\mathcal{K}_{\mathbb{Z}})$, 
e portanto vale o limite
$$ 
\lim_{\tau \rightarrow\infty} \sup_{k \in \mathbb{Z}} |B(\tau , k)| = 0 .
$$
Seja $\{(\tau _j,s_j)\}_{j\in \mathbb{Z}}$ uma seqüência em $\mathbb{R} \times \mathbb{Z}$ tal que $(\tau _j,s_j)\rightarrow\infty$ quando $j\rightarrow\infty$ e \linebreak $ h((\tau _j,s_j))\rightarrow v_1$. 
Dado $\epsilon > 0$, seja $\tau_0$ real tal que para todo $\tau > \tau_0$, 
$$
\sup_{k \in \mathbb{Z}} |B(\tau , k)| < \epsilon .
$$
Tome $j_0$ inteiro tal que, para todo $j>j_0,$ tem-se $\tau_j>\tau_0$. Temos então para $j>j_0,$
$$ 
|B(\tau_j , s_j)| < \sup_{k \in \mathbb{Z}} |B(\tau_j ,k)| < \epsilon .
$$
Mas $ B(\tau_j , s_j) \stackrel{j\rightarrow\infty}{\longrightarrow} B(v_1)$ e daí  $B(v_1)=0,$ absurdo!
\end{itemize}
Logo, $(z,v) \in Im{\iota}^{*}$ para todo  $(z,v)$ em $\mathbb{S}^1 \times S_{\infty}^1 $.

Vamos agora descrever a aplicação do símbolo:
$$ \sigma^{\dagger}: \mathcal{A}^{\dagger}\longrightarrow C(\mathbb{S}^1 \times S_{\infty}^1) . $$

Para cada ponto $(x,w)\in \mathbb{S}^1 \times S_{\infty}^1,$ existem funcionais lineares multiplicativos $f_1$ associado a $x$ e $f_2$ associado a $w$. Como ${\bf M}^{\dagger}$ e  $\mathbb{S}^1 \times S_{\infty}^1$ são homeomorfos, existe um funcional linear $\omega\in {\bf M}^{\dagger}$ tal que pela aplicação dual $\iota^{*}$, temos $\omega\circ i_1 = f_1$ e  $\omega\circ i_2 = f_2$. Portanto, podemos associar a $\omega$ o ponto $(x,w)$. Logo, 
$$ \sigma^{\dagger}_{a(M_x)}(x,w) = \omega_{(x,w)}(a(M_x)) = f_1(a(M_x)) = a(x).$$
Escrevendo $w=(\tau,\xi) \in S_{\infty}^1$, temos
\begin{enumerate}
\item $ \sigma^{\dagger}_{A_4}(x,w) = \omega_{(x,w)}(A_4) = f_2(A_4) = B_4'(w) = 0$
\item $ \sigma^{\dagger}_{A_5}(x,w) = \omega_{(x,w)}(A_5) = f_2(A_5) = B_5'(\tau, \xi) = \tau$
\item $ \sigma^{\dagger}_{A_6}(x,w) = \omega_{(x,w)}(A_6) = f_2(A_6) = B_6'(\tau, \xi) = \xi$
\end{enumerate}

Podemos ainda dizer que $\sigma^{\dagger}$ é igual a $\sigma$ restrita a álgebra $\mathcal{A}^{\dagger}.$

\end{proof}

\section{A álgebra do produto cruzado}

\subsection{O produto cruzado}

Denote por ${\mathcal{A}}^{\diamond}$ a C*-\'algebra ${\cal{A}}^{\dagger}$
acrescida dos operadores de multiplicação pelas funções periódicas.  
Denotemos por ${\mathcal{B}}^{\diamond}$ a C*-álgebra ${\mathcal{A}}^{\diamond}$ conjugada com a transformada de Fourier. 

Nesta seção iremos mostrar que ${\mathcal{B}}^{\diamond}$ é isomorfa ao produto cruzado de ${\mathcal{B}}^{\dagger}$ por uma ação que definiremos abaixo. 

Identificando ${\cal{B}}^{\dagger}$ como subálgebra de $C_b(\mathbb{R}, \mathcal{L}_{\mathbb{S}^1})$ (ver demostração do teorema \ref{edagger}), considere o seguinte automorfismo de translação $\alpha$:
$$[\alpha(B)](\tau) = B(\tau - 1), ~ \tau \in \mathbb{R}, ~ B \in \mathcal{B}^{\dagger}.$$

Lembre que a C*-álgebra gerada por $B_4 , B_5 $ e $B_6$ é isomorfa a $C(\overline{\mathbb{R}\times \mathbb{Z}})$ (lema \ref{compactificacao}). Como $C(\overline{\mathbb{R}\times \mathbb{Z}})$ é invariante por translação na variável real, então a C*-álgebra gerada por $B_4 , B_5 $ e $B_6$
é invariante por translação em $ \mathbb{R}.$ Portanto, $\mathcal{B}^{\dagger}$ é invariante por translação e o automorfismo $\alpha$ acima está bem definido. 

Melo e Silva \cite{cintia}, teorema 8, demonstraram que a C*-álgebra $A\subset\mathcal{L}_{\mathbb{R}}$, gerada por operadores de multiplicação por funções contínuas em $[-\infty, +\infty]$ e por multiplicadores por funções contínuas de período $2\pi$, é isomorfa ao produto cruzado $C([-\infty, +\infty]){\rtimes}_{\alpha} \mathbb{Z},$ onde $\alpha$ é dada pela translação por um. 

De forma análoga, temos o seguinte resultado.  

\begin{teo}\label{prodcruz} Seja $\alpha$ a ação descrita acima. Então 
$${\mathcal{B}}^{\diamond} \cong \mathcal{B}^{\dagger}{\rtimes}_{\alpha}\mathbb{Z} $$
onde
${\mathcal{B}^{\dagger}}\times_{\alpha}\mathbb{Z}$ é a C*-álgebra envolvente (\cite{dixmier}, 2.7.7) da álgebra de Banach com involução 
$\ell^1(\mathbb{Z},\mathcal{B}^{\dagger})$ de todas $\mathbb{Z}$-seqüências somáveis em $\mathcal{B}^{\dagger},$ equipadas com o produto de convolução 
$$({\bf A \ast B})(n) = \sum_{k\in\mathbb{Z}}^{}A_k\alpha^k(B_{n-k}), ~ n\in \mathbb{Z}, ~ {\bf A}=(A_k)_{k\in \mathbb{Z}},~ {\bf B}=(B_k)_{k\in \mathbb{Z}}, $$ e involução ${\bf A}^*(n)=\alpha^n(A^*_{-n}).$
\end{teo}

A  demonstração deste teorema segue exatamente à feita em \cite{cintia}, mas apresentaremos sua prova com as devidas mudanças após a seguinte proposição, \cite[proposição 2.9]{exel}, a qual será necessária no decorrer da demonstração.   

\begin{prop}\label{exel}
Sejam $G$ e $G'$ C*-álgebras e sejam $\alpha$ e $\alpha '$ ações de $S^1$ em $G$ e $G',$ respectivamente. Suponha que $\Gamma : G\rightarrow G' $ seja um homomorfismo covariante. Se a restrição de $\Gamma$ à subálgebra dos pontos fixos $G_0$ é injetiva, então $\Gamma$ é injetiva.
\end{prop}

\begin{proof}[Demonstração (teorema \ref{prodcruz})]
Considere $K({\mathbb{Z}}, \mathcal{B}^{\dagger})$ como sendo o conjunto das sequências com apenas um número finito de entradas não nulas em $\mathcal{B}^{\dagger}.$ 
Defina para as sequências $B=(B_j)_{j \in \mathbb{Z}}$ em $ K({\mathbb{Z}}, \mathcal{B}^{\dagger}) $ 
$$\varphi (B) = \sum_{j \in \mathbb{Z}} B_j(M_t)T_j ~, $$
onde $(B_j(M_t)u)(\tau)=B_j(\tau)u(\tau)$ e $(T_j u)(\tau)=u(\tau - j)$ para $u$ em $L^2(\mathbb{R}, L^2(\mathbb{S}^1))$. 

\begin{itemize}
\item \textit{Afirmação:} $\varphi$ é um $\ast$-homomorfismo contínuo. 
	
	Sejam $B=(B_j)_j$ e $D=(D_j)_j$ seqüências em $K(\mathbb{Z}, \mathcal{B}^{\dagger})$. 
	\begin{itemize}
	\item Produto: 
		\begin{eqnarray*}
		\varphi(B)\cdot \varphi(D)  & = & \sum_j \sum_k B_j(M_t)T_jD_k(M_t)T_k = \sum_j \sum_k B_j(M_t)D_k(M_t-j)T_{j+k} \\
															  & = & \sum_l \sum_k B_{l-k}(M_t)D_k(M_t-(l-k))T_{l} = \sum_l \sum_n B_{n}(M_t)D_{l-n}(M_t-n)T_{l} \\
															  & = & \sum_l \left(\sum_n B_{n}(M_t)\alpha^{n}(D_{l-n}(M_t))\right)T_{l} = \varphi(B\cdot D) 
		\end{eqnarray*}			
	\item Involução: 
		\begin{eqnarray*}
		\varphi(B)^{*}  & = & \sum_j T_{-j}B_j^{*}(M_t) = \sum_{j'} T_{j'}B_{-j'}^{*}(M_t) = \sum_{j'} B_{-j'}^{*}(M_t-j')T_{j'} \\
		                & = & \sum_{j'} \alpha^{j'}(B_{-j'}^{*})T_{j'} = \varphi(B^{*})  
		\end{eqnarray*}	
  \item Continuidade:
    \begin{eqnarray*}
		||\varphi(B)||  &   =  & || \sum_j B_{j}(M_t)T_{j} ||  \\
		                & \leq & \sum_{j} ||B_{j}(M_t)|| || T_{j} || = \sum_{j} ||B_{j}(M_t)|| = ||B||_{1}
		\end{eqnarray*}	
	\end{itemize}
\end{itemize}

Como $\varphi$ é contínua, $ K({\mathbb{Z}}, \mathcal{B}^{\dagger}) $ é denso em $\ell^1(\mathbb{Z}, \mathcal{B}^{\dagger})$ e $\{~\sum_{j \in F} B_j(M_t)T_j ~;~ F \subset \mathbb{Z}~\mbox{finito}, \\ B_j \in \mathcal{B}^{\dagger}\}$ é denso em $\mathcal{B}^{\diamond},$ temos que existe um $\ast$-homomorfismo $\tilde{\varphi}:\ell^1(\mathbb{Z}, \mathcal{B}^{\dagger})\rightarrow\mathcal{B}^{\diamond},$ extensão contínua de $\varphi$. Segue de \cite{dixmier}, 2.7.4, que existe um $\ast$-homomorfismo contínuo (ao qual chamaremos de $\varphi$ novamente)
$$ 
\varphi : {\mathcal{B}^{\dagger}}\times_{\alpha}\mathbb{Z}\rightarrow\mathcal{B}^{\diamond}
$$
estendendo $\tilde{\varphi}.$ Vamos verificar que $\varphi$ é um isomorfismo.

Como $Im \varphi$ é fechada e contém o subconjunto $\{\sum_{j\in F} B_j(M_t) T_j; F\subset \mathbb{Z} \mbox{ finito }, ~ B_j \in \mathcal{B}^{\dagger}\},$ que é denso em $\mathcal{B}^{\diamond},$ então $\varphi$ é sobrejetora. Para mostrarmos a injetividade, usaremos a proposição \ref{exel}. 

Dados $z$ em $S^1$ e $B=(B_j)$ em $\ell^1(\mathbb{Z}, \mathcal{B}^{\dagger})$, defina $\mu_{z}(B)=(z^{j}B_j)_j.$ Novamente por \cite{dixmier}, 2.7.4, podemos estender $\mu_z$ a um automorfismo de ${\mathcal{B}^{\dagger}}\times_{\alpha}\mathbb{Z}$ e obtemos então a ação $\mu: z\mapsto\mu_z$ do círculo em ${\mathcal{B}^{\dagger}}\times_{\alpha}\mathbb{Z}.$


Para $z$ em $S^1$, definimos também o operador unitário em $\mathcal{L}_{\Omega}$:  
$$(U_z v)(t, x)=z^t v(t,x), ~v\in L^2(\Omega),$$
e obtemos uma ação $\beta$ de $S^1$ em $\mathcal{B}^{\diamond}$, dada por 
$$\beta_z(B(M_t))=U_z B(M_t) U_z^{-1}~, ~~ B(M_t) ~\in ~\mathcal{B}^{\diamond}.$$ 

Temos que verificar que $\varphi$ é covariante. Para isso, é suficiente mostrar que para cada $z \in S^1$, $\beta_z(\varphi(B))=\varphi(\mu_z(B))$ para $B=(B_j)_j \in \ell^1(\mathbb{Z}, \mathcal{B}^{\dagger}).$  Temos que o conjunto $\{B\delta_0, \delta_1; B \in \mathcal{B}^{\dagger}\}$ gera $\ell^1(\mathbb{Z}, \mathcal{B}^{\dagger}),$ onde $\delta_n$ é a seqüência tal que o único elemento não nulo é 1 na posição $n$. Assim, é suficiente mostrar que vale a igualdade apenas nos geradores.  
\begin{eqnarray*}  
B\delta_0 & \longrightarrow & \mu_z(B\delta_0) = z^0 B\delta_0 = B\delta_0 ~~\mbox{e}~~ \varphi(B\delta_0) = B(M_t) ~\Longrightarrow ~ \varphi(\mu_z(B\delta_0)) = B(M_t) \\
& \longrightarrow &  \beta_z(\varphi(B\delta_0)) = \beta_z(B(M_t)) = U_z B(M_t) U_z^{-1} = B(M_t) \\
\delta_1 & \longrightarrow & \mu_z(\delta_1) = z\delta_1 ~\mbox{e}~ \varphi(z\delta_1) = z T_1 ~\Longrightarrow ~ \varphi (\mu_z(z\delta_1)) = z T_1 \\
& \longrightarrow & \varphi(\delta_1) = T_1 ~~\mbox{e}~~ \beta_z(T_1) = U_z T_1 U_z^{-1} = z T_1 ~\Longrightarrow ~\beta_z(\varphi (\delta_1)) = z T_1
\end{eqnarray*}

Para mostrar que $\varphi$ é injetiva, é suficiente mostrar que sua restrição aos pontos fixos de $\mu$ é injetiva. 

Seja $\mathcal{F}$ a subálgebra dos pontos fixos por $\mu$, 
$$
\mathcal{F} = \{x\in \mathcal{B}^{\dagger}\times_{\alpha}\mathbb{Z} ; \mu_z(x) = x ~\forall z\in S^1 \},
$$ 
e seja $E:\mathcal{B}^{\dagger}\times_{\alpha}\mathbb{Z}\rightarrow\mathcal{B}^{\dagger}\times_{\alpha}\mathbb{Z}$ definida por 
$$
E(x) = \int_{S^1} \mu_z(x) dz . 
$$
Se $x\in \mathcal{F}$, então $E(x) = x$. Assim, $\mathcal{F}$ está contida na imagem de $E$. Como para todo $j\neq 0,$ a integral $\int_{S^1} z^j dz=0,$ temos $E(a) \in \mathcal{B}^{\dagger}$ para todo $a \in \ell^1(\mathbb{Z}, \mathcal{B}^{\dagger}).$

Como $\ell^1(\mathbb{Z},\mathcal{B}^{\dagger})$ é denso em $\mathcal{B}^{\dagger}\times_{\alpha}\mathbb{Z}$ e $\mathcal{B}^{\dagger}$ é uma subálgebra fechada de $\mathcal{B}^{\dagger}\times_{\alpha}\mathbb{Z}$, temos que $Im E\subseteq \mathcal{B}^{\dagger}$. Como $\mathcal{B}^{\dagger}\subseteq \mathcal{F},$ então $ \mathcal{B}^{\dagger} = \mathcal{F}$. Conhecendo $\mathcal{F}$ e sabendo que $\varphi(B\delta_0)=B(M_t)$, temos
$$\varphi(B\delta_0)= \varphi(D\delta_0)~ \Longleftrightarrow ~ B(M_t) = D(M_t) ~ \Longleftrightarrow ~ B = D .$$

\end{proof}

\subsection{Ideal comutador e espaço símbolo de ${\mathcal{A}}^{\diamond}$}

Apesar de já termos uma boa descrição da álgebra ${\mathcal{A}}^{\diamond}$, queremos também determinar seu ideal comutador e seu espaço símbolo. 

A descrição do ideal comutador de ${\mathcal{A}}^{\diamond}$ é muito parecida com a de $\mathcal{E}_{\mathcal{A}},$ o ideal comutador da álgebra $\mathcal{A},$ por isso pode-se notar que os cálculos usados aqui serão praticamente os mesmos.     

Denote por $\mathcal{E}^{\diamond}_{\mathcal{B}}$ o ideal comutador de ${\mathcal{B}}^{\diamond}$.  

\begin{prop} O ideal $\mathcal{E}^{\diamond}_{\mathcal{B}}$ coincide com o fecho do conjunto
$$ \mathcal{E}_0^{\diamond} ~=~ \{ \sum_{j=-N}^{N} B_j(M_t)T_j ; ~ B_j \in C_0(\mathbb{R}, \mathcal{K}_{\mathbb{S}^1}), N \in \mathbb{N} \}. $$
\end{prop}

\begin{proof} Primeiro vamos verificar que os comutadores dos geradores de ${\mathcal{B}}^{\diamond}$, e seus adjuntos, estão em $ \mathcal{E}_0^{\diamond}$. É fácil ver que   
$$ 
[a(M_x), T_j] = 0 ~,~~ [B_k(M_t) , T_j] = (B_k(M_t +j)- B_k(M_t)) T_j ,
$$
para $j \in  \mathbb{Z}$ e $k = 4, 5, 6$. Em \cite{melo}, temos na demonstração da proposição 1.1 que
$B_k(\cdot +j)-B_k(\cdot) \in C_0(\mathbb{R}, \mathcal{K}_{\mathbb{S}^1})$ para $ k = 4, 5, 6.$ 
Assim, o operador $B_k(M_t +j)- B_k(M_t)$ é a multiplicação pela função $B_k(\cdot +j)-B_k(\cdot)\in C_0(\mathbb{R}, \mathcal{K}_{\mathbb{S}^1})$. Temos ainda na mesma proposição que 
$$
[a(M_x), B_k] ,~ [B_i, B_k] \in C_0(\mathbb{R}, \mathcal{K}_{{\mathbb{S}^1}}), ~i, k = 4, 5, 6,
$$
logo, os comutadores dos respectivos operadores são multiplicações por funções em $C_0(\mathbb{R}, \mathcal{K}_{{\mathbb{S}^1}})$ e o mesmo vale para os adjuntos desses operadores. Portanto, $\mathcal{E}^{\diamond}_{\mathcal{B}}$ está contido no fecho de $\mathcal{E}_0^{\diamond}.$

Sabemos que $C_0(\mathbb{R}, \mathcal{K}_{{\mathbb{S}^1}}) = \mathcal{E}^{\dagger}_{\mathcal{B}}$ está contido em $\mathcal{E}^{\diamond}_{\mathcal{B}}$. Então para todo $j$ inteiro, $K(M_t)T_j \in \mathcal{E}^{\diamond}_{\mathcal{B}}$, para $K$ em $C_0(\mathbb{R}, \mathcal{K}_{{\mathbb{S}^1}})$ (note que $K(M_t)$ é um elemento do ideal $\mathcal{E}^{\diamond}_{\mathcal{B}}$). Assim, o fecho de  $\mathcal{E}_0^{\diamond}$ está em $\mathcal{E}^{\diamond}_{\mathcal{B}}$.  
\end{proof}

Podemos ainda dizer que $\mathcal{E}^{\diamond}_{\mathcal{B}}$ é o fecho de 
$$ \{ \sum_{j=-N}^{N} a_j(M_t)T_j ; ~ a_j \in C_0(\mathbb{R}), N \in \mathbb{N}\}\otimes \mathcal{K}_{\mathbb{S}^1}. $$

\begin{prop} Sendo $ W $ a aplicação definida em (\ref{w}), temos que 
$$ W \mathcal{F} W^{-1} = C(S^1)\otimes \mathcal{K}_{\mathbb{Z}} $$
onde $\mathcal{F}$ é o fecho de 
$$\mathcal{F}^0 = \{ \sum_{j=-N}^{N} a_j(M_t)T_j ; ~ a_j \in C_0(\mathbb{R}), N \in \mathbb{N}\}.$$
\end{prop}

\begin{proof} Para $j$ inteiro e $\varphi$ real, temos
$$ Y_{\varphi} a(\varphi - M_k) Y_{-j} Y_{-\varphi} (Y_{\varphi} u^{\diamond}(\varphi)) = 
Y_{\varphi} (a(\varphi - k) u(\varphi - k +  j))_{k\in\mathbb{Z}}. $$
Como $W a (M_t) T_j W^{-1} (Wu)(\varphi) = Y_{\varphi} (a(\varphi - k) u(\varphi - k +  j))_{k\in\mathbb{Z}} $, isto implica que 
$$ W a (M_t) T_j W^{-1} = Y_{\varphi} a(\varphi - M_k) Y_{-j-\varphi} $$
para $a$ em $C_0(\mathbb{R}).$

Usando que $ \lim_{k\rightarrow\pm\infty} a(\varphi - k) = 0 $ para todo $\varphi$ real, e que 
$$ Y_{\varphi +1 } a(\varphi +1 - M_k) Y_{-(\varphi +1)- j} = Y_{\varphi} a(\varphi - M_k) Y_{-\varphi - j}, $$
segue que $ Y_{\varphi} a(\varphi - M_k) Y_{-\varphi - j} $ é uma função contínua em $S^1 = \{e^{i2\pi\varphi}; \varphi \in \mathbb{R}\}$ com valores em operadores compactos. 

Como o mergulho de $C(S^1, \mathcal{K}_{\mathbb{Z}})$ em $\mathcal{L}_{S^1\times\mathbb{Z}}$ é uma isometria e $C(S^1, \mathcal{K}_{\mathbb{Z}}) = C(S^1)\otimes\mathcal{K}_{\mathbb{Z}},$ temos que
$$W\mathcal{F}W^{-1} \subseteq C(S^1)\otimes\mathcal{K}_{\mathbb{Z}}. $$

A igualdade sai por argumentos análogos usados na demonstração do teorema \ref{edagger}, que se resume a provar que para cada $e^{2\pi i\varphi}\in S^1$, $\{WGW^{-1}(\varphi); G\in \mathcal{F}\} = \mathcal{K}_{\mathbb{Z}}$ e que $W\mathcal{F}W^{-1}$ separa os estados puros de $C(S^1)\otimes\mathcal{K}_{\mathbb{Z}}$.
\end{proof}


A próxima proposição nos dá a descrição que queríamos para o ideal comutador $\mathcal{E}^{\diamond}$ da álgebra $\mathcal{A}^{\diamond}$.

\begin{prop}\label{gamma'} $\mathcal{E}^{\diamond}$ e  $C(S^1, \mathcal{K}_{\mathbb{Z}\times\mathbb{S}^1})$ são isomorfos e este isomorfismo é dado pela conjugação por $FW^{-1}$:
$$(WF^{-1}\otimes I_{\mathbb{S}^1}) \mathcal{E}^{\diamond} (FW^{-1}\otimes I_{\mathbb{S}^1}) = C(S^1, \mathcal{K}_{\mathbb{Z}})\otimes\mathcal{K}_{\mathbb{S}^1},$$
onde $W$ é dada em (\ref{w}) e $F$ é a transformada de Fourier. 
\end{prop}
Considere o homomorfismo 
$$ \gamma ': \mathcal{A}^{\diamond} \rightarrow C(S^1, \mathcal{L}_{\mathbb{Z}\times\mathbb{S}^1})$$
dado por $\gamma_{A}'(z)=\gamma_{A}(z,\pm 1)$ para todo $A\in\mathcal{A}^{\diamond}.$ Então $\gamma '$ restrito a $\mathcal{E}^{\diamond}$ tem imagem igual a  $C(S^1, \mathcal{K}_{\mathbb{Z}\times\mathbb{S}^1})$. 

\begin{prop}\label{mdiamond} O espaço símbolo ${\bf M}^{\diamond}$ de $\mathcal{A}^{\diamond}$ é homeomorfo a $ S^1\times \mathbb{S}^1 \times S^1_{\infty}$ e a aplicação do símbolo $\sigma^{\diamond} :  \mathcal{A}^{\diamond} \longrightarrow C( S^1 \times \mathbb{S}^1 \times S^1_{\infty}) $
é dada por 
$$\sigma^{\diamond}_{A}(e^{i\theta}, x, w) = \sigma^{\dagger}_{A}(x,w), ~\forall ~ A \in \cal{A}^{\dagger}$$ 
e 
$$\sigma^{\diamond}_{e^{ijM_{\theta}}}(e^{i\theta}, x, w) = e^{ij\theta}. $$
\end{prop}

\begin{proof} Considere as inclusões
$$
\begin{array}{cccccccccc}
 \kappa_1\ :  & \! P_{2\pi} & \! \longrightarrow &
 \! \displaystyle\frac{\cal{A}^{\diamond}}{\mathcal{E}^{\diamond}} & \! \ \ \  & \!
 \kappa_2\ :  & \! \displaystyle\frac{\cal{A}^{\dagger}}{\mathcal{E}^{\dagger}} & \! \longrightarrow &
 \!\displaystyle\frac{\cal{A}^{\diamond}}{\mathcal{E}^{\diamond}} \\     
 & \!  p_j & \! \longmapsto & \! [p_j(M_t)]_{\mathcal{E}^{\diamond}} & \!  & \!  & \!  [A]_{\mathcal{E}^{\dagger}} & \! \longmapsto & \!    
 [A]_{\mathcal{E}^{\diamond}}
\end{array}
$$
onde $P_{2\pi}$ é o conjunto das funções contínuas $2\pi$-periódicas.

Sabendo que $S^1$ ~e~ $\mathbb{S}^1\times S^1_{\infty}$ são os espaços
símbolo de $P_{2\pi}$ e $\mathcal{A}^{\dagger}$, respectivamente, temos a aplicação dual 
$$ 
\begin{array}{cccc}
 {\kappa}^{*} \ :  & \! {\bf M}^{\diamond} & \! \longrightarrow &
 \! S^1 \times \mathbb{S}^1 \times S^1_{\infty} \\
 & \! \omega & \! \longmapsto & \! (\omega \circ \kappa_1, \omega \circ \kappa_2) 
\end{array}
$$
onde ${\kappa}^{*}$ será um homeomorfismo sobre a sua imagem. 
 
A idéia agora é fazer o mesmo para a álgebra $\cal{A}$, apesar de já conhecermos o seu espaço símbolo. 
Sejam 
$$
\begin{array}{cccccccccc}
 j_1\ :  & \! C([-\infty, +\infty]) & \! \longrightarrow &
 \! \displaystyle\frac{\cal{A}}{\mathcal{E}_{\mathcal{A}}} & \! \ \ \  & \!
 j_2\ :  & \! \displaystyle\frac{\cal{A}^{\diamond}}{\mathcal{E}^{\diamond}} & \! \longrightarrow &
 \!\displaystyle\frac{\cal{A}}{\mathcal{E}_{\mathcal{A}}} \\     
 & \!  b & \! \longmapsto & \! [b(M_t)]_{\mathcal{E}_{\mathcal{A}}} & \!  & \!  & \!  [A]_{\mathcal{E}^{\diamond}} & \! \longmapsto & \!    
 [A]_{\mathcal{E}_{\mathcal{A}}}
\end{array}
$$
as inclusões canônicas. A aplicação dual é dada por 
$$ 
\begin{array}{ccccc}
 {j}^{*} \ :  & \! {\bf M}_{\mathcal{A}} & \! \longrightarrow &
 \! [-\infty, +\infty] \times Im {\kappa}^{*} & \! \subseteq [-\infty, +\infty]\times S^1 \times \mathbb{S}^1 \times S^1_{\infty} \\
 & \! \omega & \! \longmapsto & \! (\omega \circ j_1, \omega \circ j_2) & \!
\end{array}
$$
No teorema \ref{ma}, temos uma descrição de ${\bf M}_{\mathcal{A}}$,  que também pode ser escrito como $X\times\mathbb{S}^1 \times S^1$, onde o conjunto $X=\{ (t,e^{i\theta})\in [-\infty, +\infty]\times S^1; \theta = t \mbox{ se } |t|<\infty\}$. Temos ainda que ${\bf M}_{\mathcal{A}}$ é homeomorfo a $Im j^{*}$. 

Dado o funcional linear multiplicativo $\omega \in {\bf M}_{\mathcal{A}}$ correspondente ao ponto $(t, e^{i\theta}, x, (\tau,\xi))$ $\in X\times\mathbb{S}^1 \times S^1$, queremos mostrar que $\omega \circ j_2$ é o funcional linear $\lambda_{(e^{i\theta}, x, (\tau,\xi))} \in {\bf M}^{\diamond}$ associado ao ponto $(e^{i\theta}, x, (\tau,\xi)).$  

Para isto, basta mostrar que $\omega \circ j_2 ([A_i]_{\mathcal{E}_{\mathcal{A}}})= \lambda_{(e^{i\theta}, x, (\tau,\xi))}([A_i]_{\mathcal{E}^{\diamond}})$ onde $A_i, ~i=1,3,4,5,6,$ são os geradores de $\mathcal{A}^{\diamond}.$
Sabemos que $\omega\circ j_2([A_i]_{\mathcal{E}_{\mathcal{A}}})= \omega([A_i]_{\mathcal{E}_{\mathcal{A}}}), ~i=1,3,4,5,6$:  
\begin{enumerate}
\item $ \omega([a(M_x)]_{\mathcal{E}_{\mathcal{A}}}) = \sigma_{a(M_x)}(t, e^{i\theta}, x, (\tau,\xi)) =  a(x)$
\item $ \omega([e^{ijM_{\theta}}]_{\mathcal{E}_{\mathcal{A}}}) = \sigma_{e^{ijM_{\theta}}}(t, e^{i\theta}, x, (\tau,\xi)) =  e^{ij{\theta}}$
\item $ \omega([A_4]_{\mathcal{E}_{\mathcal{A}}}) = \sigma_{A_4}(t, e^{i\theta}, x, (\tau,\xi)) = 0 $
\item $ \omega([A_5]_{\mathcal{E}_{\mathcal{A}}}) = \sigma_{A_5}(t, e^{i\theta}, x, (\tau,\xi)) = \tau $
\item $ \omega([A_6]_{\mathcal{E}_{\mathcal{A}}}) = \sigma_{A_6}(t, e^{i\theta}, x, (\tau,\xi)) = \xi $
\end{enumerate}
Para calcularmos $\lambda_{(e^{i\theta}, x, (\tau,\xi))}([A_i]_{\mathcal{E}^{\diamond}})$, temos de voltar para a aplicação dual 
$$ 
\begin{array}{cccc}
 {\kappa}^{*} \ :  & \! {\bf M}^{\diamond} & \! \longrightarrow &
 \! S^1 \times \mathbb{S}^1 \times S^1_{\infty} \\
 & \! \lambda & \! \longmapsto & \! (\lambda \circ \kappa_1, \lambda \circ \kappa_2 \circ i_1, \lambda \circ \kappa_2 \circ i_2 ) 
\end{array}
$$
lembrando que $i_1$ e $i_2$ vêm da demonstração do espaço símbolo de ${\bf M}^{\dagger}$.

Para $e^{ijM_{\theta}}$, o operador de multiplicação pela função contínua $p_j(\theta)= e^{ij\theta}$ de período $2\pi$, $\lambda_{(e^{i\theta}, x, (\tau,\xi))}([e^{ijM_{\theta}}]_{\mathcal{E}^{\diamond}})$ é igual ao funcional $\lambda \circ \kappa_1 $ calculado em $p_j$. Como os funcionais lineares multiplicativos de $P_{2\pi}$ são dados pela evaluação num ponto de $S^1$, então vamos associar ao funcional  $\lambda \circ \kappa_1$ o ponto $e^{i\theta}.$ Logo, 
$$\lambda_{(e^{i\theta}, x, (\tau,\xi))}([e^{ijM_{\theta}}]_{\mathcal{E}^{\diamond}}) = \lambda \circ \kappa_1(p_j)= p_j(\theta) = e^{ij\theta} .$$

Para calcularmos $\lambda_{(e^{i\theta}, x, (\tau,\xi))}([a(M_x)]_{\mathcal{E}^{\diamond}})$, para $a\in C^{\infty}(\mathbb{S}^1)$, temos que 
$$\lambda_{(e^{i\theta}, x, (\tau,\xi))}([a(M_x)]_{\mathcal{E}^{\diamond}}) = \lambda \circ \kappa_2 ([a(M_x)]_{\mathcal{E}^{\dagger}}) = \lambda \circ \kappa_2 \circ i_1(a)$$
Também neste caso, os  funcionais lineares multiplicativos de $C^{\infty}(\mathbb{S}^1)$ são dados pela evaluação num ponto de $\mathbb{S}^1$, então  associamos a $\lambda \circ \kappa_1 \circ i_1 $ o ponto $x.$ Logo,
$$\lambda_{(e^{i\theta}, x, (\tau,\xi))}([a(M_x)]_{\mathcal{E}^{\diamond}}) =a(x) .$$

Como $\sigma^{\dagger}_{A_4} = 0$ e $\lambda \circ \kappa_2$ é um funcional linear em $\textbf{M}^{\dagger}$ então 
$$\lambda_{(e^{i\theta}, x, (\tau,\xi))}([A_4]_{\mathcal{E}^{\diamond}})= \lambda \circ \kappa_2 ([A_4]_{\mathcal{E}^{\dagger}}) =  0 .$$
Temos ainda que $\lambda \circ \kappa_2 \circ i_2$ é um funcional em $\mathcal{C}$ e portanto, a ele está associado um par $(\tau, \xi) \in S_{\infty}^1$ pelo teorema \ref{mdagger},  e portanto
$$\lambda \circ \kappa_2 \circ i_2 (A_5) = B_5'(\tau, \xi) =\tau ~, ~~~~~\lambda \circ \kappa_2 \circ i_2 (A_6) = B_6'(\tau, \xi) =\xi .$$

Logo, dado $\omega$ fica bem definido o funcional $\lambda_{(e^{i\theta}, x, (\tau,\xi))} $ em ${\bf M}^{\diamond} $ para todo $(e^{i\theta}, x, (\tau,\xi))\in S^1 \times \mathbb{S}^1 \times S^1_{\infty}.$ Portanto, a imagem de $\kappa^{*}$ é igual a $S^1 \times \mathbb{S}^1 \times S^1_{\infty}$ e 
$${\sigma^{\diamond}}_{a(M_x)}(e^{i\theta}, x, (\tau, \xi)) = a(x), ~~ {\sigma^{\diamond}}_{e^{ijM_{\theta}}}(e^{i\theta}, x, (\tau, \xi)) = e^{ij\theta} ,$$
$$\sigma^{\diamond}_{A_4}(e^{i\theta}, x, (\tau, \xi)) = 0, ~~\sigma^{\diamond}_{A_5}(e^{i\theta}, x, (\tau, \xi)) = \tau , ~~
\sigma^{\diamond}_{A_6}(e^{i\theta}, x, (\tau, \xi)) = \xi .$$

\end{proof}







\chapter{A K-Teoria de $\mathcal{A}$ e suas C*-subálgebras}

Este último capítulo consiste no cálculo da K-teoria da C*-álgebra $\mathcal{A}\subset \mathcal{L}(L^2(\Omega)),$ onde $\Omega = \mathbb{R}\times \mathbb{S}^1$. 

Apresentaremos primeiro alguns conceitos básicos sobre esta teoria e resultados importantes  que serão úteis nas próximas secões. 
Recomendamos \cite{rordam} e \cite{wegge} como literatura neste assunto. 

Começaremos calculando a K-teoria das subálgebras $\mathcal{A}^{\dagger}$ e $\mathcal{A}^{\diamond}.$ Para isso será fundamental a\-na\-li\-sar a seqüência de Pimsner-Voiculescu construída a partir da realização de $\mathcal{A}^\diamond$ como um produto cruzado. 

Conhecer a K-teoria de $\mathcal{A}^{\diamond}$ nos fornecerá uma das informações necessárias para conseguirmos determinar a K-teoria de $\mathcal{A}.$ Um dos resultados mais interessantes deste último capítulo, seção 3.3, é o cálculo do índice de certos operadores de Fredholm, usando a fórmula do índice de Fedosov, uma particularização da fórmula de Atiyah-Singer. Este cálculo é necessário para determinar a aplicação do índice de uma dada seqüência exata de seis termos em K-teoria induzida por uma seqüência exata curta envolvendo $\mathcal{A}^{\diamond}$. 

\section{Conceitos básicos}

Seja $A$ uma C*-álgebra. Denotaremos por $\tilde{A}$ a unitização de $A$. Dado um $n$ natural, seja $M_n(\tilde{A})$ o conjuntos das matrizes $n\times n $ com entradas em $\tilde{A}$. Considerando $P_n(\tilde{A}):=\{ \rho \in M_{n}(\tilde{A}); \rho = {\rho}^2 = {\rho}^{\ast} \}$ o conjunto das projeções em $\tilde{A}$ e $U_n(\tilde{A}):=\{ u \in M_{n}(\tilde{A}); uu^{\ast}=u^{\ast}u=1 \}$ o conjunto dos elementos unitários em $\tilde{A}$, usaremos a seguinte notação: 
$$P_{\infty}(\tilde{A}) := \bigcup_{n\in \mathbb{N}} P_{n}(\tilde{A})~, ~~ U_{\infty}(\tilde{A}) := \bigcup_{n\in \mathbb{N}} U_{n}(\tilde{A})$$

De uma forma resumida, podemos definir a K-teoria de C*-álgebra pelos funtores $K_0$ e $K_1$ que associam a uma C*-álgebra $A$ dois grupos abelianos $K_0(A)$ e $K_1(A)$. O grupo $K_0(A)$ pode ver visto como diferenças formais $[p]-[s(p)]$ de classes de equivalência de projeções $p, s(p) \in P_{\infty}(\tilde{A})$, onde $s(p)$ é a parte escalar de $p$. $K_1(A)$ é o grupo de classes de homotopias de unitários (ou inversíveis) em $U_{\infty}(\tilde{A}).$ 

Uma propriedade desses funtores é que dada uma seqüência exata curta de C*-álgebras,
$$ 0 \longrightarrow A \stackrel{i}{\longrightarrow} B \stackrel{\pi}{\longrightarrow} C \longrightarrow 0,$$
associamos a ela uma seqüência exata de seis termos com os seus K-grupos
$$
\begin{array}{ccccc}
 K_0(A)  & \! \stackrel{i_*}{\longrightarrow} & \!  K_0(B)  & \! \stackrel{{\pi}_*}{\longrightarrow} & \!  K_0(C)    \\ \\    
 {\delta}_1 \ \uparrow & \! ~  & \! ~  & \! ~ & \! \downarrow \ {\delta}_0  \\ \\
 
 K_1(C) & \! \stackrel{{\pi}_*}{\longleftarrow} & \!  K_1(B)  & \! \stackrel{{i}_*}{\longleftarrow} & \!  K_1(A)
 \end{array}
$$
onde $i_*$ e $\pi_*$ são as aplicações induzidas funtorialmente, $\delta_1: K_1(C)\rightarrow K_0(A)$ é dita aplicação do índice e $\delta_0: K_0(C)\rightarrow K_1(A)$ é a aplicação exponencial. 

\begin{definicao} {\rm (\cite{wegge}, Definição 7.2.1)} A {\bf suspensão} de uma C*-álgebra $A$ é a C*-álgebra
\begin{eqnarray*}
SA & := & A\otimes C_0(\mathbb{R}) \\
& \cong & C_0(\mathbb{R}, A) \\
& \cong & C_0((0,1), A) \\
& \cong & \{f\in C(S^1, A); f(1)=0\}.
\end{eqnarray*}
Nos três últimos objetos, soma e produto são pontuais, adjunto é dado pelo complexo conjugado e a norma é a norma do supremo. O produto tensorial $A\otimes C_0(\mathbb{R})$ significa o produto tensorial entre C*-álgebras completado. 
\end{definicao}

A seguir, enunciaremos um isomorfismo em K-teoria que envolve a álgebra e sua suspensão.  

\begin{teo}\label{theta} Os grupos $K_1(A)$ e $K_0(SA)$ são isomorfos para toda C*-álgebra $A.$ O isomorfismo, ao qual denotaremos por $\theta_{A},$ tem a seguinte descrição: dado um unitário $u \in U_n(\tilde{A})$, existem unitários $w_t \in U_{2n}(\tilde{A}),~t\in [0,1],$ que forma uma homotopia entre $w_0= 1_{2n}$ e
$$w_1= \left(\begin{array}{cc}
u & 0 \\
0 & u^* \\
\end{array}\right)
=: diag(u,u^{\ast}) .$$
Isto dá origem a um caminho de projeções
$$q_t := w_t \cdot diag(1_n, 0) \cdot w_t^{\ast} \in M_{2n}(\tilde{A}), ~t\in [0,1],$$
com $q_0 = diag(1_n,0)$ e $q_1 = diag(u,u^{\ast})\cdot diag(1_n,0) \cdot diag(u^{\ast},u) = diag(1_n,0)$. A aplicação $q: t\mapsto q_t$  pertence a $M_{2n}(\tilde{SA})$ e $q_t - diag(1_n,0) \in M_{2n}(A)$ para todo $t\in [0,1]$. Temos então que 
$$\theta_{A}([u]_1) := [q]_0 - [1_n]_0 . $$
\end{teo}

\begin{obs}\label{obs}{\rm Para $u\in \tilde{A}$ unitário, basta tomar $w_t:=diag(u,1)\cdot u_t \cdot diag(u^{\ast},1)\cdot u_t^{\ast} \in M_2(\tilde{A}),$ (\cite{wegge}, Teorema 7.2.5, Teorema 4.2.9), onde }
$$u_t:= 
\left(
\begin{array}{cc}
\cos{\frac{\pi t}{2}} & -\mbox{sen}{\frac{\pi t}{2}} \\
\mbox{sen}{\frac{\pi t}{2}} &  \cos{\frac{\pi t}{2}} \\
\end{array}
\right).
$$
\end{obs}

Neste trabalho, vamos caracterizar a suspensão pelas funções definidas no círculo por ser mais conveniente para os nossos cálculos. Sendo assim, para $u\in \tilde{A}$ unitário,  a função $q \in P_2(S\tilde{A})$ definida acima em função de $t\in [0,1]$ deve ser uma função com domínio em $S^1$. Vamos então descrevê-la em função desta nova variável.
\begin{eqnarray*}
t\in [0,1] \longmapsto q_t & := & w_t \cdot diag(1, 0) \cdot w_t^{\ast} \\
& = & 
\left(\begin{array}{cc}
u & 0 \\
0 & 1 \\
\end{array}\right)
\cdot u_t \cdot 
\left( \begin{array}{cc}
u^{\ast} & 0 \\
0        & 1 \\
\end{array}\right)
\cdot u_t^{\ast} \cdot 
\left(\begin{array}{cc}
1 & 0 \\
0 & 0 \\
\end{array}\right)
\cdot u_t \cdot 
\left(\begin{array}{cc}
u & 0 \\
0 & 1 \\
\end{array}\right)
\cdot u_t^{\ast} \cdot 
\left(\begin{array}{cc}
u^{\ast} & 0 \\
0        & 1 \\
\end{array}\right) 
\end{eqnarray*}
Fazendo estas contas, obtemos uma matriz em função de $u, \cos{\frac{\pi t}{2}}$ e $\mbox{sen}{\frac{\pi t}{2}}:$
$$ 
\left(\begin{array}{cc}
1 - (2-u-u^{\ast})\cos^2{\frac{\pi t}{2}}\mbox{sen}^2{\frac{\pi t}{2}}  & 
(u-1)\cos^3{\frac{\pi t}{2}}\mbox{sen}{\frac{\pi t}{2}}+u(u-1)\cos{\frac{\pi t}{2}}\mbox{sen}^3{\frac{\pi t}{2}} \vspace{0.4cm}\\ 
(u^{\ast}-1)\cos^3{\frac{\pi t}{2}}\mbox{sen}{\frac{\pi t}{2}}+u^{\ast}(u^{\ast}-1)\cos{\frac{\pi t}{2}}\mbox{sen}^3{\frac{\pi t}{2}} &
(2-u-u^{\ast})\cos^2{\frac{\pi t}{2}}\mbox{sen}^2{\frac{\pi t}{2}} \\
\end{array}\right)
$$

Seja $z=z_1+iz_2=e^{2\pi it}\in {S}^1$, $z_1 = \cos 2\pi t$ e $z_2 = \mbox{sen} 2\pi t .$ Usando as igualdades trigonométricas, $\cos(\pi t+ \pi t)$ e $\mbox{sen}(\pi t+ \pi t)$, e em seguida, $\cos(\pi t /2 + \pi t/2)$ e $\mbox{sen}(\pi t /2 + \pi t/2)$, obtemos as relações
$$\cos^2(\frac{\pi t}{2})\mbox{sen}^2(\frac{\pi t}{2}) = \frac{1-z_1}{8}$$
$$\cos({\frac{\pi t}{2}})\mbox{sen}^3({\frac{\pi t}{2}}) = \frac{1}{2}\sqrt{\frac{1-z_1}{8}} - \frac{z_2}{8}~,~~
\cos^3({\frac{\pi t}{2}})\mbox{sen}({\frac{\pi t}{2}}) = \frac{1}{2}\sqrt{\frac{1-z_1}{8}} + \frac{z_2}{8} .$$
Substituindo estas igualdades na matriz anterior, temos que $q  \in P_2(\tilde{SA})$ é dada por
$$ z_1+iz_2 \in {S}^1 \mapsto q_{z_1+iz_2} =
\left(\begin{array}{cc}
1 - (2-u-u^{\ast})\frac{1-z_1}{8}  & 
\frac{u^2-1}{2}\sqrt{\frac{1-z_1}{8}} -(u-1)^2 \frac{z_2}{8} \vspace{0.3cm}\\
\frac{{u^{\ast}}^2-1}{2}\sqrt{\frac{1-z_1}{8}} -(u^{\ast}-1)^2 \frac{z_2}{8} &
(2-{u}-u^{\ast})\frac{1-z_1}{8}  \\
\end{array}\right)
$$

Denotemos por $Q(a,b,c)$, para $a$ unitário em $\tilde{A}$ e $b+ic$ em $S^1$, a seguinte matriz  
\begin{equation}
Q(a,b,c) = 
\left(\begin{array}{cc}
1 - (2-a-a^{\ast})\frac{1-b}{8}  & 
\frac{{a}^2-1}{2}\sqrt{\frac{1-b}{8}} -(a-1)^2 \frac{c}{8} \vspace{0.3cm}\\
\frac{{{a}^{\ast}}^{2}-1}{2}\sqrt{\frac{1-b}{8}} -(a^{\ast}-1)^2 \frac{c}{8} &
(2-a-a^{\ast})\frac{1-b}{8}  \\
\end{array}\right)
\label{matriz}
\end{equation}

Outro resultado importante em K-teoria é a {\bf Periodicidade de Bott}. Esta garante que $K_0(A)$ e $K_1(SA)$ são isomorfos e este isomorfismo é da seguinte forma. Para todo $n\in \mathbb{N}$ e para toda projeção $p\in P_n(\tilde{A})$, define-se uma aplicação $f_p : S^1\rightarrow U_n(\tilde{A})$ por 
$$f_p(z)=zp + (1_n - p).$$
Obtemos assim a {\bf aplicação de Boot}:
$$
\begin{array}{cccc}
 \beta_{A}: & \! K_0(A) & \! \rightarrow & \! K_1(SA) \\     
 & \! [p]_0 - [s(p)]_0& \! \mapsto & \! [f_pf^{\ast}_{s(p)}]_1
\end{array}
$$

Um resultado que usaremos constantemente nos cálculos de K-teoria é que, dada a seqüência exata curta cindida
\begin{center}
\begin{picture}(440,25)(0,0)
\put(85,0){$0$}\put(100,5){\vector(1,0){20}}
\put(130,0){$SA$}\put(160,5){\vector(1,0){20}}\put(168,10){${i}$}
\put(195,0){$C(S^1, A))$}\put(260,6){\vector(1,0){20}}\put(267,12){${p}$}
\put(295,0){$A$}\put(280,2){\vector(-1,0){20}}\put(267,-6){${s}$}
\put(315,5){\vector(1,0){20}} \put(340,0){0,}
\end{picture}
\end{center}
onde $i$ é a inclusão, $p$ é a evaluação no ponto 1 e $s$ é tal que $p\circ s = Id_{A}$, então vale que 
$$K_i(C(S^1, A))= i_{\ast}(K_i(SA))\oplus s_{\ast}(K_i(A)), ~i=0,1,$$
onde $i_{\ast}$ e $s_{\ast}$ são as aplicações induzidas.

Para o cálculo da aplicação do índice, temos o seguinte lema. 

\begin{lema}\label{lema1}{\rm (\cite{cintia}, Lema 1)} Seja $A$ uma álgebra de Banach unital e seja $\delta_1: K_1(A/{{J}})\rightarrow K_0({J})$ a aplicação do índice da seqüência exata de seis termos em K-teoria associada a seqüência exata curta $0\rightarrow {J} \rightarrow A \rightarrow A/{J} \rightarrow 0$. Se $u\in M_n({A}/{J})$ é um inversível, $\pi(a)=u$ e $\pi(b)=u^{-1},$ então 
$$
\delta_1([u]_1)= 
\left[
\left(\begin{array}{cc}
2ab-(ab)^2 & a(2-ba)(1-ba) \\
(1-ba)b & (1-ba)^2 \\
\end{array}\right)
\right]_0
-
\left[
\left(\begin{array}{cc}
1 & 0 \\
0 & 0 \\
\end{array}\right)
\right]_0
$$
\end{lema}

No caso em que $J$ é o ideal dos operadores compactos $K(H)$ de uma C*-algebra $A \subset \mathcal{L}(H),$ onde $H$ é um espaço de Hilbert,
temos que a aplicação do índice $\delta_1$ da seqüência abaixo
$$
\begin{array}{ccccc}
 \mathbb{Z} \cong K_0(\mathcal{K}) & \! {\longrightarrow} & \!  K_0(A)  & \! {\longrightarrow} & \!  K_0(A/{\mathcal{K}})   \\ \\    
 {\delta}_1 \ \uparrow & \! ~  & \! ~  & \! ~ & \! \downarrow \ {\delta}_0  \\ \\
 K_1(A/{\mathcal{K}}) & \! {\longleftarrow} & \!  K_1(A)  & \! {\longleftarrow} & \! K_1(\mathcal{K})=0
\end{array}
$$
é dada pelo índice de Fredholm, isto é, 
$$\delta_1([[T]_{\mathcal{K}}]_1) = \mathsf{ind}(T),$$
onde $\mathsf{ind}$ é o índice de Fredholm do operador $T$(\cite{rordam}, proposição 9.4.2).

\section{A K-teoria de $\mathcal{A}^{\dagger}$}

Começaremos os cálculos pela menor álgebra. Considere a seguinte seqüência exata curta:
\begin{equation}
0 \ \longrightarrow \ \mathcal{E}^{\dagger} \  \stackrel{i}{\longrightarrow} \ \mathcal{A}^{\dagger} \  
\stackrel{\pi}{\longrightarrow} \ \frac{\mathcal{A}^{\dagger}}{\mathcal{E}^{\dagger}} \ \longrightarrow \ 0, 
\label{secdagger}
\end{equation}
onde {\it i} é a inclusão e $\pi$ é a projeção canônica. 

O teorema \ref{edagger} nos dá o isomorfismo entre $\mathcal{E}^{\dagger}$ e $C_0(\mathbb{R}, \mathcal{K}_{\mathbb{S}^1}).$ Sabendo que $K_0(\mathcal{K}_{\mathbb{S}^1})\cong \mathbb{Z}$ e $K_1(\mathcal{K}_{\mathbb{S}^1}) = 0,$ e que vale $S\mathcal{K}_{\mathbb{S}^1}\cong C_0(\mathbb{R}, \mathcal{K}_{\mathbb{S}^1}),$ então podemos facilmente concluir que 
$$ K_0(\mathcal{E}^{\dagger}) = 0~~~~ \mbox{e}~~~~ K_1(\mathcal{E}^{\dagger}) \cong \mathbb{Z} .$$

\begin{prop}\label{primeiro} $K_0(\mathcal{A}^{\dagger}/{\mathcal{E}}^{\dagger})$ e $K_1(\mathcal{A}^{\dagger}/{\mathcal{E}}^{\dagger})$ são ambos isomorfos a $\mathbb{Z}^2$.
\end{prop}

\begin{proof} Antes de calcularmos estes grupos, observe que pelo teorema \ref{mdagger}, temos o seguinte isomorfismo
$$\mathcal{A}^{\dagger}/{\mathcal{E}}^{\dagger} \cong C(S^1_{\infty}\times \mathbb{S}^1 ).$$ 
Vamos primeiro calcular $K_0(C(S^1_{\infty}\times \mathbb{S}^1))$ e $K_1(C(S^1_{\infty}\times \mathbb{S}^1))$ e depois, pelo isomorfismo entre os grupos, determinar $K_0$ e $K_1$ de $\mathcal{A}^{\dagger}/{\mathcal{E}}^{\dagger}$.

Note que $C(S^1_{\infty}\times \mathbb{S}^1)\cong C(S^1_{\infty}, C(\mathbb{S}^1))$. Assim podemos construir a seqüência exata curta cindida envolvendo $C(\mathbb{S}^1)$ e sua suspensão:
\begin{center}
\begin{picture}(440,25)(0,0)
\put(65,0){$0$}\put(80,5){\vector(1,0){20}}
\put(110,0){$SC(\mathbb{S}^1)$}\put(155,5){\vector(1,0){20}}\put(163,10){${i}$}
\put(185,0){$C(S^1_{\infty}, C(\mathbb{S}^1))$}\put(260,6){\vector(1,0){20}}\put(267,12){${p}$}
\put(290,0){$C(\mathbb{S}^1)$}\put(280,2){\vector(-1,0){20}}\put(267,-6){${s}$}
\put(325,5){\vector(1,0){20}} \put(355,0){0,}
\end{picture}
\end{center}
onde $i$ é a inclusão, $p$ é a evaluação no ponto 1 e $s$ é tal que $p\circ s = I_{C(\mathbb{S}^1)}$. 

Então vale que 
\begin{equation}
K_i(C(S^1_{\infty}, C(\mathbb{S}^1)))= i_{\ast}(K_i(SC(\mathbb{S}^1)))\oplus s_{\ast}(K_i(C(\mathbb{S}^1))), ~i=0,1,
\label{susp}
\end{equation}
onde $i_{\ast}$ e $s_{\ast}$ são as aplicações induzidas. 

É conhecido que $K_0(C(\mathbb{S}^1))=\mathbb{Z}[1]_0, $ onde $[1]_0$ é um gerador do grupo $K_0(C(\mathbb{S}^1))$ e 1 representa a função $x\mapsto 1$ em $C(\mathbb{S}^1).$  Temos ainda que $K_1(C(\mathbb{S}^1))=\mathbb{Z}[\texttt{x}]_1,$ onde $[\texttt{x}]_1$ é um gerador de $K_1(C(\mathbb{S}^1))$ e $\texttt{x}$ é a função identidade em  $\mathbb{S}^1$. Sabendo disso, estamos aptos a calcular $K_i(C(S^1_{\infty}, C(\mathbb{S}^1))),~i = 0, 1,$ usando a igualdade (\ref{susp}).

\begin{itemize}
\item $K_0(C(S^1_{\infty}, C(\mathbb{S}^1))):$

Calculando $s_{*}$ no gerador de $K_0(C(\mathbb{S}^1))$, temos 
$$
s_{\ast}([1]_0) = [w \mapsto 1]_0 \in K_0(C(S^1_{\infty}, C(\mathbb{S}^1))).
$$
Vamos representar este gerador de $K_0(C(S^1_{\infty}, C(\mathbb{S}^1)))$ novamente por $[1]_0$, mas entendendo que se trata de  $[w\mapsto (x\mapsto 1)]_0.$ Assim, $s_{*}(K_0(C(\mathbb{S}^1))) = \mathbb{Z} [1]_0$. 

Passemos ao cálculo de $i_{*}(K_0(SC(\mathbb{S}^1)))$. Sabemos que $K_1(C(\mathbb{S}^1)) \cong K_0(SC(\mathbb{S}^1)),$ com o isomorfismo dado no teorema \ref{theta} e ainda, conhecemos o gerador do grupo $K_1(C(\mathbb{S}^1)).$ Temos então que
$$\theta_{C(\mathbb{S}^1)}([\texttt{x}]_1) = [w=u+iv \mapsto Q(\texttt{x}, u, v)]_0,$$
Logo, 
\begin{equation}
K_0(C(S^1_{\infty}, C(\mathbb{S}^1))) = \mathbb{Z}[1]_0 \oplus \mathbb{Z}[u+iv\mapsto Q(\texttt{x}, u, v)]_0.
\label{k0cs1s1}
\end{equation}

\item $K_1(C(S^1_{\infty}, C(\mathbb{S}^1))):$

Para calcular $s_{\ast}(K_1(C(\mathbb{S}^1)))$, calculemos $s_{\ast}$ num gerador:
$$s_{*}([\texttt{x}]_1) = [w \mapsto \texttt{x} ]_1 \in K_1(C(S^1_{\infty}, C(\mathbb{S}^1))).$$
Como a função $w\mapsto \texttt{x}$ não depende da variável de $w\in S^1_{\infty}$, podemos então escrever que  $s_{\ast}(K_1(C(\mathbb{S}^1))) = \mathbb{Z}[\texttt{x}]_1.$ 

O grupo $i_{*}(K_1(SC(\mathbb{S}^1)))$ é determinado a partir da aplicação de Bott.  Dado o isomorfismo $\beta_{C(\mathbb{S}^1)}: K_0(C(\mathbb{S}^1))\rightarrow K_1(SC(\mathbb{S}^1))$ e conhecendo o gerador de $K_0(C(\mathbb{S}^1))$, temos
$$\beta_{C(\mathbb{S}^1)}([1]_0) = [\texttt{w}]_1 ,$$
onde $\texttt{w}$ representa a função identidade em $S^1_{\infty}.$ Observe que esta função $\texttt{w} \in C(S^1_{\infty}, C(\mathbb{S}^1))$ não depende da variável em $\mathbb{S}^1$. 

Assim, concluímos que 
\begin{equation}
K_1(C(S^1_{\infty}, C(\mathbb{S}^1))) = \mathbb{Z}[\texttt{w}]_1 \oplus \mathbb{Z}[\texttt{x}]_1.
\label{k1cs1s1}
\end{equation}
\end{itemize}

De (\ref{k0cs1s1}) e de (\ref{k1cs1s1}), obtemos os isomorfismos desejados:
$$
K_0(\mathcal{A}^{\dagger}/{\mathcal{E}}^{\dagger})\cong\mathbb{Z}^2 ~~~~\mbox{e}~~~~
 K_1(\mathcal{A}^{\dagger}/{\mathcal{E}}^{\dagger})\cong\mathbb{Z}^2.
$$

Vamos agora determinar os geradores destes grupos. 
Tomando a classe em $\mathcal{A}^{\dagger}/{\mathcal{E}}^{\dagger}$ correspondente a respectiva função em  $C(S^1_{\infty}\times \mathbb{S}^1)$ por $\sigma^{\dagger}$,  conseguimos encontrar os geradores de $ K_i(\mathcal{A}^{\dagger}/{\mathcal{E}}^{\dagger}), i=0,1 :$
$$
K_0(\mathcal{A}^{\dagger}/{\mathcal{E}}^{\dagger}) = \mathbb{Z}[[I]_{{\mathcal{E}}^{\dagger}}]_0 \oplus 
\mathbb{Z}[[Q(\texttt{x}(M_x), A_5, A_6)]_{{\mathcal{E}}^{\dagger}}]_0, 
$$
$$
K_1(\mathcal{A}^{\dagger}/{\mathcal{E}}^{\dagger}) = \mathbb{Z}[[\texttt{x}(M_x)]_{{\mathcal{E}}^{\dagger}}]_1 \oplus 
\mathbb{Z}[[A_5+iA_6]_{{\mathcal{E}}^{\dagger}}]_1 ,
$$
onde Q agora é definida para operadores em $\mathcal{A}:$
\begin{equation}
Q(A,B,C) := 
\left(\begin{array}{cc}
I - (2-A-A^{\ast})\frac{I-B}{8}  & 
\frac{{A}^2-I}{2}\sqrt{\frac{I-B}{8}} -(A-I)^2 \frac{C}{8} \vspace{0.3cm}\\
\frac{{{A}^{\ast}}^{2}-I}{2}\sqrt{\frac{I-B}{8}} -(A^{\ast}-I)^2 \frac{C}{8} &
(2-A-A^{\ast})\frac{I-B}{8}  \\
\end{array}\right)
\label{matriz2}
\end{equation}

\end{proof}

Associada à seqüência (\ref{secdagger}), temos a seqüência exata de seis termos em K-teoria
\begin{equation}
\begin{array}{ccccc}
 0=K_0(\mathcal{E}^{\dagger})  & \! \stackrel{i_*}{\longrightarrow} & \! 
 K_0(\mathcal{A}^{\dagger})  & \! \stackrel{{\pi}_*}{\longrightarrow} & \!
 K_0(\mathcal{A}^{\dagger}/{\mathcal{E}}^{\dagger}) \cong \mathbb{Z}^2    \\ \\    
 {\delta}_1^{\dagger} \ \uparrow & \! ~  & \! ~  & \! ~ & \! \downarrow \ {\delta}_0^{\dagger}  \\ \\
 \mathbb{Z}^2 \cong K_1(\mathcal{A}^{\dagger}/{\mathcal{E}}^{\dagger}) & \! \stackrel{{\pi}_*}{\longleftarrow} & \! 
 K_1(\mathcal{A}^{\dagger})  & \! \stackrel{{i}_*}{\longleftarrow} & \!
 K_1(\mathcal{E}^{\dagger}) \cong \mathbb{Z}
\end{array}
\end{equation}

A descrição que temos do isomorfismo $K_1(C_0(\mathbb{R},\mathcal{K}_{\mathbb{S}^1}))\cong \mathbb{Z}$  não é explicita o suficiente para nos permitir saber exatamente qual é a imagem de $\delta_0$. Como 
$$\delta_0^{\dagger}([[I]_{\mathcal{E}^{\dagger}}]_0) = [exp(2\pi i I)]_1 = [I]_1 = 0 ,$$
temos que $[[I]_{\mathcal{E}^{\dagger}}]_0$ pertence ao núcleo de $\delta_0$. 

Portanto, apenas uma das seguintes possibilidades se verifica:
\begin{enumerate}
\item Se $\delta_0^{\dagger} = 0$ então $ K_0(\mathcal{A}^{\dagger}) \cong K_0(\mathcal{A}^{\dagger}/\mathcal{E}^{\dagger}) \cong \mathbb{Z}^2.$ Temos ainda a seqüência exata curta 
$$ 0 \rightarrow K_1(\mathcal{E}^{\dagger}) \cong \mathbb{Z} \rightarrow  K_1(\mathcal{A}^{\dagger}) \rightarrow K_1(\mathcal{A}^{\dagger}/{\mathcal{E}}^{\dagger})\cong \mathbb{Z}^2 \rightarrow 0. $$
Como $\mathbb{Z}^2$ é um módulo livre, então esta seqüência cinde e como $\mathbb{Z}$ é abeliano, podemos então escrever $K_1(\mathcal{A}^{\dagger})$ como a seguinte soma direta:
$$K_1(\mathcal{A}^{\dagger}) \cong \mathbb{Z}\oplus \mathbb{Z}^2 = \mathbb{Z}^3.$$

\item Se $\delta_0^{\dagger}([[Q(\texttt{x}(M_x), A_5, A_6)]_{\mathcal{E}^{\dagger}}]_0)$ corresponder a um inteiro $\eta \neq 0$ pelo isomorfismo $K_1(\mathcal{E}^{\dagger})\cong \mathbb{Z}$,  temos que
$$ 0 \rightarrow K_0(\mathcal{A}^{\dagger}) \rightarrow  K_0(\mathcal{A}^{\dagger}/{\mathcal{E}}^{\dagger}) \cong \mathbb{Z}^2 \rightarrow Im\delta_0^{\dagger} \cong \eta\mathbb{Z} \rightarrow 0, $$
e portanto $ K_0(\mathcal{A}^{\dagger}) \cong \mathbb{Z}.$ Podemos ainda escrever a seqüência 
$$ 0 \rightarrow \frac{K_1(\mathcal{E}^{\dagger})}{Im\delta_0^{\dagger}} \cong \mathbb{Z}_{\eta} \rightarrow  K_1(\mathcal{A}^{\dagger}) \rightarrow K_1(\mathcal{A}^{\dagger}/{\mathcal{E}}^{\dagger})\cong \mathbb{Z}^2 \rightarrow 0, $$
onde $\mathbb{Z}_{\eta} \cong \mathbb{Z}/{\eta}\mathbb{Z}$.
Como $\mathbb{Z}^2$ é um módulo livre e os grupos são todos abelianos, então podemos escrever 
$ K_1(\mathcal{A}^{\dagger})$ como a soma  direta $ \mathbb{Z}^2 \oplus \mathbb{Z}_{\eta}$.
\label{poss1}
\end{enumerate}

\section{A K-teoria de $\mathcal{A}^{\diamond}$}
 
A idéia usada na seção anterior será reproduzida aqui para calcular os K-grupos de $\mathcal{A}^{\diamond}$. Começamos construindo a seqüência exata curta envolvendo a álgebra e seu ideal comutador:
\begin{equation}
0 \ \longrightarrow \ \mathcal{E}^{\diamond} \  \stackrel{i}{\longrightarrow} \ \mathcal{A}^{\diamond} \  
\stackrel{\pi}{\longrightarrow} \ \frac{\mathcal{A}^{\diamond}}{\mathcal{E}^{\diamond}} \ \longrightarrow \ 0.
\end{equation}

Do teorema \ref{gamma'}, temos o isomorfismo entre ${\mathcal{E}^{\diamond}}$ e $C(S^{1},{\cal{K}}_{\mathbb{Z}\times \mathbb{S}^1})$. Considere a seqüência exata cindida
\begin{center}
\begin{picture}(440,20)(0,0)
\put(65,0){$0$}\put(80,5){\vector(1,0){20}}
\put(110,0){$S{\mathcal{K}}_{\mathbb{Z}\times \mathbb{S}^1}$}\put(155,5){\vector(1,0){20}}\put(163,10){${i}$}
\put(185,0){$C(S^1, {\cal{K}}_{\mathbb{Z}\times \mathbb{S}^1})$}\put(260,6){\vector(1,0){20}}\put(267,12){${p}$}
\put(290,0){${\mathcal{K}}_{\mathbb{Z}\times \mathbb{S}^1}$}\put(280,2){\vector(-1,0){20}}\put(267,-6){${s}$}
\put(325,5){\vector(1,0){20}} \put(355,0){0,}
\end{picture}
\end{center}
onde $p(f)=f(1).$ Então,  
\begin{equation}
K_0(C(S^{1},{\cal{K}}_{\mathbb{Z}\times \mathbb{S}^1})) \cong K_0({\cal{K}}_{\mathbb{Z}\times \mathbb{S}^1}) = \mathbb{Z}[E]_0,
\label{k0k}
\end{equation}
onde o isomorfismo é induzido por $p$ e $E$ é uma projeção de posto 1 de ${\cal{K}}_{\mathbb{Z}\times \mathbb{S}^1},$ e $$K_1(C(S^{1},{\cal{K}}_{\mathbb{Z}\times \mathbb{S}^1}))\cong K_1(S{\cal{K}}_{\mathbb{Z}\times \mathbb{S}^1}) = \mathbb{Z}[z\mapsto zE+(1-E)]_1,$$
pela aplicação de Bott. 
Portanto 
$$ K_i({\mathcal{E}^{\diamond}}) \cong \mathbb{Z} , ~i=0,1. $$

\begin{prop} $K_0({\mathcal{A}^{\diamond}}/{\mathcal{E}^{\diamond}}) \cong \mathbb{Z}^4$ e $K_1({\mathcal{A}^{\diamond}}/{\mathcal{E}^{\diamond}})
\cong \mathbb{Z}^4$.
\end{prop}

\begin{proof} Pela proposição \ref{mdiamond}, conhecemos o espaço símbolo de $\mathcal{A}^{\diamond}$. Comecemos então determinando $K_0$ e $K_1$ de $C(S^1_{\infty} \times \mathbb{S}^1 \times S^1)$ (mudamos a ordem do produto cartesiano por ser mais conveniente para nós). Para facilitar os nossos cálculos, vamos reescrever $C(S^1_{\infty} \times \mathbb{S}^1 \times S^1)$
como $C(S^1_{\infty}, C(\mathbb{S}^1, C(S^1))).$ Lembrando que $S^1$ é o espaço símbolo das funções $2\pi$-periódicas, temos que $K_0(C(S^1))= \mathbb{Z}[1]_0$ e  $K_1(C(S^1))= \mathbb{Z}[\texttt{z}:e^{i\theta}\mapsto e^{i\theta}]_1.$ 

Com base no que foi feito na demonstração da proposição \ref{primeiro} podemos concluir que 
$$K_0(C(\mathbb{S}^1, C(S^1))) = \mathbb{Z}[1]_0 \oplus \mathbb{Z}[x=a+ib \mapsto Q( \texttt{z} , a, b)]_0$$
e
$$K_1(C(\mathbb{S}^1, C(S^1))) = \mathbb{Z}[\texttt{x}]_1 \oplus \mathbb{Z}[\texttt{z}]_1.$$
Observe que aqui trabalhamos com círculos (de origem) diferentes daqueles da proposição \ref{primeiro}, mas os cálculos são os mesmos. 

Denote $Y:=C(\mathbb{S}^1, C(S^1)).$ Agora usamos a seqüência envolvendo a suspensão de $Y, $ 
\begin{center}
\begin{picture}(440,17)(0,0)
\put(85,0){$0$}\put(100,5){\vector(1,0){20}}
\put(135,0){$SY$}\put(160,5){\vector(1,0){20}}\put(168,10){${i}$}
\put(190,0){$C(S^1_{\infty}, Y)$}\put(250,6){\vector(1,0){20}}\put(258,12){${p}$}
\put(290,0){$Y$}\put(270,2){\vector(-1,0){20}}\put(258,-6){${s}$}
\put(310,5){\vector(1,0){20}} \put(340,0){0,}
\end{picture}
\end{center}
Então 
$$
K_i(C(S^1_{\infty}, Y)) = i_{\ast}(K_i(SY))\oplus s_{\ast}(K_i(Y)),~i=0,1.
$$
Com esta igualdade, podemos calcular $K_i(C(S^1_{\infty}, Y)), ~i=0,1.$

\begin{itemize}
\item $K_0(C(S^1_{\infty}, Y)):$

Vamos calcular $s_{\ast}$ nos geradores de $K_0(Y):$
$$ s_{\ast}([1]_0)=[w \mapsto 1]_0 ,$$
$$ s_{\ast}([a+ib \mapsto Q(\texttt{z}, a, b)]_0)=[w\mapsto \left(a+ib \mapsto Q(\texttt{z}, a, b)\right)]_0 .$$
Para calcular $i_{\ast}(K_0(SY))$ usamos o isomorfismo ${\theta}_Y : K_1(Y) \rightarrow K_0(SY) :$
$$ {\theta}_Y([\texttt{x}]_1) = [w=u+iv \mapsto Q(\texttt{x}, u, v)]_0 , $$  
$$ {\theta}_Y([\texttt{z}]_1) = [w=u+iv \mapsto Q(\texttt{z}, u, v)]_0.$$
Assim, podemos concluir que $K_0(C(S^1_{\infty}\times \mathbb{S}^1\times S^1))$ é igual a 
$$
\mathbb{Z}[(w, x,e^{i\theta})\mapsto 1]_0 
\oplus \mathbb{Z}[(w, a+ib,e^{i\theta})\mapsto Q(e^{i\theta}, a, b)]_0 
$$
\begin{equation}
\oplus \mathbb{Z}[(u+iv, x,e^{i\theta})\mapsto Q(x, u, v)]_0 
\oplus \mathbb{Z}[(u+iv, x,e^{i\theta})\mapsto Q(e^{i\theta}, u, v)]_0 
\end{equation}

\item $K_1(C(S^1_{\infty}, Y)):$

$s_{\ast}$ em $K_1(Y)$:
$$ s_{\ast}([\texttt{x}]_1)=[w \mapsto \texttt{x}]_1  ~~~~\mbox{e} ~~~~ 
s_{\ast}([\texttt{z}]_1)=[w \mapsto \texttt{z}]_1 .$$
Usamos o isomorfismo de Bott para determinar os geradores de $K_1(SY):$
$$ \beta_{Y}([1]_0)=[\texttt{w}]_1 ,$$
onde $ \texttt{w}$ é a função identidade de $C(S_{\infty}^1),$ e 
$$ \beta_{Y}([a+ib\mapsto Q(\texttt{z}, a, b)]_0)=[w \mapsto (a+ib\mapsto 1_{2} + (w -1)\cdot Q(\texttt{z}, a, b))]_1, $$
onde $1_2$ representa a matriz identidade de dimensão 2. 
Temos então que $ K_1(C(S^1_{\infty}\times \mathbb{S}^1\times S^1))$ é igual a
$$ \mathbb{Z}[(w, x,e^{i\theta})\mapsto x]_1 \oplus \mathbb{Z}[(w, x,e^{i\theta})\mapsto {w}]_1 \oplus \mathbb{Z}[(w, x,e^{i\theta})\mapsto e^{i\theta}]_1 $$
\begin{equation}
\oplus\mathbb{Z}[(w, a+ib,e^{i\theta})\mapsto 1_{2} + (w-1)Q(e^{i\theta}, a, b)]_1
\end{equation}
\end{itemize}

Para determinarmos os geradores de $K_0$ e $K_1$  de $\mathcal{A}^{\diamond}/\mathcal{E}^{\diamond}$, usamos o símbolo $\sigma^{\diamond}: \mathcal{A}^{\diamond} \rightarrow C(S^1_{\infty} \times \mathbb{S}^1 \times S^1)$. 
Considerando $Q$ como em \ref{matriz2} definida para operadores de $\mathcal{A}$, segue que
\begin{eqnarray*}
K_0({\mathcal{A}^{\diamond}}/{\mathcal{E}^{\diamond}}) & = & \mathbb{Z}[[I]_{\mathcal{E}^{\diamond}}]_0 \oplus 
\mathbb{Z}[[Q( e^{iM_{t}} , \texttt{a}(M_x),\texttt{b}(M_x))]_{\mathcal{E}^{\diamond}}]_0 \oplus 
\mathbb{Z}[[Q( e^{iM_{t}} , A_5, A_6)]_{\mathcal{E}^{\diamond}}]_0 \oplus \nonumber \\
& &  \mathbb{Z}[[Q( \texttt{x}(M_x), A_5, A_6)]_{\mathcal{E}^{\diamond}}]_0 
\end{eqnarray*}
\begin{eqnarray*}
K_1({\mathcal{A}^{\diamond}}/{\mathcal{E}^{\diamond}}) & = & \mathbb{Z}[[e^{iM_{t}}]_{\mathcal{E}^{\diamond}}]_1 \oplus 
\mathbb{Z}[[\texttt{x}(M_x)]_{\mathcal{E}^{\diamond}}]_1 \oplus 
\mathbb{Z}[[A_5 + i A_6]_{\mathcal{E}^{\diamond}}]_1 \oplus \nonumber \\
& &  \mathbb{Z}[[Id + (A_5 + i A_6 - 1)Q( e^{iM_{t}} , \texttt{a}(M_x),\texttt{b}(M_x))]_{\mathcal{E}^{\diamond}}]_1 ,
\end{eqnarray*}
onde 
$e^{iM_{t}}$ é o operador de multiplicação dado por 
$$e^{iM_{t}}(u)(t,x) = e^{it}u(t,x),$$
e $\texttt{a}(M_x)$, $ \texttt{b}(M_x)$  são dados por:
$$\texttt{a}(M_x)(u)(t,x=a+ib) = a\cdot u(t,x)~, ~~ \texttt{b}(M_x)(u)(t,x=a+ib) = b \cdot u(t,x). $$ 
\end{proof}

Para obtermos a K-teoria de $\mathcal{A}^{\diamond}$, vamos calcular as aplicações exponencial e do índice para a seqüência exata de seis termos em K-teoria abaixo:
\begin{equation}
\begin{array}{ccccc}
 \mathbb{Z} \cong K_0(\mathcal{E}^{\diamond})  & \! \stackrel{i_*}{\longrightarrow} & \! 
 K_0(\mathcal{A}^{\diamond})  & \! \stackrel{{\pi}_*}{\longrightarrow} & \!
 K_0(\mathcal{A}^{\diamond}/{\mathcal{E}}^{\diamond}) \cong \mathbb{Z}^4   \\ \\    
 {\delta}_1^{\diamond} \ \uparrow & \! ~  & \! ~  & \! ~ & \! \downarrow \ {\delta}_0^{\diamond}  \\ \\
 \mathbb{Z}^4 \cong K_1(\mathcal{A}^{\diamond}/{\mathcal{E}}^{\diamond}) & \! \stackrel{{\pi}_*}{\longleftarrow} & \! 
 K_1(\mathcal{A}^{\diamond})  & \! \stackrel{{i}_*}{\longleftarrow} & \!
 K_1(\mathcal{E}^{\diamond})\cong \mathbb{Z}
\end{array}
\label{sgr}
\end{equation}

Para determinar $\delta_1^{\diamond}$,  conseguimos mostrar que vale a proposição análoga da proposição 3 de \cite{cintia}  para $\mathcal{A}^{\diamond}$.

\begin{prop}\label{delta1diam} Seja $A$ em $\mathcal{A}^{\diamond}$ tal que $[A]_{\mathcal{E}^{\diamond}}$ seja inversível em $\mathcal{A}^{\diamond}/\mathcal{E}^{\diamond}.$  Então ${\gamma}_{A}'(z)$ é um operador de Fredholm em $\mathcal{L}_{\mathbb{Z}\times \mathbb{S}^1}$ para todo $z \in S^1$ e
$$ 
{\delta}_{1}^{\diamond}(\left[[A]_{\mathcal{E}^{\diamond}}]_1\right) =
\mathsf{ind}{\gamma}_{A}'(1)[E]_0 \in K_0({\cal{K}}_{\mathbb{Z}\times \mathbb{S}^1}),
$$
onde $\mathsf{ind}$ denota o índice de Fredholm e $E$ é uma projeção de posto 1 em ${\mathcal{K}}_{\mathbb{Z}\times \mathbb{S}^1}$.
\end{prop}

\begin{proof} Seja $B$ em $\mathcal{A}^{\diamond}$ tal que $ I - AB $ e $ I - BA $ pertençam a $\mathcal{E}^{\diamond}.$ Pela proposição \ref{gamma'}, $ {\gamma}_{I-AB}'= {\gamma}_{I}' - {\gamma}_{A}'{\gamma}_{B}' $ e $ {\gamma}_{I-BA}'= {\gamma}_{I}' - {\gamma}_{B}'{\gamma}_{A}' $ pertencem a 
$C(S^1, \mathcal{K}_{\mathbb{Z}\times \mathbb{S}^1})$. Assim, para cada $z\in S^1,~{\gamma}_{A}'(z)$ possui um elemento inverso, módulo $\mathcal{K}_{\mathbb{Z}\times \mathbb{S}^1},$ e portanto ${\gamma}_{A}'(z)$ é um operador de Fredholm em $\mathcal{L}_{\mathbb{Z}\times \mathbb{S}^1}.$

Considerando a aplicação dada pela projeção canônica $\pi : \mathcal{A}^{\diamond}\rightarrow \mathcal{A}^{\diamond}/ \mathcal{E}^{\diamond}$, tome $a=A$ uma pré-imagem para $u=[A]_{\mathcal{E}^{\diamond}}.$ Seja $B$ em $\mathcal{A}^{\diamond}$ tal que $u^{-1}=[B]_{\mathcal{E}^{\diamond}},$ então $B$ é uma pré-imagem para $u^{-1}.$ Usando o lema \ref{lema1} e o isomorfismo \ref{k0k}, temos que ${\delta}_1([[A]_{\mathcal{E}^{\diamond}}]_1)$ é igual a 
$$
\left[
\left(\begin{array}{cc}
2{\gamma}_{A}'(1){\gamma}_{B}'(1)-({\gamma}_{A}'(1){\gamma}_{B}'(1))^2 & {\gamma}_{A}'(1)(2-{\gamma}_{B}'(1){\gamma}_{A}'(1))(1-{\gamma}_{B}'(1){\gamma}_{A}'(1)) \\
(1-{\gamma}_{B}'(1){\gamma}_{A}'(1)){\gamma}_{B}'(1) & (I-{\gamma}_{B}'(1){\gamma}_{A}'(1))^2 \\
\end{array}\right)
\right]_0
-
\left[
\left(\begin{array}{cc}
I & 0 \\
0 & 0 \\
\end{array}\right)
\right]_0.
$$

Para a seqüência exata
$$  0 \ \longrightarrow \ {\cal{K}}_{\mathbb{Z}\times \mathbb{S}^1} \  \stackrel{i}{\longrightarrow} \ 
{\cal{L}}_{\mathbb{Z}\times \mathbb{S}^1} \  
\stackrel{\pi}{\longrightarrow} \ 
\frac{{\cal{L}}_{\mathbb{Z}\times \mathbb{S}^1}}{{\cal{K}}_{\mathbb{Z}\times \mathbb{S}^1}} \ \longrightarrow \ 0, 
$$
temos que a aplicação do índice $\delta_1: K_1(\mathcal{L}_{\mathbb{Z}\times \mathbb{S}^1}/\mathcal{K}_{\mathbb{Z}\times \mathbb{S}^1})\rightarrow K_0(\mathcal{K}_{\mathbb{Z}\times \mathbb{S}^1})$ é dada pelo índice de Fredholm. Assim, 
$$\delta_1([[{\gamma}_{A}'(1)]_{\mathcal{K}}]_1) = \mathsf{ind}{\gamma}_{A}'(1)[E]_0.$$
Mas pelo lema \ref{lema1}, $\delta_1([[{\gamma}_{A}'(1)]_{\mathcal{K}}]_1)$ é dada pela mesma matriz de ${\delta}_1([[A]_{\mathcal{E}^{\diamond}}]_1)$. Logo, 
$$ \delta_1([[A]_{\mathcal{E}^{\diamond}}]_1) = \mathsf{ind}{\gamma}_{A}'(1)[E]_0.$$

\end{proof}

\begin{teo}\label{sobre} A aplicação ${\delta}_1^{\diamond}$ em (\ref{sgr}) é sobrejetora.
\end{teo}

\begin{proof} Calculemos ${\delta}_1^{\diamond}$ nos geradores de $K_1(\mathcal{A}^{\diamond}/{\mathcal{E}^{\diamond}}).$ É fácil ver que os operadores de multiplicação $ \texttt{x}(M_x),~ e^{iM_t} $ são unitários em $\mathcal{A}^{\diamond}$ e portanto são pré-imagens para $[\texttt{x}(M_x)]_{{\mathcal{E}^{\diamond}}},~ [e^{iM_t}]_{{\mathcal{E}^{\diamond}}} $. Logo, 
$$
{\delta}_1^{\diamond}([[\texttt{x}(M_x)]_{\mathcal{E}^{\diamond}}]_1) = {\delta}_1^{\diamond}([[e^{iM_t}]_{\mathcal{E}^{\diamond}}]_1) = 0.
$$

Usaremos a proposição \ref{delta1diam} para calcular ${\delta}_1^{\diamond}$ nos outros dois geradores. 
Observe que $\gamma_{A}'(1)$ significa $\gamma_{A}'(e^{2\pi i \varphi}),$ com $\varphi = 0.$
Para o elemento  $[[A_5+iA_6]_{\mathcal{E}^{\diamond}}]_1$, temos que $\gamma_{A_5+iA_6}'(1)$ em $\mathcal{L}_{\mathbb{Z}\times\mathbb{S}^1}$ é o operador de multiplicação pela seqüência
\begin{eqnarray}
\left(j\left(1+j^2-\frac{\partial^2}{\partial{\beta}^2}\right)^{-1/2} + 
i \frac{1}{i} \frac{\partial}{\partial{\beta}}\left(1+j^2-\frac{\partial^2}{\partial{\beta}^2}\right)^{-1/2}\right)_j 
\label{a5a6'}
\end{eqnarray}
O próximo passo é calcular o índice de Fredholm deste operador. Para isso, podemos vê-lo como um operador em $\mathcal{L}(L^2(\mathbb{Z}) \otimes L^2(\mathbb{S}^1))$ e assim conjugá-lo com $I_{\mathbb{Z}}\otimes F_d$, ou seja, a conjugação pela transformada de Fourier discreta ocorrerá apenas em $L^2(\mathbb{S}^1)$.   
Obtemos portanto, um operador em $\mathcal{L}_{\mathbb{Z}\times\mathbb{Z}}$ de multiplicação pela seqüência 
$$
(j(1+j^2+k^2)^{-1/2} - i k(1+j^2+k^2)^{-1/2})_{j,k} ~\in ~ {\ell}^2(\mathbb{Z}\times\mathbb{Z}).
$$
Agora vamos calcular a dimensão do núcleo deste operador:
$$
(j(1+j^2+k^2)^{-1/2} - i k(1+j^2+k^2)^{-1/2}u_{j,k})_{j,k} = 0 ~~\Longleftrightarrow 
$$
$$
\frac{j- i k}{(1+j^2+k^2)^{-1/2}}u_{j,k} = 0 ~~\forall ~~j,~ k ~~\Longleftrightarrow 
u_{j,k}=0, ~\forall ~j~\mbox{ou}~k \neq 0,~\mbox{e}~u_{0,0}~\mbox{qualquer.} 
$$
Concluímos então que a dimensão do núcleo é 1. Usando o mesmo raciocínio, concluímos que a dimensão do núcleo de seu adjunto também é 1. Portanto, 
$$
\mathsf{ind}\gamma_{A_5+iA_6}'(1) = 0 .
$$ 
Logo, 
$$
{\delta}_1([[A_5 + i A_6]_{\mathcal{E}^{\diamond}}]_1) = 0.
$$

Considere o gerador $[[I+(A_5+iA_6-1)Q(e^{iM_t},\texttt{a}(M_x),\texttt{b}(M_x))]_{\mathcal{E}^{\diamond}}]_1.$ Então, 
\begin{eqnarray*}
{\gamma}_{I+(A_5+iA_6-1)Q(e^{iM_t},\texttt{a}(M_x),\texttt{b}(M_x))}'(1) &=& 
I + ({\gamma}_{A_5+iA_6}'(1) - 1){\gamma}_{Q(e^{iM_t},\texttt{a}(M_x),\texttt{b}(M_x))}'(1)
\end{eqnarray*}
Já conhecemos o operador ${\gamma}_{A_5+iA_6}'(1)\in \mathcal{L}_{\mathbb{Z}\times\mathbb{S}^1}$ em (\ref{a5a6'}). O outro operador é dado por
\begin{equation}
{\gamma}_{Q(e^{iM_t},\texttt{a}(M_x),\texttt{b}(M_x))}'(1) = Q(Y_{-1},\texttt{a}(M_x), \texttt{b}(M_x))~ \in ~M_2(\mathcal{L}_{\mathbb{Z}\times\mathbb{S}^1}). 
\label{gamma'Q}
\end{equation}

Explicaremos primeiro os passos a serem tomados para o cálculo do índice do operador \linebreak ${\gamma}_{I+(A_5+iA_6-1)Q(e^{iM_t},\texttt{a}(M_x),\texttt{b}(M_x))}'(1)$, para depois efetuá-lo. 

A idéia será conjugar este operador com a transformada de Fourier discreta na variável inteira e obter um operador de $\mathcal{L}(L^2(S^1\times\mathbb{S}^1).$ Temos que a variedade $S^1\times\mathbb{S}^1$ é compacta e tem \textit{Todd class} (\cite{fedosov}, (11)) igual a 1. Com as hipóteses satisfeitas, usaremos a fórmula do índice para operadores pseudodiferenciais elípticos de Fedosov \cite{fedosov}, uma particularização da fórmula de Atiyah-Singer, a qual pode ser escrita em termos do símbolo do operador:
\begin{equation}
\mathsf{ind} A = \frac{(-1)^{n+1}}{(2\pi i)^n}\frac{(n-1)!}{(2n-1)!}\int_{S(M)}Tr({\sigma}^{-1}d\sigma)^{2n-1} ,
\label{indice}
\end{equation}
onde $n$ é dimensão da variedade $M$, $S(M)$ é o fibrado das coesferas de $M$, $\sigma$ é o símbolo principal do operador pseudodiferencial $A$ e $({\sigma}^{-1}d\sigma)^{2n-1}$ é uma $(2n-1)$-forma. 
O símbolo principal coincide com o $\sigma$-símbolo que definimos neste trabalho,\cite{cordes3} e \cite{severino}.

Conjugando o operador ${\gamma}_{A_5+iA_6}'(1)$ com $F_d \otimes I_{\mathbb{S}^1}$, temos:
$$
F_d^{-1}{\gamma}_{A_5+iA_6}(1,+1)F_d = 
-\frac{1}{i}\frac{\partial}{\partial{\alpha}}\left(1-\frac{\partial^2}{\partial{\alpha}^2}-\frac{\partial^2}{\partial{\beta}^2}\right)^{-1/2} + 
i \frac{1}{i} \frac{\partial}{\partial{\beta}}\left(1-\frac{\partial^2}{\partial{\alpha}^2}-\frac{\partial^2}{\partial{\beta}^2}\right)^{-1/2},
$$
com $e^{i\alpha} \in S^1$~ e ~$x = e^{i{\beta}} \in \mathbb{S}^1.$ Para $Q(Y_{-1},\texttt{a}(M_x), \texttt{b}(M_x))$, 
$$
F_d^{-1}Q(Y_{-1},\texttt{a}(M_x), \texttt{b}(M_x))F_d = Q(F_d^{-1}Y_{-1}F_d,\texttt{a}(M_x), \texttt{b}(M_x)) = 
Q(e^{iM_{\alpha}},\texttt{a}(M_x), \texttt{b}(M_x))
$$

Lembrando que $\texttt{a}(M_x)$ e $\texttt{b}(M_x)$ são operadores de multiplicação pelas funções $\texttt{a} , \texttt{b} \in C^{\infty}(\mathbb{S}^1)$  tais que $ \texttt{a}: x = a+ib \mapsto Re(x) = a $, e $ \texttt{b}: x = a+ib \mapsto Im(x)= b$, podemos reescrevê-las em função de $x=e^{i\beta}$:
$$ \texttt{a}: e^{i\beta} \mapsto \cos\beta ~~, ~~~~ \texttt{b}: e^{i\beta} \mapsto \mbox{sen}\beta .$$

Obtemos então o seguinte operador em $M_2(\mathcal{L}_{S^1\times\mathbb{S}^1}):$
$$ T:=
I + \left( \frac{1}{i} 
\left( -\frac{\partial}{\partial{\alpha}} + i \frac{\partial}{\partial{\beta}} \right) \left(1-\frac{\partial^2}{\partial{\alpha}^2}-\frac{\partial^2}{\partial{\beta}^2}\right)^{-1/2} - 1 \right) 
Q(e^{iM_{\alpha}},\texttt{a}(M_{\beta}), \texttt{b}(M_{\beta}))
$$

Antes de irmos para o cálculo do índice de $T$, devemos determinar seu símbolo $\sigma$. Como $S(S^1\times \mathbb{S}^1)= 
S^1\times \mathbb{S}^1\times S^1_{\infty},$ onde $S^1_{\infty}$ representa o círculo proveniente do fibrado, temos que $\sigma_{T} \in C(S^1\times \mathbb{S}^1\times S^1_{\infty})$. Calculemos o símbolo de cada operador envolvido em $T$ separadamente. 
\begin{enumerate}
\item $\sigma_{e^{iM_{\alpha}}}((e^{i\alpha}, e^{i\beta}, e^{i\lambda}))= e^{i\alpha} $
\item $\sigma_{\texttt{a}(M_{\beta})}((e^{i\alpha}, e^{i\beta}, e^{i\lambda}))= Re(e^{i\beta})=\cos\beta $
\item $\sigma_{\texttt{b}(M_{\beta})}((e^{i\alpha}, e^{i\beta}, e^{i\lambda}))= Im(e^{i\beta})=\mbox{sen}\beta $
\item $\sigma_{\frac{1}{i}\frac{\partial}{\partial{\alpha}} \left(1-\frac{\partial^2}{\partial{\alpha}^2}-\frac{\partial^2}{\partial{\beta}^2}\right)^{-1/2}}
((e^{i\alpha}, e^{i\beta}, e^{i\lambda}))= Re (e^{i\lambda}) = \cos{\lambda} $
\item $\sigma_{\frac{1}{i}\frac{\partial}{\partial{\beta}}
\left(1-\frac{\partial^2}{\partial{\alpha}^2}-\frac{\partial^2}{\partial{\beta}^2}\right)^{-1/2}}
(e^{i\alpha}, e^{i\beta}, e^{i\lambda}))= Im( e^{i\lambda}) = \mbox{sen}{\lambda} $
\end{enumerate}
$$\Rightarrow \sigma_{T}((e^{i\alpha}, e^{i\beta}, e^{i\lambda})) = 
I + (-\cos\lambda + i \mbox{sen}\lambda - 1)Q(e^{i\alpha},\cos\beta,\mbox{sen}\beta)$$

No integrando temos
$$\sigma^{-1}d\sigma = \sigma^{-1}\left(\frac{\partial}{\partial\alpha}\sigma d\alpha + \frac{\partial}{\partial\beta}\sigma d\beta +
\frac{\partial}{\partial\lambda}\sigma d\lambda \right). $$
Chamando de $\sigma_{\alpha}= \frac{\partial}{\partial\alpha}\sigma , ~ \sigma_{\beta}=\frac{\partial}{\partial\beta}\sigma , ~
\sigma_{\lambda}= \frac{\partial}{\partial\lambda}\sigma,$ então
$$ (\sigma^{-1}d\sigma)^3 = [A_1 + A_2 +A_3 ] d\alpha d\beta d\lambda  $$
onde 
$$ A_1 = \sigma^{-1} \sigma_{\alpha} \sigma^{-1}(\sigma_{\beta} \sigma^{-1} \sigma_{\lambda} - \sigma_{\lambda} \sigma^{-1}\sigma_{\beta} ), $$
$$ A_2 = \sigma^{-1} \sigma_{\beta} \sigma^{-1}(\sigma_{\lambda} \sigma^{-1}\sigma_{\alpha} - \sigma_{\alpha} \sigma^{-1}\sigma_{\lambda} ), $$
$$ A_3 = \sigma^{-1} \sigma_{\lambda} \sigma^{-1}(\sigma_{\alpha}\sigma^{-1}\sigma_{\beta} - \sigma_{\beta} \sigma^{-1} \sigma_{\alpha}). $$

A partir daqui usamos o software \textit{Maple} para resolver diversos cálculos, como as derivadas parciais, os produtos de matrizes, o traço e por último a integral da fórmula (\ref{indice}). No apêndice pode ser encontrado os comandos usados para fazer estes cálculos. O valor da integral é $24{\pi}^{2}$ e portanto o índice de $T$ é 
$$
\mathsf{ind}~ T = \frac{(-1)^{3}}{(2\pi i)^2}~\frac{1}{3!}~24{\pi}^{2} = 1
$$
Logo, o índice de ${\gamma}_{I+(A_5+iA_6-1)Q(e^{iM_t},\texttt{a}(M_x),\texttt{b}(M_x))}'(1)$ é 1 e portanto, 
$$
{\delta}_1([[ I+(A_5+iA_6-1)Q(e^{iM_t},\texttt{a}(M_x),\texttt{b}(M_x)) ]_{\mathcal{E}^{\diamond}}]_1) = [E]_0. 
$$
Assim concluímos que ${\delta}_1^{\diamond}$ é sobrejetora. 
\end{proof}

Analisando a aplicação exponencial, temos que ${\delta}_0^{\diamond}$ nos geradores  $[[I]_{\mathcal{E}^{\diamond}}]_0$ e \linebreak $[[Q(e^{iM_t},\texttt{a}(M_x),\texttt{b}(M_x))]_{\mathcal{E}^{\diamond}}]_0 $  é igual a 0, pois estes operadores são projeções em $\mathcal{A}^{\diamond}.$ Mas não determinamos a imagem de ${\delta}_0^{\diamond}$ pelo motivo já citado no cálculo de ${\delta}_0^{\dagger}.$
Portanto, temos uma das seguintes possibilidades:

\begin{enumerate}

\item Se $\delta_0^{\diamond} = 0$ então $ K_0(\mathcal{A}^{\diamond})\cong K_0(\mathcal{A}^{\diamond}/\mathcal{E}^{\diamond}) \cong \mathbb{Z}^4$. Temos ainda a seguinte seqüência 
$$ 0 \rightarrow K_1(\mathcal{E}^{\diamond}) \cong \mathbb{Z} \rightarrow K_1(\mathcal{A}^{\diamond}) \rightarrow Ker \delta_1^{\diamond} \cong \mathbb{Z}^3 \rightarrow 0 $$
Como $\mathbb{Z}^3$ é um módulo livre, então esta seqüência cinde. Além disso, $\mathbb{Z}$ e $\mathbb{Z}^3$ são abelianos. Portanto 
$ K_1(\mathcal{A}^{\diamond}) \cong \mathbb{Z}\oplus \mathbb{Z}^3 = \mathbb{Z}^4$.

\item Se $Im \delta_0^{\diamond} \cong \mu\mathbb{Z},$ para algum inteiro $\mu$ não nulo, então  $K_0(\mathcal{A}^{\diamond}) \cong ker \delta_0^{\diamond} \cong \mathbb{Z}^3 $ e 
$ K_1(\mathcal{A}^{\diamond}) \cong \mathbb{Z}^3 \oplus \mathbb{Z}_{\mu}, $ onde   $\mathbb{Z}_{\mu} \cong \mathbb{Z}/\mu\mathbb{Z}.$
\label{poss2}
\end{enumerate}

\subsection{A seqüência de Pimsner-Voiculescu}

No teorema \ref{prodcruz}, vimos que existe um isomorfismo $\varphi : \mathcal{B}^{\dagger}{\rtimes}_{\alpha}\mathbb{Z} \rightarrow {\mathcal{B}}^{\diamond},$ onde $\alpha$ é o automorfismo de translação por 1, $\mathcal{B}^{\dagger} = F^{-1} \mathcal{A}^{\dagger} F$ e $\mathcal{B}^{\diamond} = F^{-1} \mathcal{A}^{\diamond} F$. Em \cite{pv}, teorema 2.4, Pimsner e Voiculescu obtiveram uma seqüência exata de seis termos envolvendo apenas os K-grupos da álgebra inicial e da ágebra do produto cruzado.
$$
\begin{array}{ccccc}
 K_0({\mathcal{B}^{\dagger}})  & \! \stackrel{id_* - \alpha^{-1}_* }{\longrightarrow} & \! 
 K_0({\mathcal{B}^{\dagger}})  & \! \stackrel{i_*}{\longrightarrow} & \!
 K_0(\mathcal{B}^{\dagger}{\rtimes}_{\alpha}\mathbb{Z}) \vs  \\     
       \ \uparrow & \! ~  & \! ~  & \! ~ & \! \downarrow \        \vs \\ 
 K_1(\mathcal{B}^{\dagger}{\rtimes}_{\alpha}\mathbb{Z}) & \! \stackrel{i_*}{\longleftarrow} & \! 
 K_1({\mathcal{B}^{\dagger}})  & \! \stackrel{id_* - \alpha^{-1}_*}{\longleftarrow} & \!
 K_1({\mathcal{B}^{\dagger}})
\end{array}
$$
onde $id $ é a aplicação identidade em $\mathcal{B}^{\dagger}$, $\alpha^{-1}$ é o automorfismo de translação por -1 e  $i : \mathcal{B}^{\dagger} \rightarrow \mathcal{B}^{\dagger}{\rtimes}_{\alpha}\mathbb{Z} $ é a inclusão.   

Esta nova seqüência nos fornecerá mais respostas sobre a K-teoria de $\mathcal{B}^{\dagger}$ e ${\mathcal{B}}^{\diamond}$. Observe que 
$\varphi\circ i : \mathcal{B}^{\dagger}\rightarrow{\mathcal{B}}^{\diamond}$ é a aplicação de inclusão de $\mathcal{B}^{\dagger}$ em ${\mathcal{B}}^{\diamond}$. Denotando novamente por $i$ esta inclusão, podemos reescrever a seqüência exata de Pimsner-Voiculescu como abaixo:
\begin{equation}
\begin{array}{ccccc}
 K_0({\mathcal{B}^{\dagger}})  & \! \stackrel{id_* - \alpha^{-1}_* }{\longrightarrow} & \! 
 K_0({\mathcal{B}^{\dagger}})  & \! \stackrel{i_*}{\longrightarrow} & \!
 K_0({\mathcal{B}^{\diamond}}) \vs  \\     
       \ \uparrow & \! ~  & \! ~  & \! ~ & \! \downarrow \        \vs \\ 
 K_1({\mathcal{B}^{\diamond}}) & \! \stackrel{i_*}{\longleftarrow} & \! 
 K_1({\mathcal{B}^{\dagger}})  & \! \stackrel{id_* - \alpha^{-1}_*}{\longleftarrow} & \!
 K_1({\mathcal{B}^{\dagger}})
\end{array}
\end{equation}

Temos que $id_* - \alpha^{-1}_* \equiv 0$, pois $id$ e $\alpha^{-1}$ são homotópicas. De fato, considere $h:[0,1]\rightarrow Aut(\mathcal{B}^{\dagger})$ dada por $h_t = \alpha^{-t}$. Então $h_0 = id$ e $h_1=\alpha^{-1},$ e $t\mapsto h_t (B) = \alpha^{-t}(B)= B(\cdot +t)$ é contínua em $[0,1]$ para cada $B \in \mathcal{B}^{\dagger}$ pois $\mathcal{B}^{\dagger}\subset C_b(\mathbb{R}, \mathcal{L}_{\mathbb{S}^1}).$ 

Obtemos então, as seguintes seqüências exatas
$$ 0 \rightarrow K_0(\mathcal{B}^{\dagger}) \stackrel{i_*}{\rightarrow} K_0(\mathcal{B}^{\diamond}) \rightarrow K_1(\mathcal{B}^{\dagger}) \rightarrow 0, $$
$$ 0 \rightarrow K_1(\mathcal{B}^{\dagger}) \stackrel{i_*}{\rightarrow} K_1(\mathcal{B}^{\diamond}) \rightarrow K_0(\mathcal{B}^{\dagger}) \rightarrow 0. $$
Estas por sua vez, podem ser reescritas em função de $\mathcal{A}^{\dagger}$ e $\mathcal{A}^{\diamond}$, já que $K_i(\mathcal{B}^{\dagger})\cong K_i(\mathcal{A}^{\dagger})$ e $K_i(\mathcal{B}^{\diamond})\cong K_i(\mathcal{A}^{\diamond}), i=0,1.$ 
$$ 0 \rightarrow K_0(\mathcal{A}^{\dagger}) \stackrel{i_*}{\rightarrow} K_0(\mathcal{A}^{\diamond}) \rightarrow K_1(\mathcal{A}^{\dagger}) \rightarrow 0, ~~~~
 0 \rightarrow K_1(\mathcal{A}^{\dagger}) \stackrel{i_*}{\rightarrow} K_1(\mathcal{A}^{\diamond}) \rightarrow K_0(\mathcal{A}^{\dagger}) \rightarrow 0. $$
Das possibilidades que tínhamos para $K_i({\mathcal{A}^{\dagger}}),$ página \pageref{poss1}, e $K_i({\mathcal{A}^{\diamond}})$, página \pageref{poss2}, $i=0,1,$ as únicas que fazem sentido na seqüência acima são
$$ 
K_0({\mathcal{A}^{\dagger}})\cong \mathbb{Z} ~, ~~K_1({\mathcal{A}^{\dagger}})\cong \mathbb{Z}^2\oplus \mathbb{Z}_{\eta} ~, ~~ 
K_0({\mathcal{A}^{\diamond}}) \cong \mathbb{Z}^3 ~,~~ K_1({\mathcal{A}^{\diamond}}) \cong \mathbb{Z}^3\oplus \mathbb{Z}_{\mu},
$$
com $\eta, \mu$ inteiros positivos não nulos. 

O grupo $K_0({\mathcal{A}^{\dagger}})$ tem como gerador $[Id]_0$, onde $Id$ é o operador identidade de $\mathcal{A}^{\dagger}.$ Assim,  $i_*([Id]_0) = [i(Id)]_0=[Id]_0 \in K_0({\mathcal{A}^{\diamond}}),$ onde $Id$ é o operador identidade de ${\mathcal{A}^{\diamond}}$, e como  $[Id]_0$ é um dos geradores de $K_0({\mathcal{A}^{\diamond}}),$ isso garante que 
$$\frac{K_0(\mathcal{A}^{\diamond})}{K_0(\mathcal{A}^{\dagger})} \cong \mathbb{Z}^2.$$

Logo,
$$ 
K_0({\mathcal{A}^{\dagger}})\cong \mathbb{Z} ~, ~~K_1({\mathcal{A}^{\dagger}})\cong \mathbb{Z}^2 ~, ~~ 
K_i({\mathcal{A}^{\diamond}}) \cong \mathbb{Z}^3 ~,~~ i=0,1,
$$
onde $\eta $ e $\mu$ são iguais a  1.

Deste resultado segue o seguinte corolário.

\begin{corolario}\label{delta0} A aplicação $\delta_0^{\diamond}$ é sobrejetiva.
\end{corolario}

\begin{proof} Como $\mu = 1$ então $Im \delta_0^{\diamond} = K_1(\mathcal{E}^{\diamond}). $ 
\end{proof}

\section{A K-teoria de $\mathcal{A}$}

Nesta última seção, apresentaremos a  K-teoria de $\mathcal{A}$ e de ${\mathcal{A}}/{{\mathcal{K}}_{\Omega}},$ onde ${\mathcal{K}}_{\Omega}$ representa o ideal dos operadores compactos em $L^2(\Omega)$. Construiremos três seqüências exatas envolvendo estas álgebras e o ideal comutador de $\mathcal{A}$, e a partir das seqüências exatas de seis termos em K-teoria, vamos calcular as aplicações de conexão. 

Considere a seqüência exata curta  

\begin{equation}\label{sss}
0 \ \longrightarrow \ \frac{\mathcal{E}_{\mathcal{A}}}{{\mathcal{K}}_{\Omega}} \  
\stackrel{i}{\longrightarrow} \ \frac{\mathcal{A}}{{\mathcal{K}}_{\Omega}} \  
\stackrel{\pi}{\longrightarrow} \ \frac{\mathcal{A}}{\mathcal{E}_{\mathcal{A}}} \ \longrightarrow \ 0, 
\end{equation}
onde {\it i} \'e a inclusão e $\pi$ \'e a projeç\~ao can\^onica. No capítulo 1 vimos que $\mathcal{E}_{\mathcal{A}}$ é o núcleo da aplicação 
$\sigma : \mathcal{A}\rightarrow C({\bf M_{\mathcal{A}}})$ e ${\mathcal{A}}/\mathcal{E}_{\mathcal{A}}$ é isomorfa a $C({\bf M_{\mathcal{A}}})$.

Primeiro vamos determinar os grupos $K_0$ e $K_1$ de $\mathcal{E}_{\mathcal{A}}/{\mathcal{K}}_{\Omega}.$

A aplicação $\Psi$ dada em \ref{psi} nos diz que $\mathcal{E}_{\mathcal{A}}/{\mathcal{K}}_{\Omega}$ é isomorfo a duas cópias de $C(S^{1},{\cal{K}}_{\mathbb{Z}\times \mathbb{S}^1}),$ já que  
$$
\begin{array}{ccc}
 C(S^{1}\times\{-1,+1\},{\cal{K}}_{\mathbb{Z}\times \mathbb{S}^1}) & \! \rightarrow & \!
 C(S^{1},{\cal{K}}_{\mathbb{Z}\times \mathbb{S}^1})\oplus C(S^{1},{\cal{K}}_{\mathbb{Z}\times \mathbb{S}^1}) \\     
 f & \! \mapsto & \! (f(\cdot,-1) , f(\cdot,+1))
\end{array}
$$
é um isomorfismo.

Já sabemos que $K_0(C(S^{1},{\cal{K}}_{\mathbb{Z}\times \mathbb{S}^1}))\cong\mathbb{Z}$ e $K_1(C(S^{1},{\cal{K}}_{\mathbb{Z}\times \mathbb{S}^1}))\cong\mathbb{Z},$ então 

$$K_i\left(\frac{\mathcal{E}_{\mathcal{A}}}{{\mathcal{K}}_{\Omega}}\right)\cong \mathbb{Z}\oplus\mathbb{Z},~i=0,1. $$

Agora partimos para o cálculo dos K-grupos de $\mathcal{A}/\mathcal{E}_{\mathcal{A}}.$ Como ${\bf M_{\mathcal{A}}}$ é o espaço símbolo de $\mathcal{A},$ se conhecermos $K_0$ e $K_1$ de  $C({\bf M_{\mathcal{A}}}),$ usando o $\sigma$-símbolo obteremos o resultado desejado. 

O próximo lema é um resultado de \cite{cintia}, Proposição 1, que nos dá a K-teoria da álgebra $C(X)$, onde $X$ é o subconjunto de $[-\infty,+\infty]\times S^1$  que consiste dos pontos $(t, e^{i\theta})$ tais que $t=\theta$ se $|t|<\infty$.


Sejam as funções $l, ~\tilde{l} : X\rightarrow\mathbb{C}$ definidas da seguinte forma:
$$ 
l(t,e^{i\theta})= 
\left\{\begin{array}{rl}
e^{i\theta},& \mbox{ se }  t \geq 0 \\
1,& \mbox{ se } t < 0
\end{array}\right. 
$$
e
$$
\tilde{l}(t,e^{i\theta})= 
\left\{\begin{array}{rl}
1,& \mbox{ se }  t \geq 0 \\
e^{i\theta},& \mbox{ se } t < 0
\end{array}\right.
$$
(Note que se $|t|<\infty$ então $(t,e^{i\theta})=(t,e^{it})$)

\begin{lema} $K_0(C(X))= \mathbb{Z}[1]_0$ e $K_1(C(X))= \mathbb{Z}[l]_1 \oplus \mathbb{Z}[\tilde{l}]_1.$
\end{lema}

\begin{prop} Temos que $K_0(\mathcal{A}/{\mathcal{E}}_{\mathcal{A}})$ e $K_1(\mathcal{A}/{\mathcal{E}}_{\mathcal{A}})$ são ambos isomorfos a $\mathbb{Z}^6.$
\end{prop}
\begin{proof} Primeiro vamos mostrar que $K_0(C({\bf M_{\mathcal{A}}}))\cong\mathbb{Z}^6$ e  $K_1(C({\bf M_{\mathcal{A}}}))\cong\mathbb{Z}^6.$ 

Com a descrição de ${\bf M_{\mathcal{A}}}$ dada no final do capítulo 1, no caso $\mathbb{B}=\mathbb{S}^1$, podemos escrever ${\bf M_{\mathcal{A}}}$ como $X\times \mathbb{S}^1\times S^1.$ Conseguimos então o seguinte isomorfismo
$$C({\bf M_{\mathcal{A}}})\cong C(S^1, C(\mathbb{S}^1, C(X))).$$ 
Esta descrição de $C({\bf M_{\mathcal{A}}})$ será fundamental para os nossos cálculos. Inicialmente, vamos determinar a K-teoria de $C(\mathbb{S}^1, C(X))$ para depois determinar a de $C(S^1, C(\mathbb{S}^1, C(X))).$

Lembrando que a seqüência exata 
\begin{center}
\begin{picture}(440,17)(0,0)
\put(85,0){$0$}\put(100,5){\vector(1,0){20}}
\put(126,0){$SC(X)$}\put(165,5){\vector(1,0){20}}\put(170,10){${i}$}
\put(190,0){$C(\mathbb{S}^1, C(X))$}\put(258,6){\vector(1,0){20}}\put(266,12){${p}$}
\put(288,0){$C(X)$}\put(278,2){\vector(-1,0){20}}\put(266,-6){${s}$}
\put(318,5){\vector(1,0){20}} \put(345,0){0,}
\end{picture}
\end{center}
cinde, então vale que 
$$K_i(C(\mathbb{S}^1, C(X)))= i_{\ast}(K_i(SC(X)))\oplus s_{\ast}(K_i(C(X))), ~i=0,1.$$

\begin{itemize}
\item  $K_0(C(\mathbb{S}^1, C(X))):$

Já sabemos que $K_0(C(X))=\mathbb{Z}[1]_0,$ então $s_{\ast}([1]_0)=[x\mapsto 1]_0 \in K_0(C(\mathbb{S}^1, C(X))).$ 

Conhecemos o isomorfismo entre $K_1(C(X))$ e $K_0(SC(X)).$ 
Temos então que 
$$\theta_{C(X)}([l]_1) = [x=a+ib \mapsto Q(l,a,b)]_0 ~, ~~ \theta_{C(X)}([\tilde{l}]_1) = [x=a+ib \mapsto Q(\tilde{l},a,b)]_0. $$
Logo 
\begin{equation}
K_0(C(\mathbb{S}^1, C(X))) = \mathbb{Z}[1]_0 \oplus \mathbb{Z}[a+ib \mapsto Q(l,a,b)]_0 \oplus \mathbb{Z}[a+ib \mapsto Q(\tilde{l},a,b)]_0 
\label{k0}
\end{equation}

\item  $K_1(C(\mathbb{S}^1, C(X))):$

Sabendo que $K_1(C(X))=\mathbb{Z}[l]_1\oplus\mathbb{Z}[\tilde{l}]_1,$ então 
$$s_{\ast}(K_1(C(X)))=\mathbb{Z}[x\mapsto l]_1\oplus\mathbb{Z}[x \mapsto \tilde{l}]_1 .$$

Vamos usar a aplicação de Bott para conhecer $K_1(SC(X)).$  
$$\beta_{C(X)}([1]_0) = [\texttt{x}]_1 ,$$
onde $\texttt{x} $ é a função identidade $\mathbb{S}^1$  

Temos então que 
\begin{equation}
K_1(C(\mathbb{S}^1, C(X))) = \mathbb{Z}[\texttt{x}]_1 \oplus \mathbb{Z}[l]_1 \oplus \mathbb{Z}[\tilde{l}]_1 .
\label{k1}
\end{equation}
\end{itemize}

Conhecendo os grupos dados em (\ref{k0}) e (\ref{k1}), podemos partir para a próxima seqüência exata cindida:  
\begin{center}
\begin{picture}(440,17)(0,0)
\put(40,0){$0$}\put(50,5){\vector(1,0){20}}
\put(80,0){$SC(S^1,C(X))$}\put(155,5){\vector(1,0){20}}\put(165,10){${i}$}
\put(180,0){$C(S^1,C(\mathbb{S}^1, C(X)))$}\put(280,6){\vector(1,0){20}}\put(288,12){${p}$}
\put(310,0){$C(\mathbb{S}^1,C(X))$}\put(300,2){\vector(-1,0){20}}\put(288,-6){${s}$}
\put(375,5){\vector(1,0){20}} \put(400,0){0,}
\end{picture}
\end{center}
Daí,   
$$K_i(C(S^1,C(\mathbb{S}^1, C(X))))= i_{\ast}(K_i(SC(S^1,C(X))))\oplus s_{\ast}(K_i(C(S^1,C(X)))), ~i=0,1.$$

\begin{itemize}
\item  $K_0(C(S^1(C(\mathbb{S}^1, C(X))))):$

Chamemos $Y:=C(\mathbb{S}^1,C(X)).$ Comecemos calculando $s_{\ast}(K_0(Y)):$
$$ s_{\ast}([1]_0) = [w \mapsto 1]_0 ~ , ~~s_{\ast}([a+ib \mapsto Q(l,a,b)]_0) = [w\mapsto (a+ib \mapsto Q(l,a,b))]_0 ,$$
$$ s_{\ast}([a+ib \mapsto Q(\tilde{l},a,b)]_0) = [w \mapsto (a+ib \mapsto Q(\tilde{l},a,b))]_0 $$

Para calcular $i_{\ast}(K_0(SY)),$ usaremos novamente o isomorfismo descrito no teorema \ref{theta}, \linebreak $\theta_Y : K_1(Y) \rightarrow K_0(SY):$
$$ \theta_Y([l]_1) = [w=u+iv  \mapsto Q(l,u,v)]_0~,~~\theta_Y([\tilde{l}]_1) = [w=u+iv \mapsto Q(\tilde{l},u,v)]_0 , $$
$$ \theta_Y([\texttt{x}]_1) = [w=u+iv \mapsto Q(\texttt{x},u,v)]_0 .$$

Fazendo $ m = (u+iv, x= a+ib, (t, e^{i\theta})) \in S^1\times \mathbb{S}^1 \times X$, temos 
\begin{eqnarray}
K_0(C(S^1\times \mathbb{S}^1 \times X)) & = & \mathbb{Z}[m\mapsto 1]_0 \oplus \mathbb{Z}[m \mapsto Q(l(t,e^{i\theta}),a,b)]_0 \oplus \mathbb{Z}[ m \mapsto Q(\tilde{l}(t, e^{i\theta}),a,b)]_0  \nonumber \\ 
& & \oplus \mathbb{Z}[m \mapsto Q(l(t, e^{i\theta}),u,v)]_0 \oplus \mathbb{Z}[m \mapsto Q(\tilde{l}(t, e^{i\theta}),u,v)]_0 \nonumber \\
& & \oplus \mathbb{Z}[m \mapsto Q(x,u,v)]_0
\label{k0ma}
\end{eqnarray}

\item $K_1(C(S^1(C(\mathbb{S}^1, C(X)))))$

Chamando ainda de $Y$ o conjunto $C(\mathbb{S}^1, C(X)),$ conhecemos em (\ref{k1}) o grupo $K_1(Y)$. Assim, podemos calcular facilmente $s_{\ast}(K_1(Y)).$ 
$$
s_{\ast}(K_1(Y))=\mathbb{Z}[w \mapsto \texttt{x}]_1 \oplus \mathbb{Z}[w \mapsto l]_1 \oplus \mathbb{Z}[w \mapsto \tilde{l})]_1 
$$

Para conhecermos $K_1(SY)$, calcularemos a aplicação de Bott nos geradores do grupo $K_0(Y)$. 
$$
\beta_{Y}([1]_0) = [\texttt{w}]_1 ,~ \mbox{com} ~ \texttt{w}(w) = w \cdot 1 + (1 - 1) = w
$$
$$
\beta_{Y}([a+ib \mapsto Q(l,a,b)]_0) = [ w  \mapsto 1_2 + (w - 1)(a+ib \mapsto Q(l,a,b))]_1
$$
$$
\beta_{Y}([a+ib \mapsto Q(\tilde{l},a,b)]_0) = [w \mapsto 1_2 + (w - 1)(a+ib \mapsto Q(\tilde{l},a,b))]_1
$$

Temos então, para $ m = (w=u+iv, x= a+ib, (t, e^{i\theta})) \in S^1\times \mathbb{S}^1 \times X$,
\begin{eqnarray}
K_1(C(S^1 \times \mathbb{S}^1 \times X)) & = & \mathbb{Z}[m \mapsto x]_1 \oplus \mathbb{Z}[m \mapsto l(t, e^{i\theta})]_1 \oplus \mathbb{Z}[m \mapsto \tilde{l}(t, e^{i\theta})]_1 \oplus \mathbb{Z}[m \mapsto w]_1 \nonumber \\
& &\oplus \mathbb{Z}[m \mapsto 1+(w-1)Q(l(t, e^{i\theta}),a,b)]_1 \\
& & \oplus \mathbb{Z}[m \mapsto 1+(w-1)Q(\tilde{l}(t, e^{i\theta}),a,b)]_1 \nonumber 
\label{k1ma}
\end{eqnarray}
\end{itemize} 

Como $\mathcal{A}/\mathcal{E}_{\mathcal{A}} \cong C({\bf M_{\mathcal{A}}}) \cong C(S^1\times \mathbb{S}^1\times X),$ já sabemos que 
$$ 
K_0(\mathcal{A}/{\mathcal{E}}_{\mathcal{A}})\cong \mathbb{Z}^6 ~~\mbox{e}~~ K_1(\mathcal{A}/{\mathcal{E}}_{\mathcal{A}})\cong \mathbb{Z}^6.
$$
Vamos agora determinar os geradores destes grupos usando o $\sigma$-símbolo. Para cada função definida em $C(S^1\times \mathbb{S}^1\times X),$ acharemos a sua respectiva pré-imagem em $\mathcal{A}/\mathcal{E}_{\mathcal{A}}.$ 
$$\sigma_{I}=1~, ~~\sigma_{l(M_t)}= l ~, ~~\sigma_{\tilde{l}(M_t)}= \tilde{l},$$
onde $I$ é o operador identidade e $l(M_t)$ e $\tilde{l}(M_t)$ são operadores de multiplicação pelas funções $l$ e $\tilde{l}$, respectivamente. 
Lembrando que no teorema \ref{ma} a ordem das variáveis apresentada em ${\bf M_{\mathcal{A}}}$ é 
$(t,x,w,e^{i\theta})\in [-\infty,+\infty] \times \mathbb{S}^1 \times S^1 \times S^1,$ temos
\begin{itemize}
\item[i.] $\sigma_{\texttt{x}(M_x)}((t,x,w,e^{i\theta}))=x,$ onde $(\texttt{x}(M_x)u)(t,x)= x\cdot u(t,x)$;
\item[ii.] $\sigma_{\texttt{a}(M_x)}((t,x=a+ib,w,e^{i\theta}))= a,$ onde $(\texttt{a}(M_x)u)(t,x)= a\cdot u(t,x)$, isto é, operador de multiplicação pela parte real de $x$;   
\item[iii.] $\sigma_{\texttt{b}(M_x)}((t,x=a+ib,w,e^{i\theta}))= b,$ onde $(\texttt{b}(M_x)u)(t,x)= b\cdot u(t,x)$, isto é, operador de multiplicação pela parte imaginária de $x$;   
\item[iv.] $\sigma_{A_5}((t,x,w=u+iv,e^{i\theta})) = u;$  
\item[v.] $\sigma_{A_6}((t,x,w=u+iv,e^{i\theta})) = v.$
\end{itemize}

Assim, os geradores de $K_0(\mathcal{A}/{\mathcal{E}}_{\mathcal{A}})$ e $K_1(\mathcal{A}/{\mathcal{E}}_{\mathcal{A}})$ são 
\begin{eqnarray*}
K_0(\mathcal{A}/{\mathcal{E}}_{\mathcal{A}}) & = & \mathbb{Z}[[I]_{\mathcal{E}}]_0 \oplus \mathbb{Z}[[Q(l(M_t),\texttt{a}(M_x),\texttt{b}(M_x))]_{\mathcal{E}}]_0 \oplus \mathbb{Z}[[Q(\tilde{l}(M_t),\texttt{a}(M_x),\texttt{b}(M_x))]_{\mathcal{E}}]_0 \oplus \nonumber \\
& &  \mathbb{Z}[[Q(\texttt{x}(M_x), A_5, A_6)]_{\mathcal{E}}]_0 
\oplus \mathbb{Z}[[Q(l(M_t),A_5,A_6)]_{\mathcal{E}}]_0 \oplus 
\mathbb{Z}[[Q(\tilde{l}(M_t),A_5,A_6)]_{\mathcal{E}}]_0
\label{k0ae}
\end{eqnarray*}
e
\begin{eqnarray*}
K_1(\mathcal{A}/{\mathcal{E}}_{\mathcal{A}}) & = & \mathbb{Z}[[\texttt{x}(M_x)]_{\mathcal{E}}]_1 \oplus \mathbb{Z}[[l(M_t)]_{\mathcal{E}}]_1 \oplus 
\mathbb{Z}[[\tilde{l}(M_t)]_{\mathcal{E}}]_1 \oplus \mathbb{Z}[[A_5+iA_6]_{\mathcal{E}}]_1 \oplus  \nonumber \\
& & \mathbb{Z}[[Id + (A_5+iA_6 - 1)Q(l(M_t),\texttt{a}(M_x),\texttt{b}(M_x))]_{\mathcal{E}}]_1 \oplus \nonumber \\
& & \mathbb{Z}[[Id + (A_5+iA_6 - 1)Q(\tilde{l}(M_t),\texttt{a}(M_x),\texttt{b}(M_x))]_{\mathcal{E}}]_1
\label{k1ae}
\end{eqnarray*}

\end{proof}

Da seq\"u\^encia e\-xa\-ta curta  (\ref{sss}), temos a correspondente seq\"u\^encia e\-xa\-ta  de seis termos em K-teoria: 
\begin{equation}\label{sss6}
\begin{array}{ccccc}
 \mathbb{Z}^2 \cong K_0(\mathcal{E}_{\mathcal{A}}/{\mathcal{K}}_{\Omega})  & \! \stackrel{i_*}{\longrightarrow} & \! 
 K_0(\mathcal{A}/{\mathcal{K}}_{\Omega})  & \! \stackrel{{\pi}_*}{\longrightarrow} & \!
 K_0(\mathcal{A}/{\mathcal{E}}_{\mathcal{A}}) \cong \mathbb{Z}^6   \\ \\    
 {\delta}_1 \ \uparrow & \! ~  & \! ~  & \! ~ & \! \downarrow \ {\delta}_0  \\ \\
 \mathbb{Z}^6 \cong K_1(\mathcal{A}/{\mathcal{E}}_{\mathcal{A}}) & \! \stackrel{{\pi}_*}{\longleftarrow} & \! 
 K_1(\mathcal{A}/{\cal{K}}_{\Omega})  & \! \stackrel{{i}_*}{\longleftarrow} & \!
 K_1(\mathcal{E}_{\mathcal{A}}/{\cal{K}}_{\Omega}) \cong \mathbb{Z}^2
\end{array}
\end{equation}
Nossos esforços se concentram em descobrir $K_0(\mathcal{A}/{\cal{K}}_{\Omega})$ e $K_1(\mathcal{A}/{\cal{K}}_{\Omega}).$ Para isto, devemos calcular  a aplicação exponencial ${\delta}_0$ e a aplicação do índice ${\delta}_1$. 


Novamente aqui, a descrição que temos do isomorfismo $K_1(C(S^1, \mathcal{K}_{\mathbb{Z}\times\mathbb{S}^1}))\cong \mathbb{Z}$ não nos permite determinar exatamente a imagem de $\delta_0$. Apesar disso, temos a seguinte proposição que nos garante que $\delta_0$ é diferente de zero.

\begin{prop} A aplicação $\delta_0$ em (\ref{sss6}) é não nula. Além disso, via o isomorfismo 
$$
K_1\left(\frac{\mathcal{E}_{\mathcal{A}}}{\mathcal{K}_{\Omega}}\right) \rightarrow K_1(C(S^1\times\{-1,+1\}, \mathcal{K}_{\mathbb{Z}\times\mathbb{S}^1})) \rightarrow K_0(\mathcal{K}_{\mathbb{Z}\times\mathbb{S}^1})\oplus K_0(\mathcal{K}_{\mathbb{Z}\times\mathbb{S}^1}) \rightarrow \mathbb{Z}\oplus \mathbb{Z},
$$
o elemento $(1,1)$ pertence à imagem de $\delta_0$. Isto é, existe um elemento $[[T]_{\mathcal{K}}]_1 \in Im \delta_0$ tal que este corresponde ao elemento $(1,1) \in \mathbb{Z}\oplus \mathbb{Z}$. 
\end{prop}

\begin{proof} Pelo corolário \ref{delta0} sabemos que $\delta_0^{\diamond} : K_0(\mathcal{A}^{\diamond}/\mathcal{E}^{\diamond}) \rightarrow K_1(\mathcal{E}^{\diamond})$ é sobrejetora. Portanto existe um elemento $[[A]_{\mathcal{E}^{\diamond}}]_0 \in K_0(\mathcal{A}^{\diamond}/\mathcal{E}^{\diamond})$ tal que $\delta_0([[A]_{\mathcal{E}^{\diamond}}]_0) \in K_1(\mathcal{E}^{\diamond})$ corresponde a 1 pelo isomorfismo
$$ 
K_1(\mathcal{E}^{\diamond}) \cong K_1(C(S^1, \mathcal{K}_{\mathbb{Z}\times\mathbb{S}^1})) \cong K_0(\mathcal{K}_{\mathbb{Z}\times\mathbb{S}^1}) \cong \mathbb{Z}.
$$

Conhecemos os geradores de $K_0(\mathcal{A}^{\diamond}/\mathcal{E}^{\diamond})$ e sabemos que $\delta_0^{\diamond}$ calculado em $[[I]_{\mathcal{E}^{\diamond}}]_0$ e em \linebreak $[[Q(e^{iM_t}, a(M_x), b(M_x))]_{\mathcal{E}^{\diamond}}]_0$ é zero, pois $I$ e $Q(e^{iM_t}, a(M_x), b(M_x))$ são projeções em $P_{\infty}(\tilde{\mathcal{A}^{\diamond}})$. Os outros dois geradores são $[[Q(e^{iM_t}, A_5, A_6)]_{\mathcal{E}^{\diamond}}]_0$ e $[[Q(\texttt{x}(M_x), A_5, A_6)]_{\mathcal{E}^{\diamond}}]_0.$ 

Assim, podemos escrever $[[A]_{\mathcal{E}^{\diamond}}]_0$ como sendo a soma 
$$ \alpha_1 [[I]_{\mathcal{E}^{\diamond}}]_0 + \alpha_2 [[Q(e^{iM_t}, a(M_x), b(M_x))]_{\mathcal{E}^{\diamond}}]_0 + 
\alpha_3 [[Q(e^{iM_t}, A_5, A_6)]_{\mathcal{E}^{\diamond}}]_0 + \alpha_4 [[Q(\texttt{x}(M_x), A_5, A_6)]_{\mathcal{E}^{\diamond}}]_0,$$
com $\alpha_i \in \mathbb{Z}, ~i=1,2,3,4.$
Calculando  $\delta_0^{\diamond} $ em $[[A]_{\mathcal{E}^{\diamond}}]_0$  temos que 
$$
\delta_0^{\diamond}([[A]_{\mathcal{E}^{\diamond}}]_0) = \delta_0^{\diamond} (\alpha_3 [[Q(e^{iM_t}, A_5, A_6)]_{\mathcal{E}^{\diamond}}]_0 + \alpha_4 [[Q(\texttt{x}(M_x), A_5, A_6)]_{\mathcal{E}^{\diamond}}]_0) $$
pois nos outros geradores, $\delta_0^{\diamond}$ se anula. 
Podemos então assumir que $[[A]_{\mathcal{E}^{\diamond}}]_0 = \alpha[[Q(e^{iM_t}, A_5, A_6)]_{\mathcal{E}^{\diamond}}]_0 + \beta [[Q(\texttt{x}(M_x), A_5, A_6)]_{\mathcal{E}^{\diamond}}]_0,$ com $\alpha$ e $\beta$ inteiros não nulos simultaneamente.  

Os geradores de $K_0(\mathcal{A}/\mathcal{E}_{\mathcal{A}})$, nos quais não calculamos $\delta_0$, são
$$ 
[[Q(l(M_t), A_5, A_6)]_{\mathcal{E}_{\mathcal{A}}}]_0 ~, ~~[[Q(\tilde{l}(M_t), A_5, A_6)]_{\mathcal{E}_{\mathcal{A}}}]_0~, ~~
[[Q(\texttt{x}(M_x), A_5, A_6)]_{\mathcal{E}_{\mathcal{A}}}]_0 .
$$

Dados os isomorfismos $\Psi : \mathcal{E}_{\mathcal{A}}/ \mathcal{K}_{\Omega}\rightarrow C(S^1\times\{-1,+1\}, \mathcal{K}_{\mathbb{Z}\times\mathbb{S}^1})$ e $\gamma'|_{\mathcal{E}^{\diamond}} : \mathcal{E}^{\diamond}\rightarrow C(S^1, \mathcal{K}_{\mathbb{Z}\times\mathbb{S}^1})$, sabemos que estes induzem isomorfismos $K_1(\Psi)$ e $K_1(\gamma'|_{\mathcal{E}^{\diamond}})$ nos grupos. Queremos mostrar que 
$$
K_1(\Psi)\circ \delta_0 ([[Q(l(M_t), A_5, A_6)]_{\mathcal{E}_{\mathcal{A}}}]_0) = (0 , K_1(\gamma'|_{\mathcal{E}^{\diamond}})\circ \delta_0^{\diamond}([[Q(e^{iM_t}, A_5, A_6)]_{\mathcal{E}^{\diamond}}]_0) ) $$

Pela definição da aplicação exponencial, dado $[T]_{\mathcal{E}^{\diamond}}\in P_n(({A^{\diamond}/\mathcal{E}^{\diamond}})\tilde{}~),$ 
$$ \delta_0^{\diamond}([[T]_{\mathcal{E}^{\diamond}}]_0) = - \left[exp{2\pi i \left(\frac{T+T^{\ast}}{2}\right)}\right]_1 ,$$
já que ${(T+T^{\ast})}/{2}$ é uma pré-imagem auto-adjunta em $ M_n(\tilde{\mathcal{A}^{\diamond}})$ de $[T]_{\mathcal{E}^{\diamond}}.$ Então, 
$$ \delta_0^{\diamond}([[Q(e^{iM_t}, A_5, A_6)]_{\mathcal{E}^{\diamond}}]_0) = 
- \left[exp{2\pi i \left(\frac{Q(e^{iM_t}, A_5, A_6)+Q(e^{iM_t}, A_5, A_6)^{\ast}}{2}\right)}\right]_1 $$
e calculando $K_1(\gamma'|_{\mathcal{E}^{\diamond}})$ neste elemento, temos
$$
K_1(\gamma'|_{\mathcal{E}^{\diamond}})\left( - \left[exp{2\pi i \left(\frac{Q(e^{iM_t}, A_5, A_6)+Q(e^{iM_t}, A_5, A_6)^{\ast}}{2}\right)}\right]_1\right)= 
$$
$$
- \left[\gamma'_{exp{2\pi i(\frac{Q(e^{iM_t}, A_5, A_6)+Q(e^{iM_t}, A_5, A_6)^{\ast}}{2})}}\right]_1 = 
- \left[exp{2\pi i \gamma'_{(\frac{Q(e^{iM_t}, A_5, A_6)+Q(e^{iM_t}, A_5, A_6)^{\ast}}{2})}}\right]_1 . 
$$

A aplicação $\gamma'$ é da forma $\gamma_{A}'(z)=\gamma_A (z,\pm 1)$ para todo $z\in S^1$. Assim, $\gamma'_{Q(e^{iM_t}, A_5, A_6)} $ é igual a matriz  \\
$
\left(\begin{array}{cc}
1 - (2- \gamma_{e^{iM_t}}'-{\gamma_{e^{iM_t}}'}^{\ast})\frac{1-\gamma_{A_5}'}{8}  & 
\frac{{\gamma_{e^{iM_t}}'}^2-1}{2}\sqrt{\frac{1-\gamma_{A_5}'}{8}} -(\gamma_{e^{iM_t}}'-1)^2 \frac{\gamma_{A_6}'}{8} \vspace{0.3cm}\\
\frac{{{\gamma_{e^{iM_t}}'}^{\ast}}^{2}-1}{2}\sqrt{\frac{1-\gamma_{A_5}'}{8}} -({\gamma_{e^{iM_t}}'}^{\ast}-1)^2 \frac{\gamma_{A_6}'}{8} &
(2-\gamma_{e^{iM_t}}'-{\gamma_{e^{iM_t}}'}^{\ast})\frac{1-\gamma_{A_5}'}{8}  \\
\end{array}\right) =
$ \\
$$ 
\left(\begin{array}{cc}
1 - (2- Y_{-1}-Y_{1})\frac{1-\gamma_{A_5}'}{8}  & 
\frac{{Y_{-1}}^2-1}{2}\sqrt{\frac{1-\gamma_{A_5}'}{8}} -(Y_{-1}-1)^2 \frac{\gamma_{A_6}'}{8} \vspace{0.3cm}\\
\frac{{Y_{1}}^{2}-1}{2}\sqrt{\frac{1-\gamma_{A_5}'}{8}} -(Y_{1}-1)^2 \frac{\gamma_{A_6}'}{8} &
(2-Y_{-1}-Y_{1})\frac{1-\gamma_{A_5}'}{8}  \\
\end{array}\right) = Q(Y_{-1}, \gamma_{A_5}', \gamma_{A_6}')
$$

Portanto,
$$
K_1(\gamma'|_{\mathcal{E}^{\diamond}})\circ \delta_0^{\diamond}([[Q(e^{iM_t}, A_5, A_6)]_{\mathcal{E}^{\diamond}}]_0) =  
- \left[exp{2\pi i \left(\frac{Q(Y_{-1}, \gamma_{A_5}', \gamma_{A_6}')+ Q(Y_{-1}, \gamma_{A_5}', \gamma_{A_6}')^{*}}{2}\right)}\right]_1 .
$$ 

Considere agora a aplicação $\delta_0$ calculada em $[[Q(l(M_t), A_5, A_6)]_{\mathcal{E}_{\mathcal{A}}}]_0$. Pela definição temos
$$ \delta_0([[Q(l(M_t), A_5, A_6)]_{\mathcal{E}_{\mathcal{A}}}]_0) = - \left[\left[exp{2\pi i \left(\frac{Q(l(M_t), A_5, A_6)+ Q(l(M_t), A_5, A_6)^{\ast}}{2}\right)}\right]_{\mathcal{K}_{\Omega}}\right]_1 .$$

O isomorfismo $\Psi : \mathcal{E}_{\mathcal{A}}/ \mathcal{K}_{\Omega}\rightarrow C(S^1\times\{-1,+1\}, \mathcal{K}_{\mathbb{Z}\times\mathbb{S}^1})$ é tal que $\Psi ([A]_{\mathcal{K}}) = \gamma_{A}$ para todo $A\in \mathcal{E}_{\mathcal{A}}. $ Então, 
$$
K_1(\Psi)\left( - \left[\left[exp{2\pi i \left(\frac{Q(l(M_t), A_5, A_6)+ Q(l(M_t), A_5, A_6)^{\ast}}{2}\right)}\right]_{\mathcal{K}_{\Omega}}\right]_1 \right)= 
$$
$$
- \left[\gamma_{exp{2\pi i(\frac{Q(l(M_t), A_5, A_6)+Q(l(M_t), A_5, A_6)^{\ast}}{2})}}\right]_1 = 
- \left[exp{2\pi i \gamma_{(\frac{Q(l(M_t), A_5, A_6)+Q(l(M_t), A_5, A_6)^{\ast}}{2})}}\right]_1 . 
$$
Como $l(-\infty, e^{i\theta})=1$, então $\gamma_{l(M_t)}(z,-1)= l(-\infty, Y_{-1}) = I $ para todo $z\in S^1$ e portanto temos  
\begin{eqnarray}
\gamma_{Q(l(M_t), A_5, A_6)}(z, -1) & = & Q(\gamma_{l(M_t)}(z,-1) , \gamma_{A_5}(z,-1), \gamma_{A_6}(z,-1)) \nonumber \\
& = & Q( I , \gamma_{A_5}(z,-1), \gamma_{A_6}(z,-1)) = \left(\begin{array}{cc}
I & 0 \\
0 & 0 \\  
\end{array}\right) \nonumber
\end{eqnarray}
para todo $z\in S^1$. Logo 
$$exp{2\pi i \gamma_{(\frac{Q(l(M_t), A_5, A_6)+Q(l(M_t), A_5, A_6)^{\ast}}{2})}(z,-1)} = exp 2\pi i \left(\begin{array}{cc}
I & 0 \\
0 & 0 \\
\end{array}\right) = \left(\begin{array}{cc}
I & 0 \\
0 & I \\
\end{array}\right). $$ 
para todo $z\in S^1.$
Agora, $\gamma_{l(M_t)}(z,+1)= l(+\infty, Y_{-1}) = Y_{-1} $ para todo $z$, então 
$$  \gamma_{Q(l(M_t), A_5, A_6)}(z, +1) = Q(Y_{-1}, \gamma_{A_5} (z,+1), \gamma_{A_6} (z,+1)). $$
Isto implica  que $exp{2\pi i \gamma_{(\frac{Q(l(M_t), A_5, A_6)+Q(l(M_t), A_5, A_6)^{\ast}}{2})}(z,+1)}$ é igual a 
$$
exp 2\pi i \left(\frac{Q(Y_{-1}, \gamma_{A_5} (z,+1), \gamma_{A_6} (z,+1)) + Q(Y_{-1}, \gamma_{A_5} (z,+1), \gamma_{A_6} (z,+1))^{*}}{2}\right).
$$
para todo $z\in S^1$. 

Como $\gamma_{A_5}'(z)= \gamma_{A_5}(z,+1)$ e $\gamma_{A_6}'(z)= \gamma_{A_6}(z,+1)$ pra todo $z \in S^1$, podemos concluir que 
\begin{eqnarray}
K_1(\Psi) \circ \delta_0([[Q(l(M_t), A_5, A_6)]_{\mathcal{E}_{\mathcal{A}}}]_0) & = &  
\left(
\left[\left(\begin{array}{cc}
I & 0 \\
0 & I \\
\end{array}\right)\right]_1 , 
K_1(\gamma ') \circ \delta_0^{\diamond}([[Q(e^{iM_t}, A_5, A_6)]_{\mathcal{E}^{\diamond}}]_0)
\right) \nonumber \\
& = &\left(
0 , 
K_1(\gamma ') \circ \delta_0^{\diamond}([[Q(e^{iM_t}, A_5, A_6)]_{\mathcal{E}^{\diamond}}]_0)
\right) \nonumber
\end{eqnarray}

De forma análoga, podemos concluir que 
$$ K_1(\Psi)\circ \delta_0([[Q(\tilde{l}(M_t), A_5, A_6)]_{\mathcal{E}_{\mathcal{A}}}]_0) =  (K_1(\gamma'|_{\mathcal{E}^{\diamond}})\circ \delta_0^{\diamond}([[Q(e^{iM_t}, A_5, A_6)]_{\mathcal{E}^{\diamond}}]_0) , 0). $$ 

Vamos calcular agora $K_1(\Psi) \circ \delta_0$ em $[[Q(\texttt{x}(M_x), A_5, A_6)]_{\mathcal{E}_{\mathcal{A}}}]_0.$ Como $\gamma_{\texttt{x}(M_x)}(z,\pm1) = \gamma_{\texttt{x}(M_x)}'(z),$ 
$\gamma_{A_5}(z,\pm1) = \gamma_{A_5}'(z)$ e $\gamma_{A_6}(z,\pm1) = \gamma_{A_6}'(z)$ para todo $z\in S^1$, e seguindo o mesmo raciocínio de cálculos feito acima, temos
$$ K_1(\Psi) \circ \delta_0([[Q(\texttt{x}(M_x), A_5, A_6)]_{\mathcal{E}_{\mathcal{A}}}]_0) = $$
$$(K_1(\gamma ') \circ \delta_0^{\diamond}([[Q(\texttt{x}(M_x), A_5, A_6)]_{\mathcal{E}^{\diamond}}]_0 ), 
  K_1(\gamma ') \circ \delta_0^{\diamond}([[Q(\texttt{x}(M_x), A_5, A_6)]_{\mathcal{E}^{\diamond}}]_0)) .
$$

Podemos então concluir que, para 
$$[[B]_{\mathcal{E}_{\mathcal{A}}}]_0 = \alpha[[Q(l(M_t), A_5, A_6)]_{\mathcal{E}_{\mathcal{A}}}]_0 + \alpha[[Q(\tilde{l}(M_t), A_5, A_6)]_{\mathcal{E}_{\mathcal{A}}}]_0 + \beta [[Q(\texttt{x}(M_x), A_5, A_6)]_{\mathcal{E}_{\mathcal{A}}}]_0
$$
temos
$$ K_1(\Psi) \circ \delta_0([[B]_{\mathcal{E}_{\mathcal{A}}}]_0) = (K_1(\gamma ') \circ \delta_0^{\diamond}([[A]_{\mathcal{E}^{\diamond}}]_0 ),
K_1(\gamma ') \circ \delta_0^{\diamond}([[A]_{\mathcal{E}^{\diamond}}]_0 )),$$
para $[[A]_{\mathcal{E}^{\diamond}}]_0 = \alpha[[Q(e^{iM_t}, A_5, A_6)]_{\mathcal{E}^{\diamond}}]_0 + \beta [[Q(\texttt{x}(M_x), A_5, A_6)]_{\mathcal{E}^{\diamond}}]_0.$ Logo, o elemento $\delta_0([[B]_{\mathcal{E}_{\mathcal{A}}}]_0)$ corresponde ao par $(1,1) \in \mathbb{Z}^2.$
\end{proof}

Da proposição anterior, podemos concluir também que se $\delta_0^{\diamond}([[Q(e^{iM_t}, A_5, A_6)]_{\mathcal{E}^{\diamond}}]_0)$ corresponde a um inteiro $\nu$, então $ \delta_0([[Q(l(M_t), A_5, A_6)]_{\mathcal{E}_{\mathcal{A}}}]_0) \leftrightarrow (0,\nu)$ e $ \delta_0([[Q(\tilde{l}(M_t), A_5, A_6)]_{\mathcal{E}_{\mathcal{A}}}]_0) \leftrightarrow (\nu,0).$ 

Partimos agora para o cálculo de $\delta_1: K_1(\mathcal{A}/{\mathcal{E}}_{\mathcal{A}})\rightarrow K_0(\mathcal{E}_{\mathcal{A}}/{\cal{K}}_{\Omega})$. A proposição 3 de \cite{cintia} pode ser novamente reescrita a esta álgebra, em que $\mathcal{A}\subset \mathcal{L}(L^2(\mathbb{R}\times\mathbb{S}^1))$. 

\begin{prop}\label{delta1} Seja $A$ em $\mathcal{A}$ tal que $[A]_{\mathcal{E}_{\mathcal{A}}}$ seja inversível em $\mathcal{A}/{\mathcal{E}_{\mathcal{A}}}.$ 
Então ${\gamma}_{A}(z,-1)$ e ${\gamma}_{A}(z,+1)$ são operadores de Fredholm em $\mathcal{L}_{\mathbb{Z}\times \mathbb{S}^1}$ para todo $z \in S^1$ e
$$ 
{\delta}_1(\left[[A]_{\mathcal{E}_{\mathcal{A}}}]_1\right) =
(\mathsf{ind}({\gamma}_{A}(1,-1))[E]_0,\mathsf{ind}({\gamma}_{A}(1,+1))[E]_0)
\in K_0({\cal{K}}_{\mathbb{Z}\times \mathbb{S}^1})\oplus K_0({\cal{K}}_{\mathbb{Z}\times \mathbb{S}^1}),
$$
onde $\mathsf{ind}$ denota o índice de Fredholm e $E$ é uma projeção de posto 1 em ${\mathcal{K}}_{\mathbb{Z}\times \mathbb{S}^1}$.
\end{prop}

\begin{proof}
A demonstração desta proposição é análoga à feita em \ref{delta1diam}, já que $\mathcal{E}_{\mathcal{A}}/\mathcal{K}_{\Omega}\cong C(S^1\times \{-1,+1\},{\cal{K}}_{\mathbb{Z}\times \mathbb{S}^1}) \cong  \mathcal{E}^{\diamond}\oplus \mathcal{E}^{\diamond}.$ 
\end{proof}

\begin{teo}\label{sobrejetora} A aplicação ${\delta}_1$ em (\ref{sss6}) é sobrejetora. 
\end{teo}

\begin{proof} Calculemos ${\delta}_1$ nos geradores de $K_1(\mathcal{A}/{\mathcal{E}}_{\mathcal{A}}).$ É fácil ver que os operadores de multiplicação
$$
\texttt{x}(M_x)~, ~~l(M_t)~, ~~\tilde{l}(M_t)
$$
são unitários em $\mathcal{A}$, e portanto unitários em $\mathcal{A}/\mathcal{K}_{\Omega}.$ Logo, 
$$
{\delta}_1([[\texttt{x}(M_x)]_{\mathcal{E}}]_1) = {\delta}_1([[l(M_t)]_{\mathcal{E}}]_1) = {\delta}_1([[\tilde{l}(M_t)]_{\mathcal{E}}]_1) = (0,0).
$$

Para calcular ${\delta}_1$ nos outros geradores, usaremos o resultado da proposição \ref{delta1} . Calculando $\gamma_{A_5+iA_6}(1,\pm 1)$, obtemos o seguinte operador em $\mathcal{L}_{\mathbb{Z}\times \mathbb{S}^1}$ de multiplicação pela sequência 
\begin{equation}
\gamma_{A_5+iA_6}(1,\pm 1) = 
\left(j\left(1+j^2-\frac{\partial^2}{\partial{\beta}^2}\right)^{-1/2} + 
i \frac{1}{i} \frac{\partial}{\partial{\beta}}\left(1+j^2-\frac{\partial^2}{\partial{\beta}^2}\right)^{-1/2}\right)_j .
\label{a5a6}
\end{equation}

Mas este operador é igual a $\gamma_{A_5+iA_6}'(1)$ em (\ref{a5a6'}) e já calculamos o seu índice. Portanto, $ \mathsf{ind}(\gamma_{A_5+iA_6}(1,\pm 1)) = 0 $ e  
$$
{\delta}_1([[A_5 + i A_6]_{\mathcal{E}_{\mathcal{A}}}]_1) = (0,0).
$$

Considere o gerador $[[I+(A_5+iA_6-1)Q(l(M_t),\texttt{a}(M_x),\texttt{b}(M_x))]_{\mathcal{E}}]_1$ de $K_1(\mathcal{A}/\mathcal{E}_{\mathcal{A}}).$ Então, 
\begin{eqnarray*}
{\gamma}_{I+(A_5+iA_6-1)Q(l(M_t),\texttt{a}(M_x),\texttt{b}(M_x))}(1, -1) & = & 
I + ({\gamma}_{A_5+iA_6}(1,-1)-1)
\left(\begin{array}{cc}
I & 0 \\
0  & 0 \\
\end{array}\right) \\
& = &
\left(\begin{array}{cc}
{\gamma}_{A_5+iA_6}(1,-1) & 0 \\
0  & I \\
\end{array}\right) 
\end{eqnarray*}
pois ${\gamma}_{Q(l(M_t),\texttt{a}(M_x),\texttt{b}(M_x))}(1,-1)= Q(\gamma_{l(M_t)}(1,-1),\texttt{a}(M_x),\texttt{b}(M_x)),$ e já sabemos que $\gamma_{l(M_t)}(1,-1)= l(-\infty, Y_{-1}) = I.$ 
Pelos cálculos anteriores podemos concluir que 
$$ \mathsf{ind}({\gamma}_{I+(A_5+iA_6-1)Q(l(M_t),\texttt{a}(M_x),\texttt{b}(M_x))}(1, -1)) = 0 .$$ 

Ainda falta calcular o índice de $\gamma$ neste mesmo operador no ponto $(1,+1)$:
\begin{eqnarray*}
{\gamma}_{I+(A_5+iA_6-1)Q(l(M_t),\texttt{a}(M_x),\texttt{b}(M_x))}(1, +1) &=& 
I + ({\gamma}_{A_5+iA_6}(1,+1) - 1){\gamma}_{Q(l(M_t),\texttt{a}(M_x),\texttt{b}(M_x))}(1, +1)
\end{eqnarray*}
O operador ${\gamma}_{A_5+iA_6}(1,+1)\in \mathcal{L}_{\mathbb{Z}\times\mathbb{S}^1}$ é dado em (\ref{a5a6}). O outro operador é dado por
$$
{\gamma}_{Q(l(M_t),\texttt{a}(M_x),\texttt{b}(M_x))}(1, +1) = Q(Y_{-1},\texttt{a}(M_x), \texttt{b}(M_x))~ \in ~M_2(\mathcal{L}_{\mathbb{Z}\times\mathbb{S}^1}),
$$
pois ${\gamma}_{l(M_t)}(1,+1) = l(+\infty, Y_{-1})  = Y_{-1}$. 

Note que em  (\ref{gamma'Q}), $Q(Y_{-1},\texttt{a}(M_x), \texttt{b}(M_x)) = {\gamma}_{Q(e^{iM_t},\texttt{a}(M_x),\texttt{b}(M_x))}'(1).$ Logo, temos a seguinte igualdade: 
$$ 
I + ({\gamma}_{A_5+iA_6}(1,+1) - 1){\gamma}_{Q(l(M_t),\texttt{a}(M_x),\texttt{b}(M_x))}(1, +1) = 
I + ({\gamma}_{A_5+iA_6}'(1) - 1){\gamma}_{Q(e^{iM_t},\texttt{a}(M_x),\texttt{b}(M_x))}'(1).
$$
Como o índice de $ I + ({\gamma}_{A_5+iA_6}'(1) - 1){\gamma}_{Q(e^{iM_t},\texttt{a}(M_x),\texttt{b}(M_x))}'(1)$ é igual a 1, pois já o calculamos no teorema \ref{sobre},  segue que o índice de ${\gamma}_{I+(A_5+iA_6-1)Q(l(M_t),\texttt{a}(M_x),\texttt{b}(M_x))}(1, +1)$ é 1 e portanto, 
$$
{\delta}_1([[ I+(A_5+iA_6-1)Q(l(M_t),\texttt{a}(M_x),\texttt{b}(M_x)) ]_{\mathcal{E}}]_1) = (0,[E]_0). 
$$
De forma análoga, 
$$
{\delta}_1([[ I+(A_5+iA_6-1)Q(\tilde{l}(M_t),\texttt{a}(M_x),\texttt{b}(M_x)) ]_{\mathcal{E}}]_1) = ([E]_0,0). 
$$
Assim concluímos que ${\delta}_1$ é sobrejetora. 
\end{proof}

Conhecendo a aplicação do índice da seqüência (\ref{sss6}),temos as seguintes possibilidades para os K-grupos de $\mathcal{A}/{\cal{K}}_{\Omega}$:

\begin{enumerate} \label{poss3}

\item Se $\delta_0$ é sobrejetora, temos as seqüências exatas a partir de (\ref{sss6})
$$ 0 \rightarrow K_0(\mathcal{A}/{\cal{K}}_{\Omega}) \rightarrow K_0(\mathcal{A}/{\mathcal{E}}_{\mathcal{A}})\cong \mathbb{Z}^6 \stackrel{\delta_0}{\rightarrow} K_1({\mathcal{E}}_{\mathcal{A}}/{\cal{K}}_{\Omega})\cong \mathbb{Z}^2 \rightarrow 0 ,$$
$$ 0 \rightarrow K_1(\mathcal{A}/{\cal{K}}_{\Omega}) \rightarrow K_1(\mathcal{A}/{\mathcal{E}}_{\mathcal{A}})\cong \mathbb{Z}^6 \stackrel{\delta_1}{\rightarrow} K_0({\mathcal{E}}_{\mathcal{A}}/{\cal{K}}_{\Omega})\cong \mathbb{Z}^2 \rightarrow 0 ,$$
o que implica que  $K_0(\mathcal{A}/{\cal{K}}_{\Omega})$ e $K_1(\mathcal{A}/{\cal{K}}_{\Omega})$ são isomorfos a $\mathbb{Z}^4$.

\item Se $Im \delta_0 \cong \mathbb{Z}(1,1) \oplus \mathbb{Z}(\nu,0),$ para $\nu \neq \pm 1$ não nulo, temos 
$$ 0 \rightarrow K_0(\mathcal{A}/{\cal{K}}_{\Omega}) \rightarrow K_0(\mathcal{A}/{\mathcal{E}}_{\mathcal{A}})\cong \mathbb{Z}^6 \stackrel{\delta_0}{\rightarrow} Im \delta_0 \rightarrow 0 ,$$
$$ 0 \rightarrow \frac{K_1({\mathcal{E}}_{\mathcal{A}}/{\cal{K}}_{\Omega})}{Im \delta_0}\cong \mathbb{Z}_{\nu} \rightarrow K_1(\mathcal{A}/{\cal{K}}_{\Omega}) \rightarrow Ker \delta_1 \cong \mathbb{Z}^4 \rightarrow 0 ,$$
então $K_0(\mathcal{A}/{\cal{K}}_{\Omega})\cong \mathbb{Z}^4$ e $K_1(\mathcal{A}/{\cal{K}}_{\Omega})\cong \mathbb{Z}^4 \oplus \mathbb{Z}_{\nu}$.

\item Se $Im \delta_0 \cong \mathbb{Z}(1,1)$, temos as seqüências
$$ 0 \rightarrow K_0(\mathcal{A}/{\cal{K}}_{\Omega}) \rightarrow K_0(\mathcal{A}/{\mathcal{E}}_{\mathcal{A}})\cong \mathbb{Z}^6 \stackrel{\delta_0}{\rightarrow} Im \delta_0 \cong \mathbb{Z} \rightarrow 0 ,$$
$$ 0 \rightarrow \frac{K_1({\mathcal{E}}_{\mathcal{A}}/{\cal{K}}_{\Omega})}{Im \delta_0}\cong \mathbb{Z} \rightarrow K_1(\mathcal{A}/{\cal{K}}_{\Omega}) \rightarrow Ker \delta_1 \cong \mathbb{Z}^4 \rightarrow 0 ,$$
então $K_0(\mathcal{A}/{\cal{K}}_{\Omega})\cong \mathbb{Z}^5 $ e $K_1(\mathcal{A}/{\cal{K}}_{\Omega})\cong \mathbb{Z}^5 $.

\end{enumerate}


Passamos agora, para outra seqüência exata na tentativa de obter mais informações sobre a K-teoria de $\mathcal{A}.$ Para a seqüência abaixo
$$ 0 \ \longrightarrow \mathcal{E}_{\mathcal{A}} \  \stackrel{i}{\longrightarrow} \ \mathcal{A} \  \stackrel{\pi}{\longrightarrow} \ \frac{\mathcal{A}}{\mathcal{E}_{\mathcal{A}}} \ \longrightarrow \  0 , $$ 
considere a seqüência exata de seis termos em K-teoria associada:
\begin{equation}
\begin{array}{ccccc}
 K_0(\mathcal{E}_{\mathcal{A}})  & \! \stackrel{i_*}{\longrightarrow} & \! 
 K_0(\mathcal{A})  & \! \stackrel{{\pi}_*}{\longrightarrow} & \!
 K_0(\mathcal{A}/{\mathcal{E}}_{\mathcal{A}})   \\ \\    
 {\delta}_1 \ \uparrow & \! ~  & \! ~  & \! ~ & \! \downarrow \ {\delta}_0  \\ \\
 K_1(\mathcal{A}/{\mathcal{E}}_{\mathcal{A}}) & \! \stackrel{{\pi}_*}{\longleftarrow} & \! 
 K_1(\mathcal{A})  & \! \stackrel{{i}_*}{\longleftarrow} & \!
 K_1(\mathcal{E}_{\mathcal{A}})
\end{array}
\label{seq}
\end{equation}

Já determinamos $ K_0(\mathcal{A}/{\mathcal{E}}_{\mathcal{A}})$ e $K_1(\mathcal{A}/{\mathcal{E}}_{\mathcal{A}})$ e temos argumentos suficientes para calcularmos $K_0(\mathcal{E}_{\mathcal{A}})$ e $K_1(\mathcal{E}_{\mathcal{A}})$. 

\begin{prop} $K_0(\mathcal{E}_{\mathcal{A}})\cong \mathbb{Z}^2$ e $K_1(\mathcal{E}_{\mathcal{A}}) \cong \mathbb{Z}$. 
\end{prop}

\begin{proof} A proposição \ref{e} nos dá o seguinte isomorfismo
$$ \mathcal{E}_{\mathcal{A}} \cong \mathcal{S} \otimes \mathcal{K}_{\mathbb{Z}\times\mathbb{S}^1} .$$
Ao invés de calcularmos  $K_0$ e $K_1$ de $\mathcal{E}_{\mathcal{A}}$ , vamos calcular os K-grupos de $ \mathcal{S}$, já que $\mathcal{S} \otimes \mathcal{K}_{\mathbb{Z}\times\mathbb{S}^1} $ é a estabilização da C*-álgebra $ \mathcal{S}$ (\cite{rordam}, 6.4) e portanto
$$ K_0(\mathcal{S} \otimes \mathcal{K}_{\mathbb{Z}\times\mathbb{S}^1}) \cong K_0(\mathcal{S}) ~~ \mbox{e} ~~ 
K_1(\mathcal{S} \otimes \mathcal{K}_{\mathbb{Z}\times\mathbb{S}^1}) \cong K_1(\mathcal{S}). $$

No capítulo 1, vimos que $\mathcal{K}_{S^1}$ está contido em $\mathcal{S}$ e que $\mathcal{S}/\mathcal{K}_{S^1} \cong C(S^1 \times (-1,+1)).$ A partir da seqüência exata
$$  0 \ \longrightarrow \mathcal{K}_{S^1} \  \stackrel{i}{\longrightarrow} \ {\cal{S}} \  \stackrel{\pi}{\longrightarrow} \ \frac{\cal{S}}{\mathcal{K}_{S^1}} \ \longrightarrow \ 0 ,$$
temos a seqüência exata em K-teoria
$$
\begin{array}{ccccc}
 \mathbb{Z} \cong K_0(\mathcal{K}_{S^1})  & \! {\longrightarrow} & \! 
 K_0(\mathcal{S})  & \! {\longrightarrow} & \!
 K_0(\mathcal{S}/\mathcal{K}_{S^1})  \\ \\    
 {\delta}_1 \ \uparrow & \! ~  & \! ~  & \! ~ & \! \downarrow \ {\delta}_0  \\ \\
 K_1(\mathcal{S}/\mathcal{K}_{S^1}) & \! {\longleftarrow} & \! 
 K_1(\mathcal{S})  & \! {\longleftarrow} & \!
 K_1(\mathcal{K}_{S^1}) = 0
\end{array}
$$
Sabemos que $K_0(C(S^1)\oplus C(S^1)) = \mathbb{Z}[(0,1)]_0 \oplus \mathbb{Z}[(1,0)]_0 ,$ então
$$ K_0(\mathcal{S}/\mathcal{K}_{S^1}) =  \mathbb{Z}[[b(D_{\theta})]_{\mathcal{K}_{S^1}}]_0 \oplus \mathbb{Z}[[c(D_{\theta})]_{\mathcal{K}_{S^1}}]_0 $$
onde $b(D_{\theta})= F_d^{-1}b(M_j)F_d, ~c(D_{\theta})= F_d^{-1}c(M_j)F_d$ e escolhemos $ b(M_j)$ de tal forma que, para um inteiro $n_0$ fixo, 
$$
b(j)= 
\left\{\begin{array}{rl}
1,& \mbox{ se }  j \geq n_0 \\
0,& \mbox{ se }  j < n_0
\end{array}\right.
$$
e $c(M_j)$ é tal que $c(j) = 1 - b(j).$ De fato, estas são boas escolhas de $b(D_{\theta})$ e $c(D_{\theta})$ pois 
$$(\rho_{b(D_{\theta})}(x,-1), \rho_{b(D_{\theta})}(x,+1)= (b(-\infty),b(+\infty))= (0,1), $$
$$(\rho_{c(D_{\theta})}(x,-1), \rho_{c(D_{\theta})}(x,+1)) = (c(-\infty),c(+\infty))=(1,0).$$

Para $ K_1(C(S^1)\oplus C(S^1)) = \mathbb{Z}[(\texttt{x},1)]_1 \oplus \mathbb{Z}[(1,\texttt{x})]_1$, onde $\texttt{x}:x\in S^1 \mapsto x$, temos
$$ K_1(\mathcal{S}/\mathcal{K}_{S^1}) = \mathbb{Z}[[b(D_{\theta})\texttt{x}(M_x) + c(D_{\theta})]_{\mathcal{K}_{S^1}}]_1 \oplus \mathbb{Z}[[b(D_{\theta}) + \texttt{x}(M_x)c(D_{\theta})]_{\mathcal{K}_{S^1}}]_1. $$ 

O nosso trabalho se resume a calcular $\delta_1 : K_1(\mathcal{S}/\mathcal{K}_{S^1})\rightarrow K_0(\mathcal{K}_{S^1}),$ que é dada pelo índice de Fredholm. Como o operador $b(D_{\theta})\texttt{x}(M_x) + c(D_{\theta})$ não é unitário em $\mathcal{S}$, vamos calcular seu índice.  
\begin{eqnarray}
b(D_{\theta})\texttt{x}(M_x) + c(D_{\theta}) & = & F_d^{-1} b(M_j) F_d \texttt{x}(M_x) + F_d^{-1} c(M_j) F_d  \nonumber \\
& = & F_d^{-1} (b(M_j) F_d \texttt{x}(M_x) + c(M_j) F_d)  \nonumber \\
& = & F_d^{-1} (b(M_j)Y_{-1} F_d  + c(M_j) F_d)  \nonumber \\
& = & F_d^{-1} (b(M_j)Y_{-1} + c(M_j)) F_d  \nonumber 
\end{eqnarray}
Com a igualdade acima, temos que a dimensão do núcleo de $b(D_{\theta})\texttt{x}(M_x) + c(D_{\theta})$ é igual a dimensão do núcleo de $b(M_j)Y_{-1} + c(M_j).$
$$ (b(M_j)Y_{-1} + c(M_j))(u_j)_j = 0 \Longleftrightarrow u_j = 0 ~\forall j \in \mathbb{Z} ,$$
então a dimensão do núcleo de  $b(M_j)Y_{-1} + c(M_j)$ é igual a 0. 
Vamos agora calcular a dimensão do núcleo de seu adjunto: 
$$ \overline{\texttt{x}}(M_x)b(D_{\theta}) + c(D_{\theta}) = F_d^{-1} (Y_{1}b(M_j) + c(M_j)) F_d .$$
$$ (Y_{1}b(M_j) + c(M_j))(u_j)_j = 0 \Longleftrightarrow u_j = 0 ~\forall j \neq n_0 -1, n_0 ~ \mbox{e} ~ u_{n_0-1}+u_{n_0} = 0.$$
Isto acontece pois na posição $n_0 -1$ da seqüência vamos ter o elemento $u_{n_0-1}+u_{n_0}$ e $u_{n_0-1}$ e $u_{n_0}$ só aparecem nesta posição, e nas outras posições todos os outros elementos aparecem sozinhos. Assim, dim $ker (Y_{1}b(M_j) + c(M_j)) = 1$. Logo, 
$$ \mathsf{ind}(b(D_{\theta})\texttt{x}(M_x) + c(D_{\theta})) =  0 - 1 = -1, $$
e portanto $\delta_1$ é sobrejetora. 

Segue que,
$$ K_0(\mathcal{E}_{\mathcal{A}}) \cong K_0(\mathcal{S}) \cong \mathbb{Z}^2 ~~\mbox{e} ~~K_1(\mathcal{E}_{\mathcal{A}}) \cong K_1(\mathcal{S}) \cong \mathbb{Z}$$
como queríamos. 
\end{proof}

\begin{obs} \label{obs2} Com este resultado, podemos observar na seqüência 
$$
\begin{array}{ccccc}
 \mathbb{Z} \cong K_0(\mathcal{K}_{\Omega})  & \! {\longrightarrow} & \! 
 K_0(\mathcal{E}_{\mathcal{A}}) \cong \mathbb{Z}^2  & \! {\longrightarrow} & \!
 K_0({\mathcal{E}}_{\mathcal{A}}/\mathcal{K}_{\Omega}) \cong \mathbb{Z}^2  \\ \\    
 {\delta}_1 \ \uparrow & \! ~  & \! ~  & \! ~ & \! \downarrow \ {\delta}_0  \\ \\
 \mathbb{Z}^2 \cong K_1({\mathcal{E}}_{\mathcal{A}}/\mathcal{K}_{\Omega}) & \! {\longleftarrow} & \! 
 K_1(\mathcal{E}_{\mathcal{A}})\cong \mathbb{Z}  & \! {\longleftarrow} & \!
 K_1(\mathcal{K}_{\Omega}) = 0
\end{array}
$$
que a aplicação do índice $\delta_1 : K_1(\mathcal{E}_{\mathcal{A}}/ \mathcal{K}_{\Omega}) \rightarrow  K_0(\mathcal{K}_{\Omega})$ é não nula (caso contrário, $\mathbb{Z}\cong \mathbb{Z}^2$). Se $Im \delta_1 \cong n\mathbb{Z},$ para algum $n$ inteiro não nulo. Então teríamos a seqüência 
$$ 0 \rightarrow \frac{K_0(\mathcal{K}_{\Omega})}{Im \delta_1} \cong \mathbb{Z}_n \rightarrow K_0(\mathcal{E}_{\mathcal{A}}) \cong \mathbb{Z}^2  \stackrel{\vartheta}{\rightarrow} K_0({\mathcal{E}}_{\mathcal{A}}/\mathcal{K}_{\Omega}) \cong \mathbb{Z}^2 \rightarrow 0, $$
onde $\mathbb{Z}_n \cong \mathbb{Z}/n\mathbb{Z}.$ Note que o núcleo da aplicação $\vartheta$ é isomorfo a $\mathbb{Z}_n$, mas $\mathbb{Z}_n$ não é um subgrupo de $\mathbb{Z}^2$. Logo $\delta_1$ é sobrejetora.  
Além disso, esta aplicação também é dada pelo índice de Fredholm. Portanto, sabemos que existe um operador de Fredholm $T$ em $M_n(\tilde{\mathcal{E}_{\mathcal{A}}})$ que tem índice igual a $1$. Vale ainda ressaltar que $K_0(\mathcal{E}_{\mathcal{A}}) \cong K_0({\mathcal{E}}_{\mathcal{A}}/\mathcal{K}_{\Omega}).$  
\end{obs}

\begin{prop} A aplicação $\delta_1$ em (\ref{seq}) é sobrejetora. 
\end{prop}

\begin{proof} Pelo teorema \ref{sobrejetora}, $\delta_1 : K_1({\mathcal{A}}/\mathcal{E}_{\mathcal{A}}) \rightarrow K_0({\mathcal{E}}_{\mathcal{A}}/\mathcal{K}_{\Omega})$ é sobrejetora e como $K_0(\mathcal{E}_{\mathcal{A}}) \cong K_0({\mathcal{E}}_{\mathcal{A}}/\mathcal{K}_{\Omega}), $ temos o que queríamos.   
\end{proof}

Assim, as possibilidades para $K_0(\mathcal{A})$ e $K_1(\mathcal{A})$ analisando a seqüência (\ref{seq}) são:
\begin{enumerate}
\item Se $\delta_0 \equiv 0$ então $K_0(\mathcal{A})\cong \mathbb{Z}^6$ e $K_1(\mathcal{A})\cong \mathbb{Z}^5.$
\item Se $\delta_0 $ em algum elemento de $K_0(\mathcal{A}/\mathcal{E}_{\mathcal{A}})$ for relacionado a um inteiro não nulo $\zeta$ pelo isomorfismo $K_1({\mathcal{E}}_{\mathcal{A}}) \cong \mathbb{Z}$, então $K_0(\mathcal{A})\cong \mathbb{Z}^5$ e $K_1(\mathcal{A})\cong \mathbb{Z}^4 \oplus \mathbb{Z}_{\zeta} .$
\end{enumerate}

Para finalizar, enunciamos o seguinte teorema que nos dá a K-teoria das álgebras $\mathcal{A}/{\cal{K}}_{\Omega}$ e $\mathcal{A}$. 

\begin{teo} Dada a seqüência exata curta
$$ 0 \ \longrightarrow {\cal{K}}_{\Omega} \  \stackrel{i}{\longrightarrow} \ {\cal{A}} \  \stackrel{\pi}{\longrightarrow} \ \frac{\cal{A}}{{\cal{K}}_{\Omega}} \ \longrightarrow \ 0 ,$$ 
onde $i$ é a inclusão e $\pi$ a projeção canônica, podemos afirmar: 
\begin{itemize}
\item[(i)] a aplicação do índice $\delta_1$ da seqüência exata de seis termos em K-teoria
$$
\begin{array}{ccccc}
 K_0({\cal{K}}_{\Omega}) & \! \stackrel{{i}_*}{\longrightarrow} & \! 
 K_0(\mathcal{A})  & \! \stackrel{{\pi}_*}{\longrightarrow} & \!
 K_0(\mathcal{A}/{\cal{K}}_{\Omega}) \vs  \\     
 {\delta}_1 \ \uparrow & \! ~  & \! ~  & \! ~ & \! \downarrow \ {\delta}_0 \vs \\ 
 K_1(\mathcal{A}/{\cal{K}}_{\Omega}) & \! \stackrel{{\pi}_*}{\longleftarrow} & \! 
 K_1(\mathcal{A})  & \! \stackrel{{i}_*}{\longleftarrow} & \! K_1({\cal{K}}_{\Omega})
\end{array}
$$
é sobrejetora;
\item[(ii)] $K_0(\mathcal{A})  \cong K_0(\mathcal{A}/{\cal{K}}_{\Omega}) \cong \mathbb{Z}^5;$
\item[(iii)] $K_1(\mathcal{A})\cong \mathbb{Z}^4$ e $K_1(\mathcal{A}/{\cal{K}}_{\Omega})\cong \mathbb{Z}^5$.
\end{itemize}
\end{teo}

\begin{proof}

Sabemos que existe um operador $T$ em $M_n(\tilde{\mathcal{E}_{\mathcal{A}}})$ cujo índice de Fredholm é $1$ pela observação \ref{obs2}. Portanto, existe um operador $T$ em $M_n(\tilde{\mathcal{A}})$ com índice de Fredholm igual a $ 1$. Logo, $\delta_1 ([[T]_{{\cal{K}}_{\Omega}}]_1) = \texttt{ind}(T)[P]_0$ onde $P \in {\cal{K}}_{\Omega}$ é uma projeção de posto 1. Portanto $\delta_1$ é sobrejetora. 

Como o núcleo de $i_{*}$ é $K_0({\cal{K}}_{\Omega})$ e $\delta_0$ é nula, pois $K_1({\cal{K}}_{\Omega})=0,$ então $K_0(\mathcal{A})  \cong K_0(\mathcal{A}/{\cal{K}}_{\Omega}).$ Para que este isomorfismo aconteça, as únicas possibilidades para $K_0(\mathcal{A})$ e $ K_0(\mathcal{A}/{\cal{K}}_{\Omega})$ (página \pageref{poss3}) são que ambos sejam isomorfos a $\mathbb{Z}^5.$

Como $K_0(\mathcal{A}/{\cal{K}}_{\Omega})\cong \mathbb{Z}^5,$ então podemos concluir  que $K_1(\mathcal{A}/{\cal{K}}_{\Omega})\cong \mathbb{Z}^5.$ 
O fato de $K_0(\mathcal{A})$ ser isomorfo a $\mathbb{Z}^5,$ implica que $K_1(\mathcal{A})\cong \mathbb{Z}^4\oplus \mathbb{Z}_{\zeta}.$ 
Como $\pi_{*}(K_1(\mathcal{A}))$ é o núcleo da aplicação $\delta_1,$ e portanto, um  subgrupo de $K_1(\mathcal{A}/{\cal{K}}_{\Omega})$, então $\zeta=\pm 1$ e concluímos que  $K_0(\mathcal{A}) \cong \mathbb{Z}^4.$

\end{proof}

\chapter*{Apêndice}
\addcontentsline{toc}{chapter}{Apêndice}
Apresentamos aqui, os comandos usados no software \textit{Maple} para calcular a integral dada em (\ref{indice}): 

\begin{eqnarray}
\mathsf{ind} T & = & \frac{(-1)^{3}}{(2\pi i)^2}\frac{(2-1)!}{(2\cdot 2 - 1)!}\int_{S^1 \times {S}^1 \times {S}^1 }Tr({\sigma}^{-1}d\sigma)^{3} \nonumber \\
& = & \frac{1}{24{\pi}^2}\int_{S^1 \times {S}^1 \times {S}^1 }Tr({\sigma}^{-1}d\sigma)^{3} \nonumber
\end{eqnarray}
onde o símbolo principal do operador $T$ coincide com o símbolo que definimos no nosso trabalho:
\begin{equation}
\sigma_{T}((e^{i\alpha}, e^{i\beta}, e^{i\lambda})) = I + (-\cos\lambda + i \mbox{sen}\lambda - 1)Q(e^{i\alpha},\cos\beta,\mbox{sen}\beta).
\label{t}
\end{equation}
Obtemos então que $ (\sigma^{-1}d\sigma)^3 = [A_1 + A_2 +A_3 ] d\alpha d\beta d\lambda $, para 
$$ A_1 = \sigma^{-1} \sigma_{\alpha} \sigma^{-1}(\sigma_{\beta} \sigma^{-1} \sigma_{\lambda} - \sigma_{\lambda} \sigma^{-1}\sigma_{\beta} ), $$
$$ A_2 = \sigma^{-1} \sigma_{\beta} \sigma^{-1}(\sigma_{\lambda} \sigma^{-1}\sigma_{\alpha} - \sigma_{\alpha} \sigma^{-1}\sigma_{\lambda} ), $$
$$ A_3 = \sigma^{-1} \sigma_{\lambda} \sigma^{-1}(\sigma_{\alpha}\sigma^{-1}\sigma_{\beta} - \sigma_{\beta} \sigma^{-1} \sigma_{\alpha}). $$
onde $\sigma_{\alpha}= \frac{\partial}{\partial\alpha}\sigma , ~ \sigma_{\beta}=\frac{\partial}{\partial\beta}\sigma , ~ \sigma_{\lambda}= \frac{\partial}{\partial\lambda}\sigma.$  

O objetivo é determinar $A_1, A_2$ e $A_3$, para depois calcular a integral. 

Seguem os comandos:

\begin{itemize} 
 \item $ m:= 1 - (1-\cos(\alpha))^{*}(1-\cos(\beta))/4;$ (entrada $a_{11}$ da matriz $Q$)
 \item $ p:= ((exp(I^{*}2^{*}\alpha)-1)/4)^{*}((1-\cos(\beta))/2)^{**}(1/2)- \sin(\beta)^{*}(exp(I^{*}x) -1)^{**}2/8;$ (entrada $a_{12}$ da matriz $Q$)
 \item $ q:= ((exp(-I^{*}2^{*}\alpha)-1)/4)^{*}((1-\cos(\beta))/2)^{**}(1/2)- \sin(\beta)^{*}(exp(-I^{*}x) -1)^{**}2/8;$ (entrada $a_{21}$ da matriz $Q$)
 \item $ n:= (1-\cos(\alpha))^{*}(1-\cos(\beta))/4;$ (entrada $a_{22}$ da matriz $Q$)
 \item $ A:= <<m,q> | <p,n>>;$  (a matriz $Q$)
 \item $ S:= 1+ f(\lambda)^{*}A;$ (\ref{t})
 \item $ Z:= S^{**}(-1); $ (inversa de $S$)
 \item $ md\alpha := diff(m, \alpha);$ (derivadas em relação a $\alpha$ das entradas da matriz ) 
 \item $ pd\alpha := diff(p, \alpha);$
 \item $ qd\alpha := diff(q, \alpha);$
 \item $ nd\alpha := diff(n, \alpha);$
 \item $ Sdx:= f(\lambda)^{*}<<md\alpha,qd\alpha> | <pd\alpha,nd\alpha>>;$ ( matriz  da derivada parcial em relação a $\alpha$) 
 \item $ md\beta := diff(m, \beta);$ (derivadas em relação a $\beta$ das entradas da matriz ) 
 \item $ pd\beta := diff(p, \beta);$
 \item $ qd\beta := diff(q, \beta);$
 \item $ nd\beta := diff(n, \beta);$
 \item $ Sd\beta := f(\lambda)^{*}<<md\beta,qd\beta> | <pd\beta,nd\beta>>;$ ( matriz  da derivada parcial em relação a $\beta$) 
 \item $ Sd\lambda := g(\lambda)^{*}A;$  (matriz onde $g(\lambda)$ é a derivada em relação a $\lambda$)
 \item $ C1:= Sd\beta . Z . Sd\lambda - Sd\lambda . Z . Sd\beta ; $ ($ = \sigma_{\beta} \sigma^{-1} \sigma_{\lambda} - \sigma_{\lambda} \sigma^{-1}\sigma_{\beta} $)
 \item $ C2:= Sd\lambda . Z . Sd\alpha - Sd\alpha . Z . Sd\lambda ; $ ($ = \sigma_{\lambda} \sigma^{-1}\sigma_{\alpha} - \sigma_{\alpha} \sigma^{-1}\sigma_{\lambda} $)
 \item $ C3:= Sd\alpha . Z . Sd\beta - Sd\beta . Z . Sd\alpha ; $ ($ = \sigma_{\alpha}\sigma^{-1}\sigma_{\beta} - \sigma_{\beta} \sigma^{-1} \sigma_{\alpha} $)
 \item $ A1 := Z.Sd\alpha . Z . C1 ;$
 \item $ A2 := Z.Sd\beta . Z . C2 ;$
 \item $ A3 := Z.Sd\lambda . Z . C3 ;$
 \item $ T1 := Trace(A1);$
 \item $ T2 := Trace(A2);$
 \item $ T3 := Trace(A3);$
 \item $ U := simplify(T1);$ (simplifica a expressão - opcional!)
 \item $ V := simplify(T2);$
 \item $ X := simplify(T3);$
 \item $ u1 := Int(U, \alpha=-Pi .. Pi);$ ( integração na variável $\alpha$  de $-\pi$ a $\pi$ do traço de $A1$ )
 \item $ h1 := value(u1);$  (mostra o valor da integral)
 \item $ u2 := Int(h1, \beta =-Pi .. Pi);$ ( integração na variável $\beta$  de $-\pi$ a $\pi$ de $h1$ )
 \item $ h2 := value(u2);$  (mostra o valor da integral que está em função de $f(\lambda)$ e $g(\lambda)$)
 $$   h2  := \frac{-2~I ~ Pi ~f(\lambda)^2 ~g(\lambda)}{f(\lambda)^2~+~ 2~f(\lambda)~ +~ 1} 
 $$
 \item $ r:= (-2^{*}I^{*}Pi^{*}(-\cos(\lambda)+ I^{*}\sin{\lambda}-1)^{**}2^{*}(\sin(\lambda)+ I^{*}\cos(\lambda))/ $ \\
 $((-\cos(\lambda)+ I^{*}\sin{\lambda}-1)^2~+~ 2~(-\cos(\lambda)+ I^{*}\sin{\lambda}-1)~ +~ 1) $ (vamos substituir $r$ em $h2$)
 \item $ u3 := Int(r, \lambda =-Pi .. Pi);$
 \item $ h3 := value(u3);$
 $$    h3  := 8~Pi^{2} $$
 \item $ v1 := Int(V, \alpha=-Pi .. Pi);$ ( integração na variável $\alpha$  de $-\pi$ a $\pi$ do traço de $A2$ )
 \item $ l1 := value(v1);$
 \item $ v2 := Int(l1, \beta =-Pi .. Pi);$
 \item $ l2 := value(v2);$
 $$    l2  := \frac{-2~I ~ Pi ~f(\lambda)^2 ~g(\lambda)}{f(\lambda)^2~+~ 2~f(\lambda)~ +~ 1} $$
  que é igual a $h2$, portanto a integral na variável $\lambda$ será  igual a $h3$.
 $$   l3  := 8~Pi^{2}$$ 
 \item $ x1 := Int(X, \alpha=-Pi .. Pi);$ ( integração na variável $\alpha$  de $-\pi$ a $\pi$ do traço de $A3$ )
 \item $ t1 := value(x1);$
 \item $ x2 := Int(t1, \beta =-Pi .. Pi);$
 \item $ t2 := value(x2);$
 $$   t2  := \frac{-2~I ~ Pi ~f(\lambda)^2 ~g(\lambda)}{f(\lambda)^2~+~ 2~f(\lambda)~ +~ 1} $$
 que é novamente igual a $l2$ e a $h2$, portanto 
 $$   t3  := 8~Pi^{2} $$
\end{itemize}

Logo 
$$\int_{-\pi}^{\pi}\int_{-\pi}^{\pi}\int_{-\pi}^{\pi} (T1+T2+T3)d\alpha d\beta d\lambda = 24 {\pi}^2 .$$



\bibliographystyle{plain}

\bibliography{referencia}




\end{document}